\documentclass[11pt]{amsart}
\usepackage[dvipsnames,usenames]{xcolor}
\usepackage{hyperref}
\usepackage{ upgreek }
\usepackage{graphicx, import}
\usepackage{epsfig}
\usepackage[latin1]{inputenc}
\usepackage{amsmath}
\usepackage{amsfonts}
\usepackage{amssymb}
\usepackage{amsthm}
\usepackage{amscd}
\usepackage{verbatim}
\usepackage{subcaption}
\usepackage{caption}
\usepackage{pinlabel}
\usepackage{stmaryrd}
\usepackage{enumerate, enumitem}
\usepackage{todonotes}
\usepackage{bm}
\usepackage{thmtools}
\usepackage{thm-restate}
\usepackage{lipsum}
\usepackage{setspace}
\usepackage{mathtools}
\usepackage[all]{xypic}
\usepackage[abs]{overpic}
\usepackage[all,cmtip]{xy}

\allowdisplaybreaks

\usepackage{tikz-cd}
\usepackage{mathdots}
\usepackage{float}
\usepackage{tikz}
\usetikzlibrary{arrows}
\usetikzlibrary{decorations.pathreplacing}
\usepackage{verbatim}
\usetikzlibrary{cd}
\tikzset{taar/.style={double, double equal sign distance, -implies}}
\tikzset{amar/.style={->, dotted}}
\tikzset{dmar/.style={->, dashed}}
\tikzset{aar/.style={->, very thick}}

\usepackage[section]{placeins}

\newtheorem{theorem}{Theorem}[section]

\newtheorem{lemma}[theorem]{Lemma}
\newtheorem{proposition}[theorem]{Proposition}

\newtheorem{corollary}[theorem]{Corollary}

\theoremstyle{definition}
\newtheorem{definition}[theorem]{Definition}

\theoremstyle{remark}
\newtheorem{remark}[theorem]{Remark}
\newtheorem{example}[theorem]{Example}

\theoremstyle{Question}
\newtheorem{Question}[theorem]{Question}

\def\im{\operatorname{im}}

\def\CFK{\mathit{CFK}}

\newcommand{\Id}{\operatorname{Id}}

\usepackage{svg}
\usepackage{xcolor}
\newcommand\crule[3][black]{\textcolor{#1}{\rule{#2}{#3}}}

\usepackage{etoolbox,refcount}
\usepackage{multicol}

\newcounter{countitems}
\newcounter{nextitemizecount}
\newcommand{\setupcountitems}{%
  \stepcounter{nextitemizecount}%
  \setcounter{countitems}{0}%
  \preto\item{\stepcounter{countitems}}%
}
\makeatletter
\newcommand{\computecountitems}{%
  \edef\@currentlabel{\number\c@countitems}%
  \label{countitems@\number\numexpr\value{nextitemizecount}-1\relax}%
}
\newcommand{\nextitemizecount}{%
  \getrefnumber{countitems@\number\c@nextitemizecount}%
}
\newcommand{\previtemizecount}{%
  \getrefnumber{countitems@\number\numexpr\value{nextitemizecount}-1\relax}%
}
\makeatother    
\newenvironment{AutoMultiColItemize}{%
\ifnumcomp{\nextitemizecount}{>}{3}{\begin{multicols}{2}}{}%
\setupcountitems\begin{itemize}}%
{\end{itemize}%
\unskip\computecountitems\ifnumcomp{\previtemizecount}{>}{3}{\end{multicols}}{}}

\usepackage{graphicx}
\usepackage[capitalise]{cleveref}

\author[J. Patwardhan]{Jay Patwardhan}
\email{jap600@scarletmail.rutgers.edu}
\address{Rutgers University New Brunswick, New Brunswick, NJ, USA}
\thanks{JP was partially supported by NSF CAREER Grant DMS-2019396.}

\author[Z. Xiao]{Zheheng Xiao}
\email{zx2377@columbia.edu}
\address{Columbia University, New York, NY, USA}
\thanks{ZX was partially supported by NSF CAREER Grant DMS-2019396.}

\numberwithin{equation}{section}

\title{Generalized Mazur patterns and immersed Heegaard Floer homology}

\begin{document}
\maketitle
\begin{abstract} Generalizing prior work of Levine, we give infinitely many examples of pattern knots $P$ such that $P(K)$ is not slice in any rational homology $4$-ball, for any companion knot $K$. To show this, we establish a closed formula for the concordance invariants $\tau$ and $\epsilon$ of a family of satellite knots obtained from generalized Mazur patterns. Our main computational tool is the immersed curve technique from bordered Heegaard Floer homology arising from the work of Chen-Hanselman.
\end{abstract}
\tableofcontents
\section{Introduction}
Two knots in $S^3$ are said to be \textit{smoothly concordant} if they co-bound a smoothly embedded annulus in $S^3\times I$. The set of concordance classes of knots form a group $\mathcal{C}$, with addition given by connect sum, and identity given by the concordance class of the unknot. In particular, knots in the concordance class of the unknot are called \textit{smoothly slice}. The classical study of knot concordance looks to classify which knots in $S^3$ are smoothly slice in $D^4$. 

One may extend the notion of sliceness to knots to more general $3$-manifolds. A knot $K$ in the boundary of a smooth $4$-manifold $M$ is said to be \textit{smoothly slice} if it bounds a smoothly embedded disk in $M$. There are also weaker notions of concordance that we consider: two knots are \textit{exotically concordant} if they co-bound a smoothly embedded annulus in a smooth $4$-manifold homeomorphic to $S^3 \times I$ but possibly with an exotic smooth structure, and (for a ring $R$) \textit{$R$-homology concordant} if they co-bound a smoothly embedded annulus in a smooth manifold with the same $R$-homology of $S^3 \times I$. Then, a knot $K \subset S^3$ is\textit{ exotically slice} or \textit{$R$-homology slice} if it is exotically or $R$-homology concordant to the unknot, respectively; this is the same as saying that $K$ bounds an embedded disk in a contractible $4$-manifold (which is homeomorphic to $D^4$ by Freedman \cite{Fre82}) or an $R$-homology $4$-ball, respectively. In particular, $\mathbb{Q}$-homology concordance is stronger than exotic concordance.

For any two knots $K\subset S^3$ and $P\subset S^1 \times D^2$, let $P(K)$ denote the satellite knot of $K$ with pattern $P$. If $K$ is concordant to $K'$, then $P(K)$ is concordant to $P(K')$, so we may regard $P$ as an operator $P: \mathcal{C} \rightarrow\mathcal{C}$, called the \textit{satellite operator}. Note that this map is generally not a homomorphism. The satellite operator is well-studied in literature; see for example \cite{Hed07, Hom14a, Lev16}. Moreover, we define the \textit{winding number} $w(P)$ to be the number of signed intersections of a meridional disk with $P$.

Problem 1.45 in Kirby's problem list \cite{Kir97}, attributed to Akbulut, asks whether there exists a winding number $\pm 1$ satellite operator $P$ for which $P(K)$ is never exotically slice, that is, $P(K)$ does not bound a contractible 4-manifold. Levine answered a stronger version of this in the affirmative:

\begin{theorem}\cite[Theorem 1.2]{Lev16}\label{Kirby problem}
    There exists a pattern knot $P \subset S^1 \times D^2$ with winding number $1$ such that for any knot $K \subset S^3$, $P(K)$ is not slice in any rational homology $4$-ball.
\end{theorem}

Levine's strategy was to find a pattern $Q$ which induces a non-surjective satellite operator on the rational homology concordance classes, so that there exists a knot $L$ which is not concordant to $Q(K)$ for all $K \subset S^3$. Then, $P = Q \# -L \subset S^1 \times D^2$ satisfies the conclusion of the Theorem. In particular, he chose $Q$ to be the Mazur pattern, shown in Figure \ref{fig: Mazur}.

\begin{figure}[!htb]
    \centering
\includegraphics[scale=0.9]{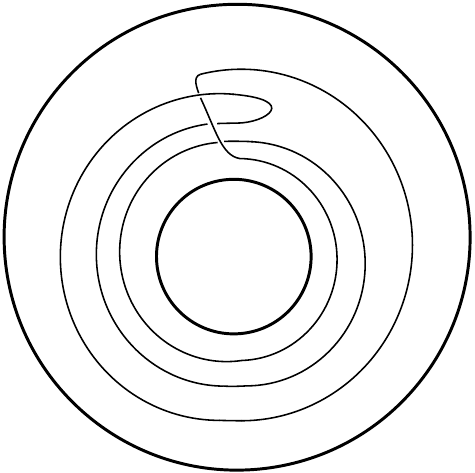}
    \caption{The Mazur pattern $Q_{2,1}$ embedded in the solid torus $V=S^1\times D^2$.}
    \label{fig: Mazur}
\end{figure}

The goal of this paper is to expand on this result to obtain an infinite family of pattern knots with winding number $\pm 1$ whose satellites are never slice in any rational homology $4$-ball. To do this, we consider a generalization of the Mazur pattern $Q$.

\begin{definition} \label{def gmp}
 Starting with a point in the solid torus, wind $m$ times around the torus, then turn around and wind $n$ times. Join the top endpoint of the arc to the bottom endpoint by crossing under the first $n$ times, and over the next $m$ times, resulting in a clasp where the pattern turns around. The resulting pattern knot is a \textit{generalized Mazur pattern} $Q_{m,n}$, and a picture of this is shown in Figure \ref{fig:gen mazur pic}.
\end{definition}

\begin{figure}[]
    \centering
\includegraphics[scale=0.7]{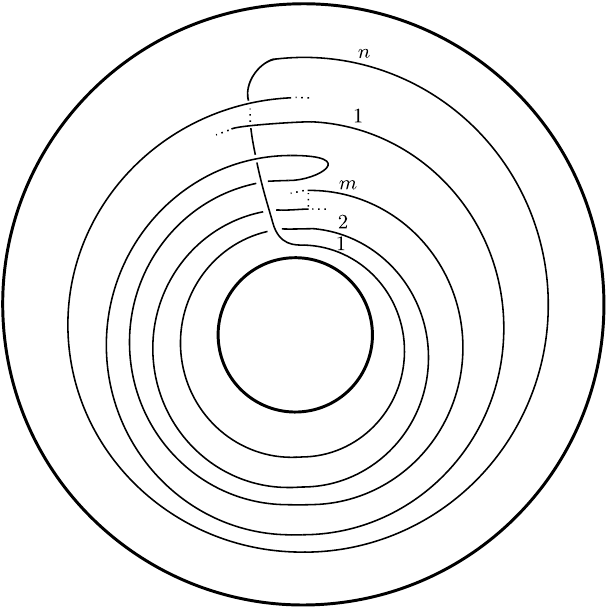}
    \caption{The generalized Mazur pattern $Q_{m,n}$}
    \label{fig:gen mazur pic}
\end{figure}

A pattern knot $P$ is a $(1,1)$-pattern if it admits a genus one doubly-pointed Heegaard diagram, and is an \textit{unknot pattern} if $P(U) \simeq U$ in $S^3$, where $U$ is the unknot \cite{Che19,CH23}. Generalized Mazur patterns $Q_{m,n}$ are part of the class of $(1,1)$-unknot patterns, and have winding number $\pm (m-n)$ depending on the orientation. We show that the pattern remains the same if we switch $m$ and $n$.

\begin{proposition} \label{qmn isotopic}
    Inside the solid torus $S^1 \times D^2$, $Q_{m,n}$ is isotopic to $Q_{n,m}$.
\end{proposition}

In the early 2000s, Ozsv{\'a}th and Szab{\'o} \cite{OS04a, OS04c} introduced  Heegaard Floer homology, a collection of invariants of three-manifolds and knots and links inside them. The knot version, also independently introduced by Rasmussen \cite{Ras03}, associates to every knot $K\subset S^3$ a $\mathbb{Z}\oplus \mathbb{Z}$-filtered, free $\mathbb{F}[U, U^{-1}]$-complex $CFK^\infty (K)$, called the knot Floer complex of $K$. Knot Floer homology has several nice properties; for example, it categorifies the Alexander polynomial \cite{OS04a}, detects the knot genus \cite{OS04b}, and detects fiberedness \cite{Ghi08, Ni07}. There are a variety of concordance invariants arising from $CFK^\infty(K)$; in this paper, we are interested in the integer-valued $\tau$-invariant \cite{OS03} and the $\{-1, 0, 1\}$-valued $\epsilon$-invariant \cite{Hom14b}, which are maps with domain the set of rational homology concordance classes.

In \cite{Lev16}, Levine computed $\tau(Q(K))$ and $\epsilon(Q(K))$ for satellites along the Mazur pattern.
\begin{theorem}\cite[Theorem 1.6]{Lev16}\label{Levine theorem} If $Q$ is the Mazur pattern, then for any knot $K \subset S^3$, we have
    $$
    \tau(Q(K))= \begin{cases}\tau(K) & \text { if } \tau(K) \leq 0 \text { and } \epsilon(K) \in\{0,1\} \\ \tau(K)+1 & \text { if } \tau(K)>0 \text { or } \epsilon(K)=-1.\end{cases}
    $$

and 

    $$
    \epsilon(Q(K)) = \begin{cases}0 & \text { if } \tau(K) = \epsilon(K) = 0 \\
    1 & \text { otherwise }
    \end{cases}
    $$
\end{theorem}

In this paper, we generalize this result to the generalized Mazur patterns $Q_{m,n}(K)$:
\begin{theorem}\label{main thm}
Let $Q_{m,n}$ be the generalized Mazur pattern embedded in the solid torus $V$. If $m\ne n$, then for any knot $K\subset S^3$, we have \begin{equation}\label{main eq}
    \tau(Q_{m,n}(K)) = \begin{cases}
    |m-n|\tau(K) & \text{if } \tau(K) \le 0 \text{ and } \epsilon(K) \in \{0,1\}, \\
    |m-n|\tau(K) + |m-n| & \text{if } \tau(K) < 0 \text{ and } \epsilon(K)=-1, \\
    |m-n|\tau(K) + \min(m,n) & \text{if } \tau(K) > 0 \text{ and }\epsilon(K) = 1,\\
    |m-n|\tau(K) + \max(m,n) -1& \text{if } \tau(K) \ge 0 \text{ and }\epsilon(K) = -1.
    \end{cases}
    \end{equation}
In the case where $m=n$, we have 
\begin{equation}\label{main eq 3}
    \tau(Q_{m,m}(K)) = \begin{cases}
    0 & \text{if } \tau(K) < 0, \\
    m-1 & \text{if } \tau(K) = 0, \\
    m & \text{if } \tau(K) > 0.
    \end{cases}
    \end{equation}
Also,
\begin{equation}
    \epsilon(Q_{m,n}(K)) = \begin{cases}
        0 & \text{if } \tau(K) = \epsilon(K) = 0, \\
        1 & \text{otherwise.}
    \end{cases}
\end{equation}
\end{theorem}

\begin{remark}
    To put this computation into context, we make the following remark about $(1,1)$-patterns. Given a fixed pattern knot $P$ and a fixed companion knot $K$, there are many available tools in the literature that compute the $\tau$-invariant of $P(K)$, stemming from bordered Heegaard Floer homology. However, it is challenging to find a closed formula for $\tau(P(K))$ as we vary the companion knot. The formula above recovers the results in the literature for satellites along the Whitehead double $Q_{1,1}$ \cite[Theorem 1.4]{Hed07} and the Mazur pattern $Q_{2,1}$ \cite[Theorem 1.6]{Lev16}. Partial results for this formula have also been computed in \cite[Proposition 3.5]{Ray15}, using certain conditions on the Thurston Bennet number.\footnote{In Definition \ref{def gmp}, we require $m$ and $n$ to be positive because if $m$ (equivalently $n$) is $0$, then $Q_{0,p} \simeq Q_{p,0}$ is the $(p,1)$-cable. Hom has computed this case in \cite{Hom14a}, and the result does not follow the pattern of Theorem \ref{main thm}. This is expected, since the methods we used to obtain the results above rely on the parameters $m$ and $n$ being strictly positive.} Since this paper was first posted, Bodish \cite{Bod24} studied similar families of patterns, called the $n$-twisted generalized Mazur patterns and denoted by $Q^{i,j}_n$. Similar to our approach, he showed that satellites along those patterns are never slice in any rational homology $4$-ball by computing their $\tau$ and $\epsilon$ invariants. Notably, the patterns $Q_{m,1}$ in this paper overlap with the patterns $Q^{0,m-1}_0$, but they differ for $Q_{m,n}$ when $n\ge 2$.
\end{remark}

Similar to \cite{Lev16}, we use techniques in bordered Heegaard Floer homology, due to Lipshitz, Ozsv{\'a}th and Thurston \cite{LOT18}, which is well-adapted to study 3-manifolds with parametrized boundary. In particular, we use the immersed curve interpretation of the bordered pairing theorem developed by Hanselman \cite{Han23}. The proof strategy for Theorem \ref{main thm} is inspired by \cite[Section 6]{CH23}, which recovers Levine's computation using immersed curve techniques.

We are now ready to construct our infinite family of winding number $\pm 1$ pattern knots which induce non-surjective satellite operators on the rational concordance groups. Since $\epsilon(Q_{m,m-1}(K)) \neq -1$ for all $K \subset S^3$, we may pick a knot $L \subset S^3$ such that $\epsilon(L) = -1$, such as the left handed trefoil. Then, $P_{m,m-1} = Q_{m,m-1} \# -L \subset S^1 \times D^2$ satisfies the conclusion of Theorem \ref{Kirby problem}. Thus, we have shown the following.

\begin{corollary}
        For $m > 1$, the infinite family of winding number 1 pattern knots $P_{m,m-1}$ have the property that $P_{m,m-1}(K)$ is not smoothly slice in any rational homology $4$-ball for any knot $K \subset S^3$. 
\end{corollary}

The computation of Theorem \ref{main thm} also allows us to recover the pattern knot genus of generalized Mazur patterns $Q_{m,n}$. Recall that for a pattern $P$ with winding number $w(P)$, a \textit{relative Seifert surface} for $P$ is a surface $\widehat{\Sigma}$ in $S^1 \times D^2$ such that the interior of $\widehat{\Sigma}$ is disjoint from $P$, and the boundary of $\widehat{\Sigma}$ consists of $P$ together with $w(P)$ coherently oriented longitudes. The \textit{genus} of a pattern $g(P)$ inside the solid torus is defined as the minimal genus of a relative Seifert surface for $P$. For a satellite knot $P(K)$ with nontrivial companion $K$, a classical result of Schubert \cite{Sch53} shows the relation between the three-genus of the satellite knot $g(P(K))$ and $w(P)$, $g(P)$, and $g(K)$:
\begin{equation} \label{schubert genus equation}
    g(P(K)) = |w(P)|g(K) + g(P).
\end{equation}
We know $w(Q_{m,n})$, so the value of $g(Q_{m,n})$ is determined by $g(Q_{m,n}(K))$ and $g(K)$ for some companion knot $K$ of our choice. Using the fact that knot Floer homology detects knot genus, we obtain a formula for the genus of $Q_{m,n}$.

\begin{proposition} \label{genus qmn}
    The patterns $Q_{m,n}$ have genus $g(Q_{m,n}) = \min(m,n)$.
\end{proposition}

We end with a discussion on a related problem. While Theorem \ref{main thm} allows us to compute the $\tau$ and $\epsilon$-invariant of generalized Mazur patterns, one might hope that there is a general formula for $\tau(P(K))$ and $\epsilon(P(K))$ for any $(1,1)$-unknot pattern $P$. Moreover, we know that $(1,1)$-unknot patterns can be parameterized by a pair of integers $(r,s)$, so we expect that $\tau(P(K))$ and $\epsilon(P(K))$ can also be obtained with respect to $r$ and $s$ \cite{Che19}.

\begin{Question} \label{intro question}
    Is there a closed formula for the $\tau$-invariant of satellite knots with $(1,1)$-unknot patterns?
\end{Question}

\textbf{Organization.}
In Sections \ref{review background} and \ref{intro immersed}, we review some background from bordered and immersed Heegaard Floer homology and introduce the concordance invariants $\tau$ and $\epsilon$. In particular, we focus on obtaining immersed curves for knot complements, methods to recover Alexander gradings, and the relevant pairing theorems which give a strategy for computing $\tau(Q_{m,n}(K))$ and $\epsilon(Q_{m,n}(K))$. In Section \ref{general mazur patterns}, we describe the generalized Mazur patterns and some properties, including a constructive procedure to recover their bordered Heegaard diagrams, and their 2-bridge link representation. Here, we recover Proposition \ref{qmn isotopic}. In Section \ref{IMMERSED COMPUTATION}, we find both the $\tau$ and $\epsilon$-invariant for $Q_{m,n}$ using techniques from immersed Heegaard Floer homology, proving Theorem \ref{main thm}. We also use the pairing diagrams to recover Proposition \ref{genus qmn}. Lastly, Appendix \ref{compute (m,n)} recovers the computation for the $\tau$-invariant using only the ordinary bordered theory. Throughout the paper, the coefficients of Floer homology groups are taken in $\mathbb{F}=\mathbb{Z}/2\mathbb{Z}$.\\

\textbf{Acknowledgements.}
The authors are grateful to Kristen Hendricks and Abhishek Mallick for their guidance, to Robert Lipshitz, Jonathan Hanselman, Wenzhao Chen, and Adam Levine for helpful discussions, and to the Summer 2023 DIMACS REU for providing the opportunity for their research.

\section{Knot Floer homology}\label{review background}
In this section, we will briefly review those features of most relevance. We assume that the reader is familiar with Heegaard Floer homology. For an introductory overview, see \cite{Man14} and \cite{OS06}. In Sections \ref{review KHF} to \ref{CFA hat}, we closely follow the notation from \cite{Hom14a, Hom20}. In Section \ref{sec: CFD hat}, we follow the notation from \cite{Lev16}.

\subsection{The knot Floer complex and concordance invariants} \label{review KHF}
 To each $K\subset S^3$ we associate a $\mathbb{Z}\oplus \mathbb{Z}$-filtered (freely finitely generated) $\mathbb{F}[U,U^{-1}]$ chain complex $CFK^\infty(K)$, called the \textit{full knot Floer complex} of $K$\cite{OS03}. This complex is well-defined up to chain homotopy equivalence. It  admits two $\mathbb{Z}$-gradings called the \textit{Alexander} and \textit{Maslov} (or \textit{homological}) gradings, denoted respectively by the maps $A,M:I\to \mathbb{Z}$, where $I$ is a finite set of points specified by $K$.
    
    Equivalently, $CFK^\infty$ can be seen as an $\mathbb{F}$-vector space, freely generated by elements of the form 
    \[
    [x;i,j],\textrm{ where } x\in \mathbb{T}_\alpha \cap \mathbb{T}_{\beta},\ (i,j)\in \mathbb{Z}\oplus \mathbb{Z},\ \textrm{and } j-i=A(x).\]

    \noindent The triple $[x;i,j]$ corresponds to the generator $U^{-i}x$, and the $(i,j)$-filtration level is given by $\mathcal{F}_{i,j}=\{[x,i',j']\in CFK^{\infty}: i'\le i, j'\le j\}$. 
    The differential $\partial$ decreases the Maslov grading by 1, respects the Alexander filtration, and is $U$-equivariant, that is, 
    \begin{itemize}
        \item $M(\partial x)=M(x)-1$
        \item $A(\partial x)\le A(x)$
        \item $\partial(U^nx)=U^n\partial x$
    \end{itemize}
    Moreover, multiplication by $U$ decreases each filtration level by 1, lowers the Maslov grading by 2, and the Alexander grading by 1. In other words, 
    \begin{itemize}
        \item $U[x;i,j] = [x;i-1,j-1]$
        \item $M(U[x;i,j])=M([x;i-1,j-1])=M([x;i,j])-2$
        \item $A(U[x;i,j])=A([x;i-1,j-1])=A([x;i,j])-1$
    \end{itemize}
The Maslov and Alexander gradings for $[x;i,j]$ are then given by
\[M([x;i,j])=M(x)+2i,\ \ \  A([x;i,j])=j-i.\]
    
    Graphically, we can represent $CFK^\infty$ on a plane by drawing the element $[x;i,j]$ at $(i,j)$ and the differentials as arrows that point (non-strictly) downwards and to the left, as seen in Figure \ref{fig:RHT}. Multiplication by $U$ decreases the Alexander grading of a generator by 1. The $j$ coordinate is the generator's Alexander grading, and the $i$ coordinate is the negative of its $U$ power. The Maslov grading is not represented in this picture.
    \begin{example}
        The full knot Floer complex for the right-handed trefoil $T_{2,3}$ is depicted in Figure \ref{fig:RHT}. As a $\mathbb{F}[U,U^{-1}]$-vector space, it has three generators $a,b,c$, with differentials
        \[\partial a = Ub+c,\ \ \partial b = \partial c= 0.\]
        The Alexander gradings of $a,b,c$ are given by their $j$-coordinates, which are $1,0,-1$, respectively. The Maslov gradings of $a,b,c$ are $-1, 0,-2$, respectively. The homology of this complex is generated by $[b]=[U^{-1}c]$ over $\mathbb{F}[U,U^{-1}]$.
\begin{figure}[!htb]
    \centering
    \includegraphics[scale =0.7]{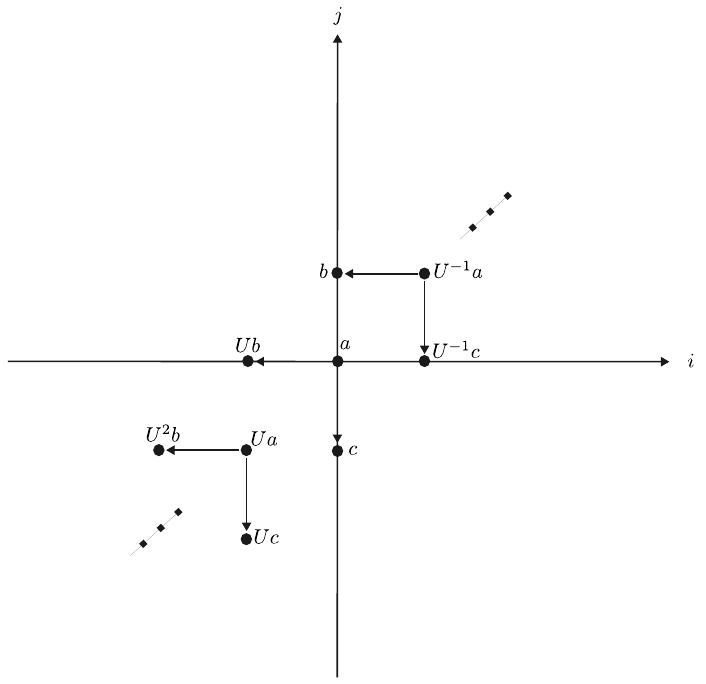}
    \caption{The full knot Floer complex for the right-handed trefoil.}
    \label{fig:RHT}
\end{figure}
\end{example}
    
   We can obtain various flavors of the knot Floer complex by taking different subcomplexes of $CFK^\infty(K)$. For a set $S\subset \mathbb{Z}\oplus \mathbb{Z}$, let $C\{S\}$ be the set of elements in $CFK^\infty(K)$ whose $(i,j)$-coordinates are in $S$. We then define
    \[CFK^-(K)= C\{i\le 0\}\]
    to be the $\mathbb{F}[U]$-module whose elements have non-positive $i$-coordinates, and whose differential is the induced differential. This complex has a natural $\mathbb{Z}$-filtration, induced by the Alexander filtration of $CFK^\infty(K)$. Denote the associated graded of $CFK^-(K)$ by $gCFK^-(K)$, and let the homology of the associated graded be 
    \[HFK^-= H_*(gCFK^-(K)).\]
    \noindent Similarly, we can also take the $\mathbb{Z}$-filtered chain complex 
    \[\widehat{CFK}(K)=C\{i= 0\}\]
    with the induced differential, and denote the homology of the associated graded of $\widehat{CFK}(K)$ by \[\widehat{HFK}(K)= H_*(g\widehat{CFK}).\] This is commonly referred to as the \textit{knot Floer homology} of $K$. While the homology of $\widehat{CFK}(K)$ is always $\mathbb{F}$, the homology of $g\widehat{CFK}(K)$ is more interesting. As a bigraded vector space, the knot Floer homology decomposes as
    \[
    \widehat{HFK}(K) = \bigoplus_{i,j}\widehat{HFK}_i(K,j),
    \]
    where $i$ and $j$ indicate the Maslov and Alexander grading, respectively. It also satisfies symmetry under orientation reversal \cite[Section 3.5]{OS04a}:
\[\widehat{HFK}_i(K,j)=\widehat{HFK}_{i-2j}(K,j).\]
    Furthermore, the knot Floer homology categorifies the Alexander polynomial in the following sense. Its graded Euler characteristic is the Alexander polynomial \cite[Equation $(1)$]{OS04a}:
    \[\Delta_K(t) = \sum_{i,j}(-1)^i\dim\widehat{HFK}_j(K,s)t^s.\]
    \noindent While the Alexander polynomial bounds the Seifert genus of $K$ from below, the knot Floer homology detects the genus \cite{OS04b} by
    \[g(K)=\max\{s|\ \widehat{HFK}(K,s)\ne 0\}. \]
    And whereas the Alexander polynomial obstructs fiberedness, knot Floer homology detects it \cite{Ghi08, Ni07}:
    \[K\textrm{ is fibered }\Longleftrightarrow \widehat{HFK}(K,g(K))=\mathbb{F}.\]

     Note that the full knot Floer complex is defined over the base ring $\mathbb{F}[U,U^{-1}]$. We may define analogously a bigraded chain complex over the base ring $\mathcal{R} = \mathbb{F}[U,V]/UV$, denoted by $CFK_\mathcal{R}(K)$. To obtain $CFK_\mathcal{R}(K)$, we decorate each vertical arrow in $CFK^\infty(K)$ with $V^n$, where $n$ is its vertical length, then set all the $UV$ arrows to 0. In particular, $CFK^\infty(K)$ and $CFK_\mathcal{R}(K)$ have the same set of original generators and differentials, except that for $CFK^\infty(K)$ we extend them linearly over $\mathbb{F}[U,U^{-1}]$, whereas for $CFK_\mathcal{R}(K)$ we extend them linearly over $\mathcal{R}$. We may also take subcomplexes of $CFK_\mathcal{R}(K)$ to obtain different flavors of the knot Floer chain complex over $\mathcal{R}$, such as $\widehat{CFK}_\mathcal{R}(K)$. We will make use of these chain complexes over $\mathcal{R}$ in Section \ref{IMMERSED COMPUTATION}.

    Next, we introduce some concordance invariants arising from the knot Floer complex. Recall that $\widehat{CF}$ is an invariant of 3-manifolds, that $\widehat{HF}$ is the homology of $\widehat{CF}$, and that in particular $\widehat{HF}(S^3)=\mathbb{F}$ (see \cite{OS04b, OS04c}).
    For a knot $K\subset S^3$, the Alexander filtration on $\widehat{CFK}(K)$ induces a spectral sequence that converges to $\widehat{HF}(S^3)$. The invariant $\tau(K)$ is defined to be the Alexander grading of the unique cycle that survives to the $E_\infty$ page. 
    By symmetry of the full knot Floer complex, we have equivalently
     \[\tau(K)=-\max\{s\mid U^n\cdot HFK^-(K,s)\ne 0 \textrm{ for all } n\ge 0\}.\] 
    In other words, $\tau(K)$ is minus the Alexander grading of the non-vanishing generator for $\mathbb{F}[U]$ in $HFK^-(S^3, K)$. This is different from but equivalent to the original definition of $\tau(K)$ in \cite{OS03}; for a proof of the equivalence, see \cite[Lemma A.2]{OS08}.

    Define the \textit{horizontal complex}
    \[C^{\mathrm{horz}}=C\{j=0\}\]
    to be the subquotient complex of $CFK^\infty(K)$, for which the elements have zero $j$-coordinates, and the differential $\partial^{\mathrm{horz}}$ is the induced differential. The horizontal complex has a $\mathbb{Z}$-filtration induced by $\CFK^\infty(K)$. Similarly, we 
    define
    the \textit{vertical complex}
     \[C^{\mathrm{vert}}=C\{i=0\},\]
 to be the subquotient complex along the $i$-axis, which we previously called $\widehat{CFK}(K)$, with induced differential $\partial^{\mathrm{vert}}$. Note that the vertical complex and the horizontal complex are homotopy equivalent.

For any $\mathbb{Z}\oplus \mathbb{Z}$-filtered chain complex $(C, \partial)$, we say that $\{x_i\}$ is a \textit{filtered basis} for $(C, \partial)$ if for all pairs $(a,b)$, the set 
\[\big\{x_i\mid x_i\in C\{i\le a, j\le b\}\big \}\]
is a basis for $C\{i\le a, j\le b\}$. Let $\{\eta_i\}$ be a filtered basis over $\mathbb{F}[U]$ for a \textit{reduced} complex $CFK^-(K)$, that is,  the arrows in $CFK^-(K)$ point strictly downwards or to the left (or both). We say that $\{\eta_i\}$ is \textit{horizontally simplified} if exactly one of the following situations holds for every $\eta_i$:
\begin{enumerate}
    \item $\eta_i\in \im(\partial^{\mathrm{horz}})$, and there exists a unique $\eta_{i-1}$ such that $\partial^{\mathrm{horz}}\eta_{i-1}=\eta_i$. \label{itemone}
    \item $\eta_i\not\in \im(\partial^{\mathrm{horz}})$, but $\eta_i\in \ker(\partial^{\mathrm{horz}})$ 
    \label{itemtwo}
    \item $\eta_i\not\in \ker(\partial^{\mathrm{horz}})$, and $\partial^{\mathrm{horz}}\eta_i=\eta_{i+1}$.
    \label{itemthree}
\end{enumerate}
    Recall that $H_*(C^{\mathrm{horz}})\cong H_*(\widehat{CFK(K)})\cong \mathbb{F}$, i.e. the horizontal complex is generated by one distinguished basis element upon taking homology. After reordering, we call this distinguished element $\eta_0$.

    A vertically simplified basis $\{\xi_j\}$ is defined similarly, replacing $\partial^{\mathrm{horz}}$ by $\partial^{\mathrm{vert}}$. Let $\xi_0$ be the distinguished non-vanishing basis element that generates $H_*(C^{\mathrm{vert}})$ after reordering the basis. Note that if we have a horizontally simplified basis $\{\eta_i\}$, then $\tau(K)=-A(\eta_0).$ If we have a vertically simplified basis $\{\xi_j\}$, then $\tau(K)=A(\xi_0).$

    Hom \cite[Lemma 3.2 and 3.3]{Hom14a} showed that $CFK^-(K)$ is always homotopy equivalent to a reduced chain complex $\mathcal{C}$, which admits a horizontally simplified basis $\{\eta_i\}$ such that some $\eta_i$ is the distinguished element of a vertically simplified basis. We then consider the position of $\eta_i$ in the horizontal complex. It can have three situations, as enumerated above. To each situation we assign the numbers $1$, $0$ and $-1$, respectively; see Figure \ref{fig:epsilon}.

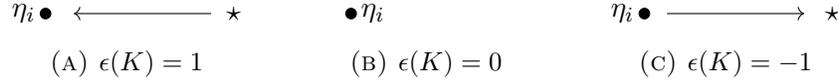
\begin{figure}[!htb]
  \begin{subfigure}{0.2\textwidth}
    \begin{tikzpicture}
      \node at (2.5,0) {$\star$};
      \filldraw (0,0) circle (2pt) node[left] {$\eta_i$};
      \draw[->, shorten >=5pt] (2.2,0) -- (0.2,0);
    \end{tikzpicture}
    \caption{$\epsilon(K)=1$}
  \end{subfigure}
  \hspace{1cm} 
  \begin{subfigure}{0.13\textwidth}
    \begin{tikzpicture}
      \filldraw (3,0)circle (2pt) node[right]{$\eta_i$};
    \end{tikzpicture}
    \caption{$\epsilon(K)=0$}
  \end{subfigure}
  \hspace{1cm} 
  \begin{subfigure}{0.2\textwidth}
    \begin{tikzpicture}
      \filldraw (0,0) circle (2pt) node[left] {$\eta_i$};
      \node at (2.5,0) {$\star$};
      \draw[->, shorten >=5pt] (0.3,0) -- (2.3,0);
    \end{tikzpicture}
    \caption{$\epsilon(K)=-1$}
  \end{subfigure}
  \caption{Three positions of $\eta_i$ in the horizontal complex.}
  \label{fig:epsilon}
\end{figure}
This assignment is well-defined up to concordance, giving us the \textit{$\epsilon$-invariant}, 
\[\epsilon: \mathcal{C}\to \{-1,0,1\}\]
    which satisfies the following properties \cite{Hom14a}:
    \begin{itemize}
        \item $\epsilon(K)= -\epsilon(\Bar{K})$.
        \item If $K$ is alternating, then $\epsilon(K) =\mathrm{sgn}(\tau(K)).$
        \item If $\epsilon(K)=\epsilon(K')$, then $\epsilon(K\# K')=\epsilon(K)=\epsilon(K')$. If $\epsilon(K)=0$, then $\epsilon(K\# K')=\epsilon(K').$
        \item If $\epsilon(K)=0$, then $\tau(K)=0$.
    \end{itemize}
By symmetry, $CFK^-(K)$ also admits a vertically simplified basis $\{\xi_j\}$ such that some $\xi_j$ is the distinguished element of a horizontally simplified basis, and we can similarly obtain the $\epsilon$-invariant by looking at vertical position of $\xi_j$.
\begin{example}
    In the full knot Floer complex of the right-handed trefoil shown in Figure \ref{fig:RHT}, $\{a, b, c\}$ is a vertically simplified basis, with $b$ being the distinguished basis element. So we have $\tau(T_{2,3}) = A(b)=1$. Looking at the horizontal position of $b$, we get $\epsilon(T_{2,3}) = 1$.  
\end{example}
Even though we defined the $\tau$ and $\epsilon$ invariants using the $CFK$ complexes, they can be computed in exactly the same way using the $CFK_\mathcal{R}$ complexes. This is because $CFK^\infty(K)$ and $CFK_\mathcal{R}(K)$ have the same set of generators, and almost the same set of differentials on those generators, except that on $CFK_\mathcal{R}(K)$ we decorate the vertical arrows with the extra $V$ variable. In particular, $CFK^\infty(K)$ and $CFK_\mathcal{R}K$ have the same vertical and horizontal complexes. So if we have a horizontally simplified basis $\{\eta_i\}$ on $CFK_\mathcal{R}(K)$, then $\tau(K)=-A(\eta_0) = A(\xi_0)$ and $\epsilon(K)$ is determined by the vertical position of $\eta_0$. This is the strategy we will employ to determine the $\epsilon$-invariant in Section \ref{IMMERSED COMPUTATION}.

   The $\tau$-invariant and $\epsilon$-invariant often admit a nice decomposition under the satellite operation. For instance, Hom \cite[Theorem 1]{Hom14a} showed that for the $(p,q)$-cable of a knot $K$, denoted $K_{p,q}$, we have 
    \begin{align*}
        \tau(K_{p,q})&= \begin{cases}
        p\tau(K)+\frac{(p-1)(q-1)}{2} & \text { if } \epsilon(K)=1, \text{ or } \epsilon(K)=0 \text{ and } q>0, \\
        p\tau(K)+\frac{(p-1)(q+1)}{2} & \text { if } \epsilon(K)=-1, \text{ or } \epsilon(K)=0 \text{ and } q<0.\\
    \end{cases}\\
    \epsilon(K_{p,q})&=\begin{cases}
     \epsilon(K) & \text { if } \epsilon(K)\ne 0,\\
         -1 & \text { if } q<-1 \textrm{ and }\epsilon(K)=0,\\
        0 & \text { if } |q|=1 \textrm{ and } \epsilon(K)=0,\\ 
         1 & \text { if } q>1 \textrm{ and } \epsilon(K)=0. \end{cases}
   \end{align*}
   Shortly thereafter, Levine gave a formula for the $\tau$ and $\epsilon$ invariants of satellite knots with Mazur pattern $Q$, as stated in Theorem \ref{Levine theorem}:
    \begin{align*}
    &\tau(Q(K))= \begin{cases}\tau(K) & \text { if } \tau(K) \leq 0 \text { and } \epsilon(K) \in\{0,1\}, \\ \tau(K)+1 & \text { if } \tau(K)>0 \text { or } \epsilon(K)=-1.\end{cases}\\
    &\epsilon(Q(K))= \begin{cases}0 & \text { if } \tau(K)=\epsilon(K)=0, \\ 1 & \text { otherwise.}\end{cases}
    \end{align*}
\noindent The main goal of this paper is to find an analogous result for the generalized Mazur patterns $Q_{m,n}$, that is, calculate $\tau(Q_{m,n}(K))$ and  $\epsilon(Q_{m,n}(K))$.

\subsection{Bordered Heegaard Floer homology}\label{review bordered}
The bordered version of Heegaard Floer homology extends the theory to manifolds with boundary. In this section, we give a brief tour of the tools we will use in bordered Heegaard Floer homology; a more comprehensive discussion is given in \cite{LOT18}. For most of this section, we follow the notation in \cite{Hom14a}.

First, we define the algebraic structures in bordered Heegaard Floer homology, such as $A_\infty$-modules and type $D$ structures. Let $\mathcal{A}$ be a unital differential graded algebra over $\mathbb{F}$, equipped with a subalgebra of idempotents $\mathcal{I}\subset \mathcal{A}$ generated by an orthogonal basis that sum up to 1. Let $M$ be a (right) differential graded module over $\mathcal{A}$. Denote $M[n]$ as the module defined by $M[n]_d=M_{d-n}$. We say that $M$ is a (right unital) \textit{$\mathcal{A}_\infty$-module} if the family of right $\mathcal{I}$-actions
\[
m_i: M\otimes \mathcal{A}^{i-1}\to M[2-i],\ \ \  i\ge 1
\]
satisfies the $A_\infty$ relations 
\begin{align*}
0 & =\sum_{i=0}^n m_{n-i+1}\left(m_{i+1}\left(x \otimes a_1 \otimes \cdots \otimes a_i\right) \otimes a_{i+1} \otimes \cdots \otimes a_n\right) \\
& +\sum_{i=1}^n m_{n+1}\left(x \otimes a_1 \otimes \cdots \otimes a_{i-1} \otimes d\left(a_i\right) \otimes a_{i+1} \otimes \cdots \otimes a_n\right) \\
& +\sum_{i=1}^{n-1} m_n\left(x \otimes a_1 \otimes \cdots \otimes a_{i-1} \otimes a_i a_{i+1} \otimes a_{i+2} \otimes \cdots \otimes a_n\right)
\end{align*}
and the unital conditions
\begin{align*}
    m_2(x,1)&=x,\\m_i(x,\cdots, 1,\cdots,)&=0, \ \ \ i>2.
\end{align*}

A \textit{type D structure over $\mathcal{A}$} is a $\mathbb{F}$-vector space $N$, with left $\mathcal{I}$-action satisfying 
\[N=\bigoplus_{i=1}^n\iota_iN,\]
and a map 
\[\delta_1:N\to \mathcal{A}\otimes_{\mathcal{I}}N,\]
satisfying the type $D$ relations
\[(\mu\otimes \Id_N)\circ (\Id_\mathcal{A}\otimes \delta_1)\circ \delta_1+(d\otimes \Id_N)\circ \delta_1 = 0,\]
where $\mu:\mathcal{A}\times \mathcal{A}\to \mathcal{A}$ denotes the multiplication map on $\mathcal{A}$. Inductively, we can define maps on $N$
\[\delta_k: N\to \mathcal{A}^{\otimes_k}\otimes_\mathcal{I}N\]
by setting 
\begin{align*}
    \delta_0&=\Id_N\\
    \delta_i &= \Id_\mathcal{A}^{\otimes_{i-1}}\otimes \delta_1\circ \delta_{i-1}.
\end{align*}

Given an $A_\infty$-module $M$ and a type $D$ structure $N$, we define the \textit{box tensor product} $M\boxtimes N$ to be the $\mathbb{F}$-vector space $M\otimes N$, equipped with the differential map
\[\partial^{\boxtimes}(x\otimes y)=\sum_{k=0}^\infty (m_{k+1}\otimes \Id_N)(x\otimes \delta_k(y)).\]

    Let $Y$ be a compact, oriented $3$-manifold with connected boundary $\partial Y=F$. To the surface $F$, we associate a differential graded algebra $\mathcal{A}(F)$. To the $3$-manifold $Y$,
    we associate two invariants: $\widehat{CFD}(Y)$, which is a type $D$ structure, and $\widehat{CFA}(Y)$, which is a right $\mathcal{A}_{\infty}$-module over $\mathcal{A}(F)$. To a knot $K_1$ in $Y_1$, we may associate either $\widehat{CFA}(Y_1,K_1)$, which is $\mathbb{Z}$-filtered $\mathcal{A}_\infty$-module, or $CFA^{-}(Y_1,K_1)$, which is an $\mathcal{A}_\infty$-module over $\mathcal{A}(F)$ with ground ring $\mathbb{F}[U]$. The pairing theorem \cite[Theorem 1.3 and 11.19]{LOT18} states that gluing 3-manifolds along their boundaries corresponds to taking the box tensor of their invariants. More concretely, let $Y_1$ and $Y_2$ be compact oriented 3-manifolds with boundary, along with an orientation-reversing diffeomorphism $f$ from $\partial Y_1$ to $\partial Y_2$. Let $Y=Y_1\cup_f Y_2$ be the 3-manifold obtained by gluing $Y_1$ and $Y_2$. Then there exists a homotopy equivalence 
    \[\widehat{CF}(Y)\simeq \widehat{CFA}(Y_1)\boxtimes \widehat{CFD}(Y_2).\]
    Moreover, if we have a knot $K_1\subset Y_1$ whose image $K\subset Y$ under identification is null-homologous, then we get a homotopy equivalence of $\mathbb{F}[U]$-modules
     \[
    gCFK^-(Y,K)\simeq CFA^-(Y_1,K_1)\boxtimes \widehat{CFD}(Y_2),
    \]
    where $gCFK^-(Y, K)$ denotes the associated graded of $CFK^-(Y,K)$. We will discuss these invariants in more detail in Sections \ref{CFA hat} and \ref{sec: CFD hat}.   
    \begin{remark}
    There is an alternate version of the pairing theorem, which gives a homotopy of $\mathbb{Z}$-filtered chain complexes
    \[\widehat{CFK}(Y,K)\simeq \widehat{CFA}(Y_1, K_1)\boxtimes \widehat{CFD}(Y_2),\]
    where $\widehat{CFA}$ is the hat version of $CFA^{-}$. For more detail, see \cite{LOT18}.
    \end{remark}
    
    In this paper, we are concerned with satellite knots, so we restrict our discussion to the case where $F$ is the torus $T^2$. We  let $Y_1$ be the solid torus $V$ equipped with a pattern knot $P$, and $Y_2$ be the bordered manifold $X_K= S^3\backslash N^{\mathrm{o}}(K)$ with the bordered structure given by the $0$-framing (where $N^{\mathrm{o}}(K)$ denotes a regular neighborhood of $K$). Upon gluing $Y_1$ and $Y_2$, the knot $K\subset Y$ becomes the satellite knot $P(K)$ inside of $S^3$. The pairing theorem then tells us that
    \begin{equation} \label{pairing equation}  
    gCFK^{-}(P(K))\simeq CFA^{-}(V,P)\boxtimes \widehat{CFD}(X_K).
    \end{equation}
    Setting $P$ to be the generalized Mazur patterns, we see that computing $gCFK^-(Q_{m,n}(K))$ reduces to finding the $A_\infty$-module for the generalized Mazur patterns and the type $D$ structure for $X_K$. We discuss how one may obtain these two components in Sections \ref{CFA hat} and \ref{sec: CFD hat}.

Next, we describe the graded algebra $\mathcal{A}(F)$ in the case where $F=T^2$; in the general case where $g(F)\ge 2$, $\mathcal{A}(F)$ is a differential graded algebra, as discussed in \cite{LOT18}. Recall that one may specify a torus via handle decomposition: a disk $D^2$ with two $1$-handles attached, such that the boundary is connected and can be capped off with a disk. We represent this information by a \textit{pointed matched circle} $(Z,z,\{a_1,a_3\},\{a_2,a_4\})$, where $Z$ is an oriented circle, $z\in Z$ is a fixed basepoint, and $\{a_1,a_3\}, \{a_2,a_4\}$ are two pairs of points on $Z$ disjoint from $z$ (see Figure \ref{fig: PMC}). We may recover the torus from the pointed matched circle; see \cite[Construction 1.2]{LOT14}. 

\begin{figure}[!htb]
    \centering
    \includegraphics[scale=0.5]{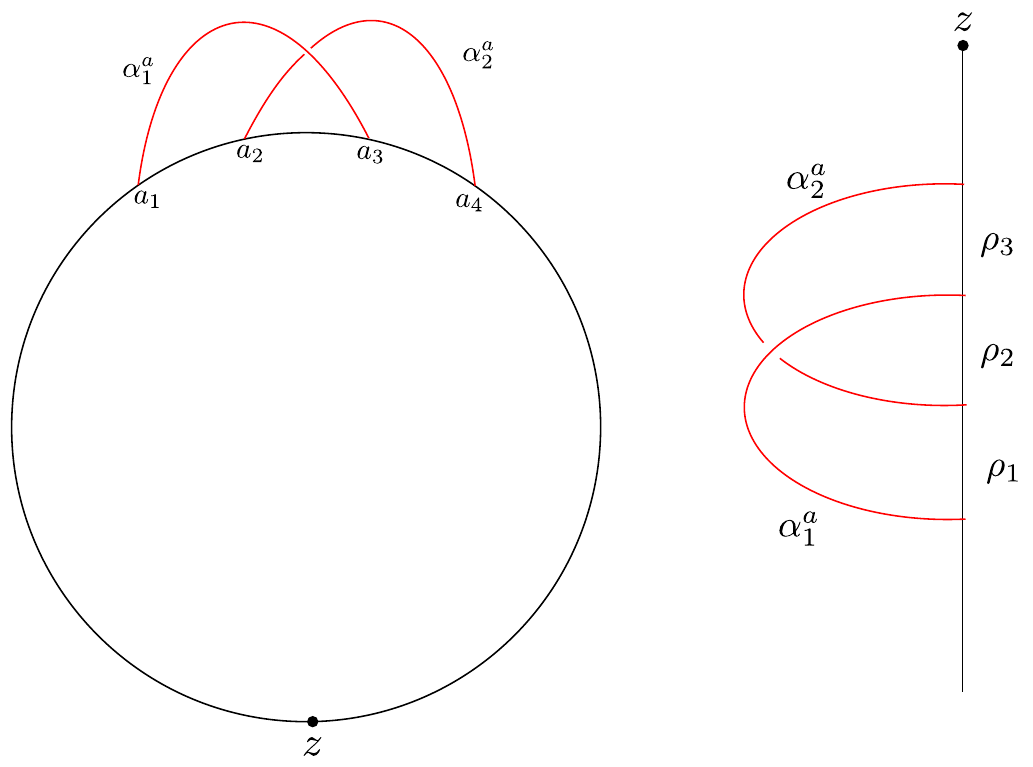}
    \caption{Left: the pointed matched circle for the surface $T^2$. Right: the same pointed matched circle cut open at $z.$}
    \label{fig: PMC}
\end{figure}
Given a torus $T^2$ parametrized by a pointed matched circle, the graded algebra $\mathcal{A}(T^2)$ is generated over $\mathbb{F}$ by two idempotents $\iota_0$ and $\iota_1$ (where $\iota_0+\iota_1=1$) and the six ``Reeb" elements $\rho_1$, $\rho_2$, $\rho_3$, $\rho_{12}$, $\rho_{23}$, $\rho_{123}$, satisfying the comptability conditions
\[
\begin{array}{ccc}
\rho_1=\iota_1 \rho_1=\rho_1 \iota_2 & \rho_2=\iota_2 \rho_2=\rho_2 \iota_1 & \rho_3=\iota_1 \rho_3=\rho_3 \iota_2 \\
\rho_{12}=\iota_1 \rho_{12}=\rho_{12} \iota_1 & \rho_{23}=\iota_2 \rho_{23}=\rho_{23} \iota_2 & \rho_{123}=\iota_1 \rho_{123}=\rho_{123} \iota_2
\end{array}
\]
and the nonzero products 
\begin{equation}\label{algebra}
\begin{array}{ccc}
\rho_1\rho_2 = \rho_{12}, &\rho_2\rho_3 = \rho_{23}, &\rho_1\rho_2\rho_3 = \rho_{12}\rho_3 = \rho_1\rho_{23} = \rho_{123}.
\end{array}
\end{equation}
Schematically, the algebra elements are shown in Figure \ref{fig:PMC}.
\begin{figure}[!htb]
    \centering
    \includegraphics[]{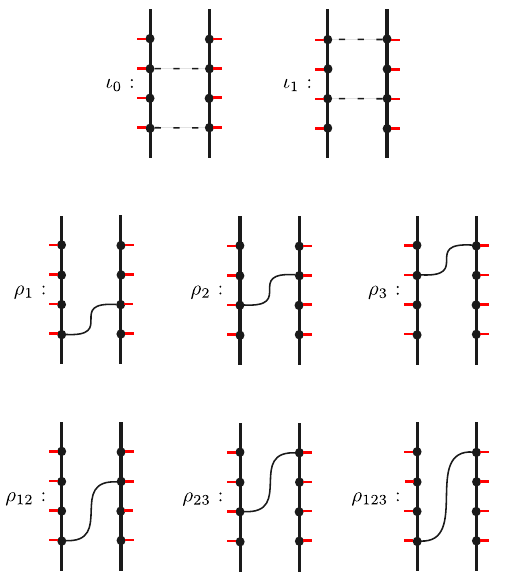}
    \caption{The idempotents and the ``Reeb" elements in $\mathcal{A}(T^2)$}.
    \label{fig:PMC}
\end{figure}

For a three manifold $Y$ with $\partial Y = T^2$, a \textit{doubly pointed bordered Heegaard diagram} for $Y$ is a tuple $\mathcal{H} = (\Bar{\Sigma}, \alpha^c, \alpha^a, \beta, z, w)$ such that 
\begin{itemize}
    \item $\Bar{\Sigma}$ is a compact oriented surface of genus $g$ with a single boundary component.
    \item There is a $(g-1)$-tuple of pairwise disjoint circles $\alpha^c = (\alpha_1^c,\cdots, \alpha_{g}^c)$ in $\Bar{\Sigma}^{\mathrm{o}}$, the interior of $\Bar{\Sigma}$.
    \item There is a pair of disjoint arcs $\alpha^a = (\alpha_1^a, \alpha_2^a)$ in $\Bar{\Sigma}\backslash \alpha^c$ with endpoints on $\partial \Bar{\Sigma}$.
    \item There is a $g$-tuple of pairwise disjoint circles $\beta = (\beta_1,\cdots, \beta_g)$ in $\Bar{\Sigma}^{\mathrm{o}}$. 
    \item $z$ and $w$ are basepoints in $\Bar{\Sigma}\backslash (\alpha^a\cup \alpha^c\cup \beta)$, with $z\in \partial \Bar{\Sigma}$.
\end{itemize}
In addition, we require that the $\alpha$-curves and $\beta$-curves intersect transversely, that the complement of the $\alpha$-curves and the complement of the $\beta$-curves are connected, and that $(\partial \Bar{\Sigma}, z, \partial \alpha_1^a, \partial \alpha_2^a)$ describes a pointed matched circle. Whenever the context is clear, we refer to doubly pointed bordered Heegaard diagrams simply as Heegaard diagrams.

A Heegaard diagram specifies a knot embedded in $Y$ by connecting $w$ to $z$ in the complement of the $\beta$ curves and connecting $z$ to $w$ in the complement of the $\alpha$ arcs. If the two sets of arcs cross, the second set crosses under the first set. Moreover, every knot in a bordered three manifold can be realized by a doubly pointed Heegaard diagram. 

A pattern knot embedded in the solid torus $V =S^1\times D^2$ is called a \textit{$(1,1)$-pattern knot} if it admits a Heegaard diagram with genus 1 \cite{Che19}. All the pattern knots in this paper are $(1,1)$-pattern knots, including the Mazur pattern whose bordered Heegaard diagram was shown in Figure \ref{fig:mazur_diagram} \cite{Lev16}. The arc $\alpha^{a}_{1}$ represents a meridian $\mu_{V} = \{\textrm{pt}\} \times \partial D^2$, and the arc $\alpha^{a}_{2}$ represents a longitude $\lambda_{V} = S^1 \times \{\textrm{pt}\}$. Note that there are no $\alpha$-arcs in the interior of our Heegaard diagrams because their genus is 1. 

\begin{figure}[!htb]
    \centering
    \includegraphics[scale=0.4]{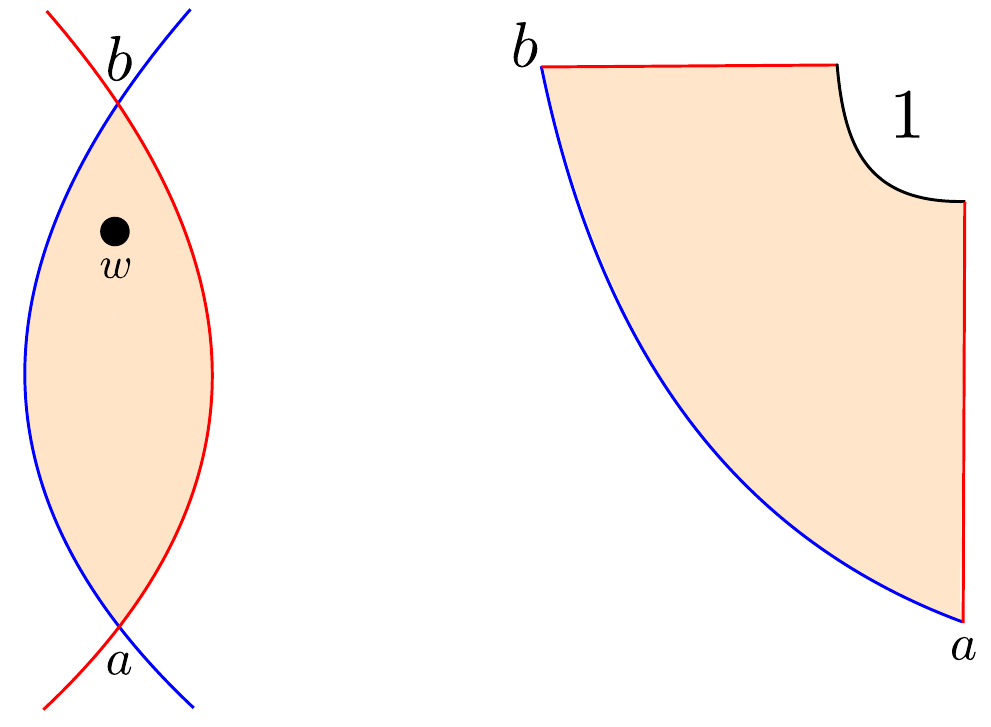}
    \caption{Pseudoholomorphic disks bounded by an $\alpha$-arc (red) and a $\beta$-arc (blue). The $CFA^{-}(V,P)$ relations corresponing to these disks are  $m_1(a) = Ub$ (left) and  $m_2(a,\rho_1) = b$ (right)}
    \label{fig:howtomakecfa}
\end{figure}

\begin{figure}[!htb]
    \centering
    \includegraphics[scale=0.4]{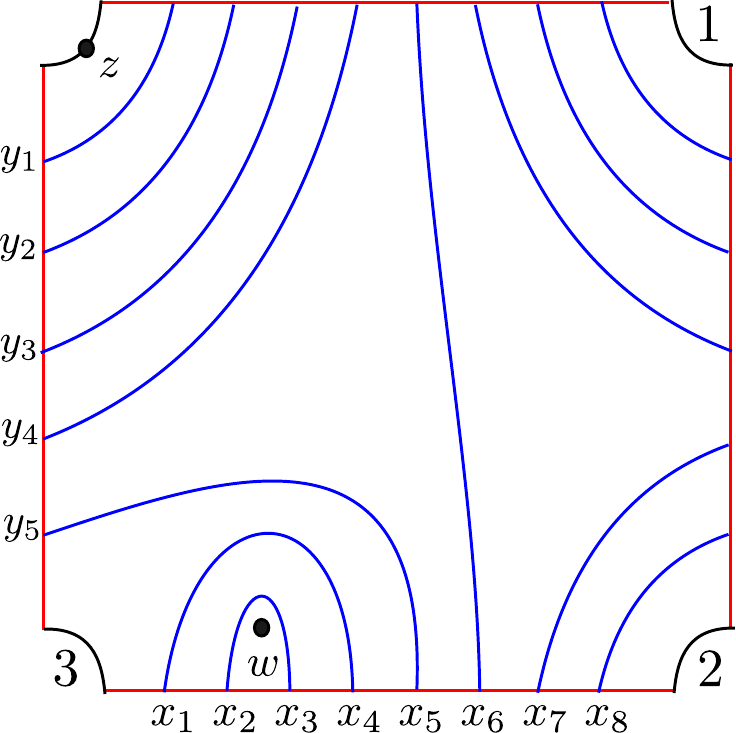}
    \caption{Genus 1 Heegaard diagram for the Mazur pattern $Q$ (equivalently $Q_{2,1}$)}.
    \label{fig:mazur_diagram}
\end{figure}

\subsection{Obtaining type $A$ structures from Heegaard diagrams} \label{CFA hat}
Recall that the pairing theorem states that 
\[gCFK^{-}(Y,K) \simeq CFA^{-}(Y_1,K_1)\boxtimes \widehat{CFD}(Y_2).\]
In this subsection, we outline how one may obtain $CFA^-(V,P)$ from a Heegaard diagram $\mathcal{H}$ associated with $P$, and in the next subsection, we discuss how one determines the type $D$ structure. We may assume $P$ to be a $(1,1)$-pattern knot, since all pattern knots are $(1,1)$ in this paper. As such, the $\mathbb{F}[U]$-vector space $CFA^-(V,P)$ is generated by the set of intersection points of $\alpha$-arcs and $\beta$-arcs, denoted $\mathfrak{S}(\mathcal{H})$. The right $\mathcal{I}$-action is given by 
\begin{align*}
x\cdot \iota_0 &=\begin{cases}
    x & \textrm{if } x \textrm{ lies on the arc } \alpha_1^a,\\
    0 & \textrm{otherwise,}
\end{cases}\\
x\cdot \iota_1 &=\begin{cases}
    x & \textrm{if } x \textrm{ lies on the arc } \alpha_2^a,\\
    0 & \textrm{otherwise.}
\end{cases}
\end{align*}
We may obtain the $A_\infty$-structure on $CFA^-(V,P)$ by counting certain pseudoholomorphic curves, giving us the family of $\mathcal{I}$-actions:
\[m_{j+1}:CFA^-(V,P)\otimes \mathcal{A}^{\otimes j}\to CFA^-(V,P).\]
Specifically, we extend the Heegaard diagram $\mathcal{H}$ to its universal cover, and count the pseudoholomorphic disks $\phi$ that are bounded by an $\alpha$-arc and a $\beta$-arc intersecting acutely at two points in $\mathfrak{S}(\mathcal{H})$. Label the intersection points by $a$ and $b$, such that if we start at $a$ and travel counterclockwise along the boundary, we encounter the $\alpha$-arc first (see Figure \ref{fig:howtomakecfa}). Note that the boundary of $\phi$ may contain parts of the pointed matched circle, which we denote by their corresponding algebra elements 
$\rho_{i_1},\cdots,\rho_{i_j}$ (concatenated when possible per equation \eqref{algebra}), in the order that we encounter travelling counterclockwise from $a$. The relation in $CFA^-(V,P)$ corresponding to $\phi$ is
\[m_{j+1}(a,\rho_{i_1},\cdots,\rho_{i_j})=U^{n_w(\phi)}b,\] where $n_w(\phi)$ is the number of times $w$ appears inside $\phi$. Note that the region bounded by $\phi$ must be a disk, and in particular has trivial fundamental group (this is true only in the genus one case). Furthermore, $\phi$ may not intersect the top left corner of $\mathcal{H}$. In Figure \ref{fig:howtomakecfa}, the corresponding relations are $m_1(a) = Ub$, and  $m_2(a) = b$, respectively. An example of finding these disks and writing their corresponding relations in $CFA^-$ is given in Example \ref{Ex:Q31}.

\subsection{Obtaining type $D$ structures from $CFK^-(K)$} \label{sec: CFD hat}

The module $\widehat{CFD}(X_K)$ may be algorithmically computed in terms of a basis in $CFK^-(K)$. The algorithm in the following theorem is originally due to \cite{LOT18}; we state a slight enhancement of it, due to \cite{Lev16}.

Recall that $CFK^-(K)$ always admit a horizontally and a vertically simplified basis up to chain homotopy. Let $\big\{\widetilde{\xi}_0, \ldots, \widetilde{\xi}_{2 n}\big\}$ and $\big\{\widetilde{{\eta}}_0, \ldots, \widetilde{{\eta}}_{2 n}\big\}$ be its vertically and horizontally simplified basis, respectively. Then the type $D$ structure $\widehat{CFD}(X_K)$ may be described as follows. 

\begin{theorem} \cite[Theorem 2.6]{Lev16}\label{thm algo} Let $K$ be a knot in $S^3$, and let $\big\{\widetilde{\xi}_0, \ldots, \widetilde{\xi}_{2 n}\big\}$ be the vertically and horizontally simplified basis described above. The type D structure $\widehat{\mathrm{CFD}}\left(X_K\right)$ satisfies the following properties:
\begin{itemize}
    \item The summand $\iota_0 \cdot \widehat{\mathrm{CFD}}\left(X_K\right)$ has dimension $2 n+1$, with designated bases $\left\{\xi_0, \ldots, \xi_{2 n}\right\}$ and $\left\{\eta_0, \ldots, \eta_{2 n}\right\}$ related by
$$
\xi_p=\sum_{q=0}^{2 n} a_{p, q} \eta_q \quad \text { and } \quad \eta_p=\sum_{q=0}^{2 n} b_{p, q} \xi_q .
$$
These elements are all homogeneous with respect to the grading by relative spin$^c$ structures.
\item The summand $\iota_1 \cdot \widehat{\mathrm{CFD}}\left(X_K\right)$ has dimension $\sum_{j=1}^n\left(k_j+l_j\right)+s$, where $s=$ $2|\tau(K)|$, with basis
$$
\bigcup_{j=1}^n\left\{\kappa_1^j, \ldots, \kappa_{k_j}^j\right\} \cup \bigcup_{j=1}^n\left\{\lambda_1^j, \ldots, \lambda_{l_j}^j\right\} \cup\left\{\mu_1, \ldots, \mu_s\right\} .
$$
\item For $j=1, \ldots, n$, corresponding to the vertical arrow $\xi_{2 j-1} \rightarrow \xi_{2 j}$, there are coefficient maps
\begin{equation} \label{vertical}
\xi_{2 j} \stackrel{D_{123}}{\longrightarrow} \kappa_1^j \stackrel{D_{23}}{\longrightarrow} \cdots \stackrel{D_{23}}{\longrightarrow} \kappa_{k_j}^j \stackrel{D_1}{\longleftarrow} \xi_{2 j-1} .
\end{equation}
\item For $j=1, \ldots, n$, corresponding to the horizontal arrow ${\eta}_{2 j-1} \rightarrow {\eta}_{2 j}$, there are coefficient maps
\begin{equation} \label{horizontal}
\eta_{2 j-1} \stackrel{D_3}{\longrightarrow} \lambda_1^j \stackrel{D_{23}}{\longrightarrow} \cdots \stackrel{D_{23}}{\longrightarrow} \lambda_{l_j}^j \stackrel{D_2}{\longrightarrow} \eta_{2 j}
\end{equation}
\item Depending on $\tau(K)$, there are additional coefficient maps
\begin{equation}\label{unstable}
\begin{cases}\eta_0 \stackrel{D_3}{\longrightarrow} \mu_1 \stackrel{D_{23}}{\longrightarrow} \cdots \stackrel{D_{23}}{\longrightarrow} \mu_s \stackrel{D_1}{\longleftarrow} \xi_0 & \tau(K)>0 \\ \xi_0 \stackrel{D_{12}}{\longrightarrow} \eta_0 & \tau(K)=0 \\ \xi_0 \stackrel{D_{123}}{\longrightarrow} \mu_1 \stackrel{D_{23}}{\longrightarrow} \cdots \stackrel{D_{23}}{\longrightarrow} \mu_0 \stackrel{D_2}{\longrightarrow} \eta_0 & \tau(K)<0\end{cases}
\end{equation}
\end{itemize}
\end{theorem} 
We refer to the subspaces of $\widehat{CFD}(X_K)$ spanned by the generators in \eqref{vertical}, \eqref{horizontal}, and \eqref{unstable} as the \textit{vertical chains, horizontal chains,} and \textit{unstable chain}, respectively. 
    
For example, running the algorithm for the right-handed trefoil $T_{2,3}$, we obtain the type $D$ structure for its complement in $S^3$, as shown in Figure \ref{fig:CFD of T_{2,3}}. The basis for $\iota_0\cdot \widehat{CFD}(X_{T_{2,3}})$ is $\{\xi_0,\xi_1,\xi_2\}$, and the basis for $\iota_1\cdot \widehat{CFD}(X_{T_{2,3}})$ is $\{\kappa, \lambda, \mu_1,\mu_2\}$. 

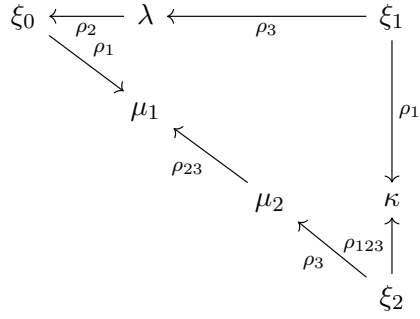
\begin{figure}[!htb]
\begin{tikzcd}
    \xi_0 \arrow[dr,"\rho_1"]
    & \lambda \arrow[l, "\rho_2"] & & \xi_1 \arrow[ll, "\rho_3"] \arrow[dd, "\rho_1"]\\
      & \mu_1 & &\\
    & & \mu_2 \arrow[ul, "\rho_{23}"] & \kappa\\
     & & & \xi_2 \arrow[ul,"\rho_3"] \arrow[u, "\rho_{123}"]
\end{tikzcd}
\caption{Obtaining $\widehat{CFD}$ form $CFK^-$ for the right-handed trefoil complement}.
\label{fig:CFD of T_{2,3}}
\end{figure}

In general, we may use Theorem \ref{thm algo} to determine the structure for $\widehat{CFD}(X_K)$, where $K$ is an arbitrary knot in $S^3$. The type D structure when $\tau(K) > 0$ is displayed in Figure \ref{fig:unstablechaint>0}, and the type D structures when $\tau(K) \le 0$ are determined similarly. In practice, we are concerned with finding a particular summand of the box tensor $CFA^{-}(V,P)\boxtimes \widehat{CFD}(X_K)$. For this purpose, we will only need to consider a neighborhood of the unstable chain in $\widehat{CFD}(X_K)$, which we call the \textit{unstable neighborhood.}

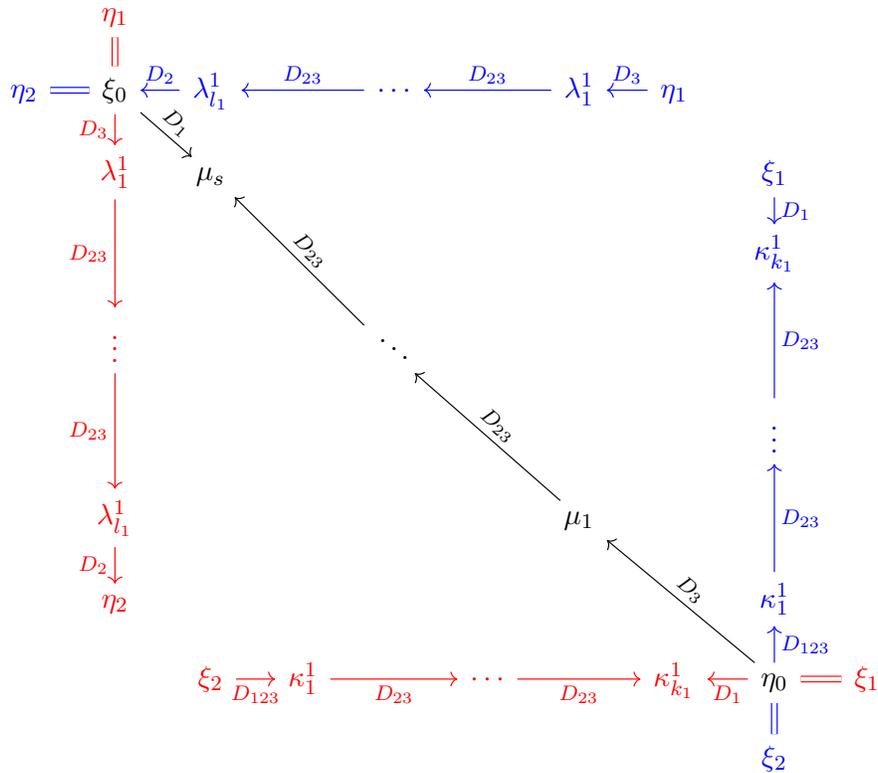
\begin{figure}[!htb]
\begin{tikzcd}[sep = small]
& \color{red}\eta_1 \arrow[equal,red]{d}
\\
\color{blue}\eta_2 \arrow[equal,blue]{r} & \xi_0  \arrow[d,swap, "D_3",red] \arrow[rd, "D_1",sloped] & \color{blue}\lambda^{1}_{l_1} \arrow[l, swap, "D_2", blue] & & \color{blue}\cdots \arrow[ll, swap, "D_{23}", blue] & & \color{blue}\lambda^{1}_{1} \arrow[ll, swap, "D_{23}", blue] & \color{blue}\eta_1 \arrow[l, swap, "D_3", blue]
\\
& \color{red}\lambda^{1}_{1} \arrow[dd, swap, "D_{23}",red]  & \mu_{s} & & & & & & \color{blue}\xi_1 \arrow[d, "D_1", blue]
\\
& & & & & & & & \color{blue}\kappa^{1}_{k_1}
\\
& \color{red}\vdots \arrow[dd, swap, "D_{23}",red] & & & \ddots \arrow[uull, "D_{23}",sloped]
\\
& & & & & & & & \color{blue}\vdots \arrow[uu, "D_{23}", blue, swap]
\\
& \color{red}\lambda^{1}_{l_1} \arrow[d, "D_2", swap, red] & & & & & \mu_1 \arrow[uull, "D_{23}",sloped]
\\
& \color{red}\eta_2 & & & & & & & \color{blue}\kappa^{1}_{1} \color{blue}\arrow[uu, "D_{23}", blue, swap]
\\
& & \color{red}\xi_2 \arrow[r, swap, "D_{123}",red] & \color{red}\kappa^{1}_{1} \arrow[rr, swap, "D_{23}", red] & & \color{red}\cdots \arrow[rr, swap, "D_{23}", red] & & \color{red}\kappa^{1}_{k_1} & \eta_0 \arrow[l, "D_1", red] \arrow[uull, "D_3",sloped] \arrow[equal,red]{r} \arrow[u, "D_{123}", swap, blue] & \color{red}\xi_1
\\
& & & & & & & & \color{blue}\xi_2 \arrow[equal,blue]{u}
\end{tikzcd}
\caption{Neighborhood of the unstable chain of $\widehat{CFD}(X_K)$ for $\tau(K)>0$. When $\epsilon =-1$, the unstable neighborhood is the left and bottom edges of the diagram square, colored in red. When $\epsilon =1$, the unstable neighborhood is the left and bottom edges of the diagram square, colored in blue.}
\label{fig:unstablechaint>0}
\end{figure}

\section{Immersed Heegaard Floer homology} \label{intro immersed}

 We introduce the immersed curve interpretation of the pairing theorem in \cite[Theorem 1.3]{LOT18}. In Section \ref{typeD to immersed}, we sketch out how one may convert a type $D$ structure to an immersed curve in the torus $T^2$. In Section \ref{pairing diagrams}, we introduce the notion of a pairing diagram and the method to determine Alexander grading. This gives us an algorithmic strategy to compute the $\tau$ and $\epsilon$ invariant of the satellite $P(K)$, which is sketched out in Section \ref{isotopy explained}.

\subsection{From type $D$ structures to immersed curves} \label{typeD to immersed} Given a type D structure $N=\widehat{CFD}(X_K)$ over the torus algebra $\mathcal{A}$, we can convert it into an immersed curve in the torus $T^2$ using an algorithm described in \cite[Sections 2.3 and 2.4]{HRW23}. 

The first step is to convert $N$ into an $\mathcal{A}$-decorated graph $\Gamma$, which is a directed graph whose vertex is either $\bullet$ or $\circ$, and whose edges are labeled with one element from $\{\emptyset, 1,2,3,12,23,123\}$, along with additional requirements specified in \cite[Section 2.4]{HRW23}. The conversion is given as follows. For each generator of $N$ which is either in $\iota_0$ or $\iota_1$, we put a $\bullet$ or a $\circ$ in $\Gamma$, respectively. Suppose that there are two vertices corresponding to generators $x$ and $y$ such that $\rho_I\otimes y$ is a summand of $\delta_1(x)$; in this case, we put a corresponding edge labeled by $I$ from $x$ to $y$ in $\Gamma$. The higher differentials $\delta_k$ in the type D structure then correspond to directed paths in $\Gamma$.  We say a decorated graph (and its associated type D structure) is \textit{reduced} if there are no edges labeled by $\emptyset$. 

Next, we convert the $\mathcal{A}$-decorated graph $\Gamma$ into an immersed train track. Let $T$ be a torus $\mathbb{R}^2/\mathbb{Z}^2$ punctured at $z=(1-\epsilon, 1-\epsilon)$, and let $\mu$ and $\lambda$ be the images of the $x$ and $y$ axes, respectively. We embed the vertices of $\Gamma$ so that the $\bullet$'s are distinct points in the interval $\{0\}\times [\frac{1}{4},\frac{3}{4}]$ within $\lambda$, and the $\circ$'s are distinct points in the interval $[\frac{1}{4},\frac{3}{4}] \times \{0\}$ within $\mu$. Then, we embed the edges of $\Gamma$ according to Figure \ref{fig: CFD to train track}. 
\begin{figure}[!htb]
    \centering
    \includegraphics[]{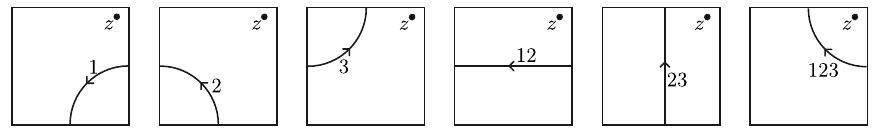}
    \caption{Correspondence between the edges of the decorated graph $\Sigma$ and directed edges in the punctured torus $T$.}
    \label{fig: CFD to train track}
\end{figure}
This gives us an immersion from $\Gamma$ to $T$, which we denote as $\alpha_K$. We further require that all intersections in the image of $\alpha_K$ are transversal, and edges intersect in points that are away from $\lambda$ and $\mu$. In this case, we call $\alpha_K$ the \textit{immersed train track} associated with $K$. Figure \ref{fig:immersed train track} shows the immersed train track associated with the right-handed trefoil.
\begin{figure}[!htb]
    \centering
    \includegraphics[scale = 0.5]{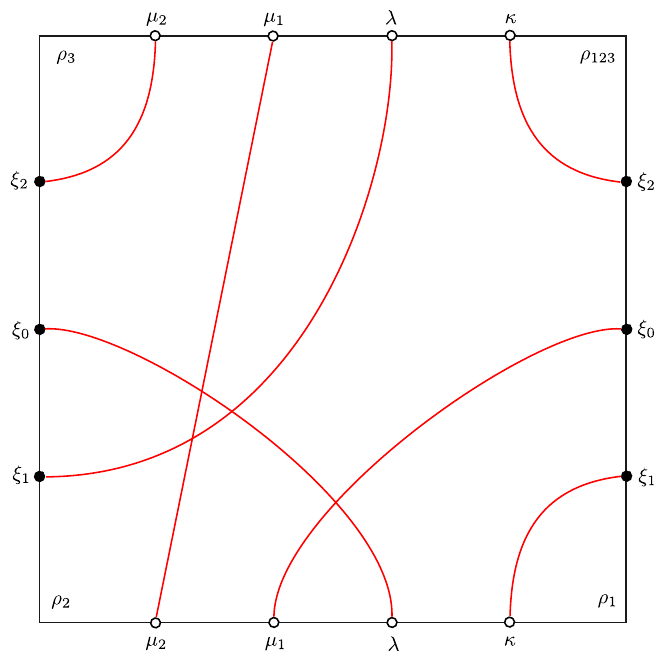}
    \caption{Immersed train track associated with the right-handed trefoil.}
    \label{fig:immersed train track}
\end{figure}

For an arbitrary knot $K$, we know how to compute $\widehat{CFD}(X_K)$ from Proposition \ref{thm algo} (also see Figure \ref{fig:unstablechaint>0}). Given the algorithm converting $\widehat{CFD}$ to immersed curves stated above, we may describe the shape of $\alpha_K$ when lifted to the universal cover. First, we describe the segment in $\alpha_K$
corresponding to the unstable chain, which we call the \textit{unstable segment} of $\alpha_K$.

\begin{lemma}\cite[Lemma 2.2]{Bod23} \label{unstable segment shape}
    Suppose that $K$ is a knot in $S^3$ and that $\gamma_0$ is the unstable segment of $\alpha_K$ lifted to the universal cover.
    \begin{itemize}
        \item If $\epsilon(K) = 1$ and $\tau(K) \geq 0$, $\gamma_0$ slopes upwards for $2\tau(K)$ rows and turns down at the top and up at the bottom.
        \item If $\epsilon(K) = -1$ and $\tau(K) \geq 0$, $\gamma_0$ slopes upwards for $2\tau(K)$ rows and turns up at the top and down at the bottom.
        \item If $\epsilon(K) = 1$ and $\tau(K) \leq 0$, $\gamma_0$ slopes downwards for $2\tau(K)$ rows and turns down at the bottom and up at the top.
        \item If $\epsilon(K) = -1$ and $\tau(K) \leq 0$, $\gamma_0$ slopes downwards for $2\tau(K)$ rows and turns up at the bottom and down at the top.
        \item If $\epsilon(K) = 0$, then $\tau(K) = 0$ and $\gamma_0$ is horizontal at height $0$.
    \end{itemize}
\end{lemma}

\begin{figure}[!htb]
    \centering
    \includegraphics[scale=0.5]{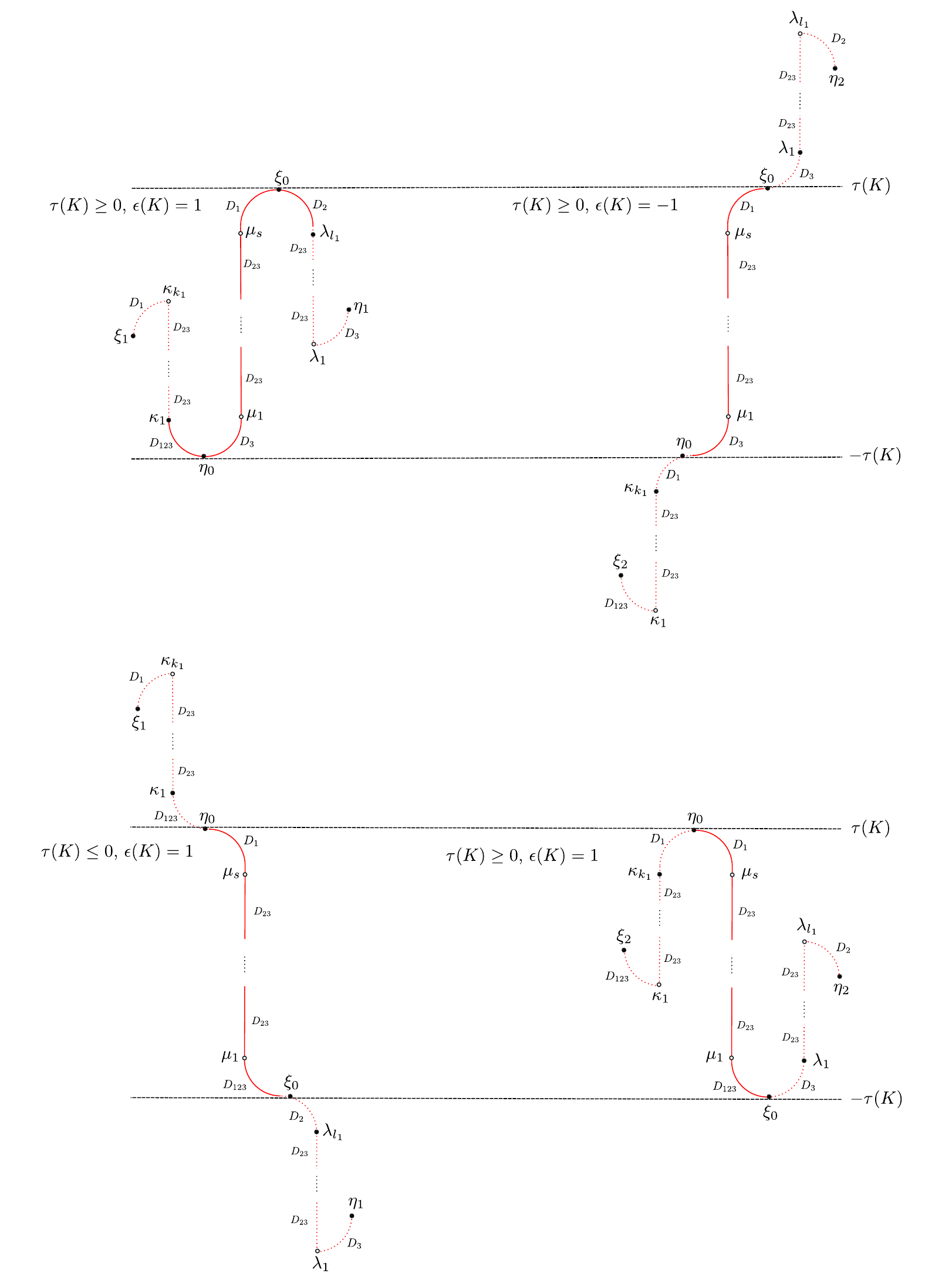}
    \caption{The unstable segment and its neighborhood in both directions in $\alpha_K$. The points $\xi_0$ and $\eta_0$ are at heights $\tau(K)$ and $-\tau(K)$ respectively when $\tau(K) \geq 0$, and the points $\xi_0$ and $\eta_0$ are at heights $-\tau(K)$ and $\tau(K)$ respectively when $\tau(K) \leq 0$. The length of the dotted neighborhood can be arbitrarily large or small}.
    \label{fig:immersedcfdall}
\end{figure}

We may extend the results of Lemma \ref{unstable segment shape} to describe a neighborhood of the unstable segment using the correspondence between elements of $\widehat{CFD}(X_K)$ and their immersed curve representations shown in Figure \ref{fig: CFD to train track}. This is displayed in Figure \ref{fig:immersedcfdall}. For example, the top figure shows the neighborhood of the unstable segment when $\tau(K)\ge 0$, corresponding to the unstable neighborhood in $\widehat{CFD}(X_K)$ appearing previously in Figure \ref{fig:unstablechaint>0}. The bottom figure displays the neighborhood of the unstable segment when $\tau(K)<0$.

\subsection{Pairing diagrams} \label{pairing diagrams}
In this section, we introduce the notion of a \textit{pairing diagram} for the satellite $P(K)$. Let $\alpha_K$ be the immersed curve associated to $K$ as previously described, and let $\beta_P$ be the $\beta$-curve in $\mathcal{H}$, the genus-one Heegaard diagram of $P$. Following the notation in \cite{CH23}, we denote the pairing diagram of $P(K)$ by $\mathcal{H}(\alpha_K)$. Intuitively, the pairing diagram is a way of laying $\beta_P$ over $\alpha_K$ in the torus $T^2=[0,1]^2/\sim$. Specifically, we divide $T^2$ into four quadrants, placing $\alpha_K$ into the first quadrant and $\beta_P$ (along with the $w,z$ basepoints) into the third quadrant. For the second and fourth quadrant, extend both curves horizontally and vertically. See Figure \ref{fig:pairing diagram} for an example of a pairing diagram of the Mazur pattern $Q_{2,1}$ with the right-handed trefoil as companion. The intersection points in the second quadrant correspond to the generators in $CFA^{-}(S^3, P)\boxtimes \widehat{CFD}(S^3, X_K)$ coming from the $\iota_0$ idempotent, and the intersection points in the fourth quadrant correspond to those generators  coming from the $\iota_1$ idempotent. 

\begin{figure}[!htb]
\centering
\begin{minipage}{0.5\textwidth}
  \centering
  \includegraphics{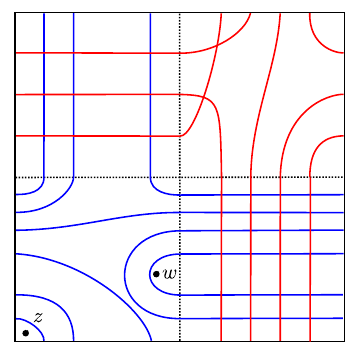}
  \caption{Pairing diagram of $Q_{2,1}(T_{2,3})$.}
  \label{fig:pairing diagram}
\end{minipage}%
\begin{minipage}{0.58\textwidth}
  \centering
  \includegraphics{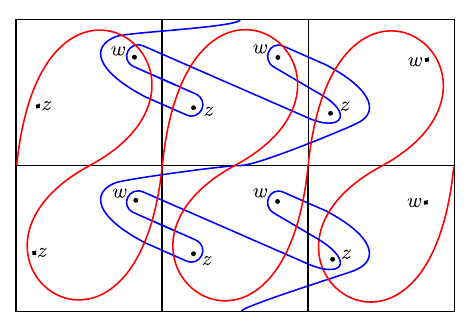}
  \caption{Pairing diagram of $Q_{2,1}(T_{2,3})$ in the universal cover after isotopy.}
  \label{fig:pairing lift}
\end{minipage}
\label{fig:test}
\end{figure}

In practice, we usually draw the pairing diagram $\mathcal{H}(\alpha_K)$ in the universal cover $\pi:\mathbb{R}^2\to \mathbb{R}^2/ \mathbb{Z}^2$, with a single lift of $\alpha_K$ and a single lift of $\beta_P$, which we denote by $\Tilde{\alpha}$ and $\Tilde{\beta}$, respectively. Intuitively, we can imagine the universal cover $\mathbb{R}^2$ as a board with a lattice of pegs nailed in at the $w$ and $z$ basepoints, and the curve $\Tilde{\alpha}$ as a rubber band that winds around those pegs. This curve is called a \textit{peg-board diagram}. When $\Tilde{\alpha}$ is pulled tight, we require that it intersects $\Tilde{\beta}$ transversally, and that every Whitney disk connecting two intersection points must contain at least a $w$ or a $z$ basepoint. Figure \ref{fig:pairing lift} shows the pairing diagram of $Q_{2,1}(T_{2,3})$ in the universal cover after isotopy, so that there are minimal intersection points between $\Tilde{\alpha}$ and $\Tilde{\beta}$.
   
Given a pairing diagram, we may recover the Alexander grading of intersection points representing generators of $gCFK^{-}(P(K))$, which we describe below. First, we can recover the relative Alexander grading with the following lemma. 
\begin{lemma}\cite[Lemma 4.1]{Che19} \label{relative filtration}
    Let $x,y$ be two intersection points of $\alpha$ and $\beta$. Let $l$ be the section of the $\beta$-curve from $x$ to $y$, and let $\delta_{w,z}$ be the straight arc from $w$ to $z$. Then 
    \[A(y)-A(x)=l\cdot\delta_{w,z},\]
    where multiplication on the right counts the algebraic intersection between $l$ and $\delta_{w,z}$.
\end{lemma}

To determine the absolute Alexander grading, note that the pairing diagram is symmetric under hyperelliptic involution. In other words, the complex remains the same if we rotate it by $\pi$ and exchange the $w$ and $z$ basepoints. The intersection point fixed under this involution must necessarily have Alexander grading 0. From this we may recover the absolute gradings of all intersection points using Lemma \ref{relative filtration}. In Example \ref{Q21+T23}, we use this method to compute the absolute gradings of certain intersection points in the pairing diagram of $Q_{2,1}(T_{2,3})$. 

\subsection{Computation of $\tau(P(K))$ and $\epsilon(P(K))$ from the pairing diagram} \label{isotopy explained}
From the pairing diagram $\mathcal{H}(\alpha_K)$ for the satellite $P(K)$, we may obtain $\tau(P(K))$ and $\epsilon(P(K))$ via a combinatorial computation similar to \cite{CH23}, which we describe here. First, we construct a bigraded chain complex $CFK_\mathcal{R}(\mathcal{H}(\alpha_K),\partial)$ over $\mathcal{R}$ by counting bigons combinatorially in the pairing diagram $\mathcal{H}(\alpha_K)$. Specifically, the generators of the complex are intersections of the $\tilde{\alpha}$ and $\tilde{\beta}$, and the differentials are given by bigons connecting two generators. Suppose there is a bigon $\phi$ bounded by $\tilde{\alpha}$ and $\tilde{\beta}$ intersecting acutely at $a$ and $b$, such that we encounter $\tilde{\alpha}$ first if we travel counterclockwise from $a$. The differential in $CFK_\mathcal{R}(\mathcal{H}(\alpha_K),\partial)$ corresponding to $\phi$ is 
\[\partial a = U^{n_w(\phi)}V^{n_z(\phi)}b,\]
where $n_w(\phi)$, $n_z(\phi)$ are the number of $w$ and $z$ basepoints in $\phi$, respectively. As we are working with $CFK_\mathcal{R}$, we recall that $UV=0$. See \cite[Figure 1]{CH23} for an example of the complex $CFK_\mathcal{R}$ of the $(2,1)$ cable of the right-handed trefoil $T_{2,3}$.

Recall the pairing theorem \cite[Theorem 11.19]{LOT18} which states  \begin{align*}    
gCFK^{-}(Y,K)&\simeq CFA^{-}(Y_1,K_1)\boxtimes \widehat{CFD}(Y_2).
\end{align*} In \cite{CH23}, Chen and Hanselman reinterpreted the pairing theorem in terms of the immersed curves in the case when $CFA^{-}(V, P)$ is paired with $\widehat{CFD}(X_K)$, thus obtaining $gCFK^{-1}(S^3, P(K))$. Here $P$ must be a $(1,1)$-pattern knot, as is the case for generalized Mazur patterns. The main theorem in \cite{CH23} states that there is a bigraded homotopy equivalence of chain complexes
\[CFK_\mathcal{R}(\mathcal{H}(\alpha_K),\mathcal{\partial})\cong CFK_\mathcal{R}(P(K)).\]
In other words, the pairing of $CFA^{-}(V,P)$ with $\widehat{CFD}(X_K)$ is entirely captured in the $CFK_\mathcal{R}$ complex of the pairing diagram. Hence, to recover the $\tau$ and $\epsilon$ invariants of $P(K)$, it suffices to compute the relevant portion of $CFK_\mathcal{R}(\mathcal{H}(\alpha_K),\mathcal{\partial})$.

First, we describe how to recover $\tau(P(K))$, using an algorithm in \cite{Che19}. Recall that the Alexander filtration on $\widehat{CFK}(K)$ induces a spectral sequence converging to $\widehat{HF}(S^3)=\mathbb{F}$, and that the $\tau$-invariant is the Alexander grading of the cycle that survives to the $E_\infty$ page. Passing from one page to the next in the spectral sequence amounts to eliminating the differentials that connect elements of minimal Alexander filtration difference. (For a more detailed discussion of why this is the case, see for example \cite{Bod23, Che19}.)
This can be visualized in the pairing diagram: the differentials correspond to Whitney disks that contain $z$ basepoints but not $w$ basepoints, and eliminating the differential amounts to isotoping the $\beta$-curve over such Whitney disks that connect intersection points with minimal filtration difference. This isotopy is illustrated in Figure \ref{fig: kill intersection}. Each isotopy removes a pair of intersection points, so that at the end of such isotopies, only one intersection point remains. This final intersection point represents the cycle that survives to the $E_\infty$ page, hence its Alexander grading is the $\tau$-invariant of the satellite knot. 

\begin{figure}[!htb]
    \centering
    \includegraphics[scale = 0.8]{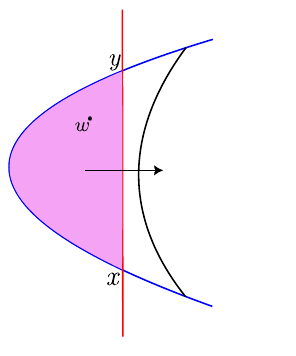}
    \caption{Eliminating a pair of intersection points $x$ and $y$ with minimal filtration difference. The Whitney disks connecting $x$ and $y$ is highlighted in pink. The small black arrow is the A-buoy placed along the $\beta$-curve.}
    \label{fig: kill intersection}
\end{figure}

In practice, we need to remember the filtration difference of the intersection points, and we do so by placing \textit{A-buoys} along the $\beta$-curve, which are small arrows introduced in \cite[Section 4]{Che19}. Figure \ref{fig: kill intersection} illustrates the placement of these A-buoys. Recall that from Lemma \ref{relative filtration} that the relative Alexander grading between two intersection points $x$ and $y$ is given by the algebraic intersection between the arc on the $\beta$-curve connecting $x$ and $y$ with $\delta_{w,z}$. When we perform the isotopies described above, this algebraic intersection may change. To correctly compute the filtration difference, we must count both the algebraic intersection of the corresponding arc on the $\beta$-curve with $\delta_{w,z}$ and the arc with the A-buoys. 

Note that this algorithm corresponds to doing successive pages of a spectral sequence, so we have to eliminate pairs of intersection points in order of their original filtration difference. In other words, we first eliminate all disks connecting points of filtration difference one. Once these are removed, we proceed to eliminate all disks connecting points of filtration difference two, and so on, making sure to place the A-buoys appropriately. 

To compute $\epsilon(P(K))$, we look for a subcomplex of $CFK_{\mathcal{R}}(\mathcal{H}(\alpha_K))$ that contains the cycle surviving to the $E_\infty$ page. This subcomplex must contain the distinguished element of some horizontally simplified basis, whose vertical situation determines the value of $\epsilon(P(K))$. In general, locating such a subcomplex within $CFK_\mathcal{R}(\mathcal{H}(\alpha_K))$ is challenging. However, performing isotopy across $z$ basepoints helps us track down one element in that subcomplex, making it computationally more tractable to recover the entire subcomplex.

After identifying the subcomplex, we simplify it by performing isotopies that reduce the number of elements as much as possible while ensuring that it is reduced. Differentials obtained from crossing over $z$ are called vertical differentials, and are labelled with powers of $V$ corresponding to the multiplicity of $z$. Similarly, differentials obtained from crossing over $w$ are called horizontal differentials, labelled with powers of $U$ corresponding to the multiplicity of $w$. In the horizontal complex consisting of the horizontal differentials, we look for a cycle which survives in the homology $\widehat{HF}(S^3)$. This is the distinguished element whose vertical situation determines $\epsilon(P(K))$, and whose Alexander grading is $\tau(P(K))$.
\begin{example} \label{Q21+T23}
We compute $\tau(Q_{2,1}(T_{2,3}))$ using the algorithm described above. The pairing diagram of the Mazur pattern and the right-handed trefoil is shown in Figure \ref{fig: Mazur+T_{2,3}}. 
\begin{figure}[!htb]
    \centering
    \includegraphics[scale=0.9]{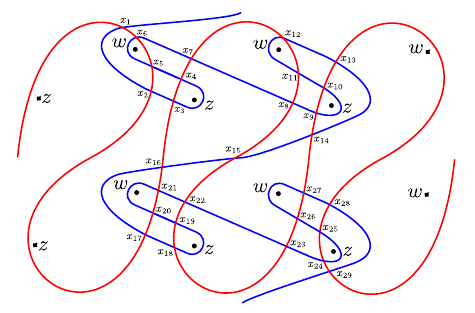}
    \caption{Pairing diagram of $Q_{2,1}(T_{2,3})$.}
    \label{fig: Mazur+T_{2,3}}
\end{figure}
First, we eliminate all pairs of intersection points of filtration difference 1, as shown in Figure \ref{fig: Mazur+T_{2,3}1}. After isotoping the right side of the $\beta$-curve to obtain Figure \ref{fig: Mazur+T_{2,3}2}, there are no more differentials in the new complex, and the only remaining intersection point is $x_3$. By the symmetry under elliptic involution, we see that $A(x_{15})=0$. Then using Lemma \ref{relative filtration}, we have $A(x_3)=-2$, therefore $\tau(Q_{2,1}(T_{2,3}))=-A(x_3) = 2$. This matches up with the bordered computation in \cite[Theorem 1.6]{Lev16}.
To compute $\epsilon(Q_{2,1}(T_{2,3}))$, we refer the reader to \cite{CH23}, where the differentials in the $CFK_{\mathcal{R}}$ complex are given. Note that the Mazur pattern in \cite{CH23} has the opposite orientation, which 
amounts to switching the $w$ and $z$ basepoint in the pairing diagram. The differentials in $CFK_{\mathcal{R}}$ are changed accordingly by switching the $U$ and $V$ powers. From there, one may compute the homology and find that the generator containing $x_3$ has an incoming vertical arrow. Therefore $\epsilon(Q_{2,1}(T_{2,3})) = 1$.

\begin{figure}[!htb]
    \centering
    \includegraphics[scale=0.9]{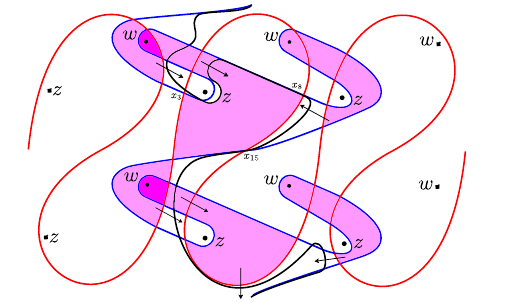}
    \caption{Eliminating pairs of intersection points of filtration difference 1. The Whitney disks connecting those pairs are highlighted in pink. After the isotopies, we get the black curve as our new $\beta$-curve.}
    \label{fig: Mazur+T_{2,3}1}
\end{figure}
\begin{figure}[!htb]
    \centering
    \includegraphics[scale=0.9]{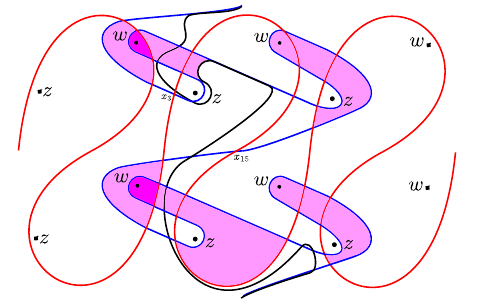}
    \caption{The point $x_3$ is the only intersection point remaining.}
    \label{fig: Mazur+T_{2,3}2}
\end{figure}

\end{example}

\section{Generalized Mazur patterns}\label{general mazur patterns}
In this section, we introduce the generalized Mazur patterns $Q_{m,n}$ and describe some of their properties, including presentations of their associated 2-bridge links using Schubert normal forms. We begin by describing how to construct Heegaard diagrams for $Q_{m,n}$, which will be utilized in our computation of the $\tau$- and $\epsilon$-invariant in Section \ref{IMMERSED COMPUTATION}.

\begin{proposition}
\label{h diagram}
Starting from a Heegaard diagram $H_{1,1}$ for $Q_{1,1}$, we may obtain a Heegaard diagram $H_{m,n}$ for $Q_{m,n}$ by an inductive procedure on $m$ and $n$.
\end{proposition}
We first describe the inductive construction of the Heegaard diagram $H_{m,n}$, and then for the rest of this section, we justify how it recovers $Q_{m,n}$. The inductive procedure begins by constructing a Heegaard diagram $H_{1,1}$ for $Q_{1,1}$, which is shown in Figure \ref{fig:Q11}. Then it constructs $H_{m,1}$ and finally $H_{m,n}$. Before we describe the inductive step, however, we need to introduce some terminologies. 

Recall that a Heegaard diagram consists of left and right vertical $\alpha_1$ arcs, top and bottom $\alpha_2$ arcs, and a $\beta$ curve which crosses the $\alpha_i$ arcs, breaking it into strands inside the diagram. We label the strands as in Figure \ref{fig:Q11} and \ref{fig:construction-terminology}, based on where they are located in the diagram. Specifically, the strands that intersect both the $\alpha_2$ and $\alpha_1$ arcs are
\begin{itemize}
    \item $T/L$, which intersects the top $\alpha_2$ arc and the left $\alpha_1$ arc
    \item $T/R$, which intersects the top $\alpha_2$ arc and the right $\alpha_1$ arc
    \item $B/R$, which intersects the bottom $\alpha_2$ arc and the right $\alpha_1$ arc
    \item $B/L$, which intersects the bottom $\alpha_2$ arc and the left $\alpha_1$ arc, enclosing a region that doesn't contain the $w$ basepoint
    \item $B/L/W$, which intersects the bottom $\alpha_2$ arc and the left $\alpha_1$ arc, enclosing a region that contains the $w$ basepoint
    
\end{itemize}
There are also strands which only intersect the $\alpha_2$ arcs:
\begin{itemize}
    \item $B/M$, which intersects the bottom $\alpha_2$ arc twice in a rainbow shape. The enclosed region between this strand and the bottom $\alpha_2$ arc contains the $w$ basepoint
    \item $V/L$ and $V/R$, which are vertical strands that intersect both the top and bottom of the $\alpha_2$ arcs. These are placed in between $T/L$ and $T/R$, as well as between $B/L$ and $B/R$. The final label $L$ or $R$ determine whether they are to the left or right of the rainbow $B/M$
\end{itemize}
In this terminology, the Heegaard diagram $H_{1,1}$ consists two strands of type $T/L$, two strands of type $T/R$, one strand of type $B/L/W$, one strand of type $B/M$, and one strand of type $B/R$.

We now describe the algorithm to construct $H_{m,n}$ for any $m$ and $n$, which consists of three steps. In each step, we add strands that are a part of the $\beta$ curve. 

\underline{\textit{Step 1:}} Construct $H_{m,1}$ from $H_{1,1}$ by adding:
\begin{itemize} 
    \item $2(m-1)$ strands of type $T/L$
    \item $m-1$ strands of type $T/R$
    \item $m-1$ strands of type $B/R$
    \item $m-1$ strands of type $V/R$
    \item $m-1$ strands of type $B/M$
\end{itemize}
Using the above procedure, for example, we may obtain the Heegaard diagram $H_{2,1}$ in Figure \ref{fig:mazur_diagram} for the Mazur pattern from $H_{1,1}$, shown in Figure \ref{fig:Q11}. Figure \ref{fig:Qm1} provides a schematic for the Heegaard diagram for $H_{m,1}$, where $m$ is any positive integer.

\underline{\textit{Step 2:}} Construct $H_{m,2}$ from $H_{m,1}$ by first taking the strand of type $B/L/W$ (there should be only one at this stage) and push it down through $\rho_3$ to form a rainbow $B/M$. There should now be $m+1$ strands of type $B/M$ in the diagram. Then, we add:
\begin{itemize}
    \item 1 strand of type $T/L$
    \item 1 strand of type $V/R$
    \item $2m$ strands of type $V/L$
\end{itemize}
A schematic for this procedure is shown in Figure \ref{fig:Qmn}, which uses $m = 1$ for simplicity, and omits all parts of the diagram except for the one $B/M$ rainbow present in $H_{1,1}$.

\underline{\textit{Step 3:}} Construct $H_{m,n}$ from $H_{m,2}$ by adding:
\begin{itemize}
    \item $(2m+1)(n-2)$ strands of type $V/L$
    \item $n-2$ strands of type $V/R$
\end{itemize}
Figure \ref{fig:Qmn_diagram} provides a schematic for the general Heegaard diagram $H_{m,n}$ where $n\ge 2$.

To verify that this construction indeed recovers the generalized Mazur pattern, one must show that the pattern $Q_{m,n}'$ obtained from the Heegaard diagram $H_{m,n}$ is isotopic to $Q_{m,n}$. This is not immediately obvious, but well-chosen isotopies will reduce the number of self intersections until the same ones as described in Definition \ref{def gmp} remain. This results in $2m+1$ intersection points down the vertical alpha curve and $2m+2n+2nm-2$ intersection points along the horizontal alpha curve. To circumvent the technical details of isotoping Heegaard diagrams, we instead rely on two-bridge link invariants which we can extract from these Heegaard diagrams. This will be explored in the rest of this section.

\begin{figure}[!htb]
    \centering
    \includegraphics[scale =0.6]{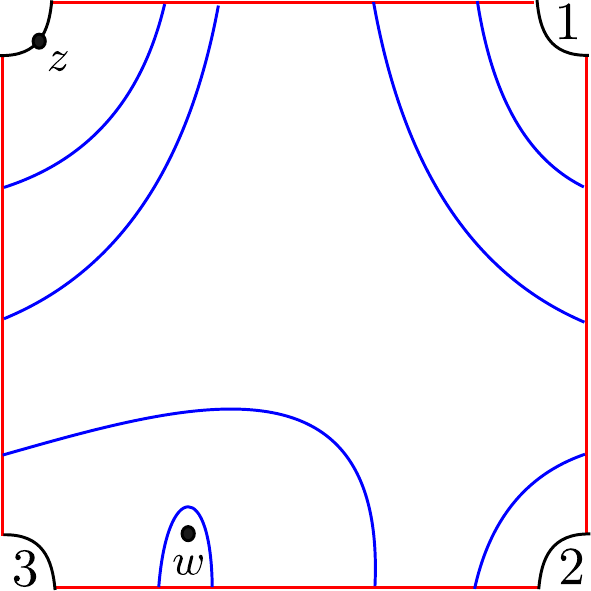}
    \caption{The Heegaard diagram $H_{1,1}$, consisting of 2 strands of type $T/L$, 2 strands of type $T/R$, 1 strand of type $B/L/W$, 1 strand of type $B/M$, and 1 strand of type $B/R$.}
    \label{fig:Q11}
\end{figure}

\begin{figure}[!htb]
    \centering
    \includegraphics[scale=0.6]{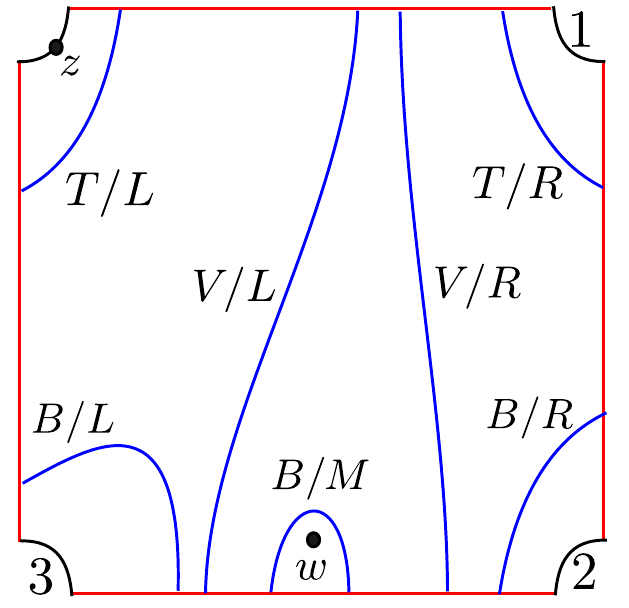}
    \caption{Terminology for describing the Heegaard diagram $H_{m,n}$ when $n\ge 2$. The labels refer to the strands, not the regions.
    }
    \label{fig:construction-terminology}
\end{figure}

\begin{figure}[!htb]
    \centering
    \includegraphics[scale = 0.45]{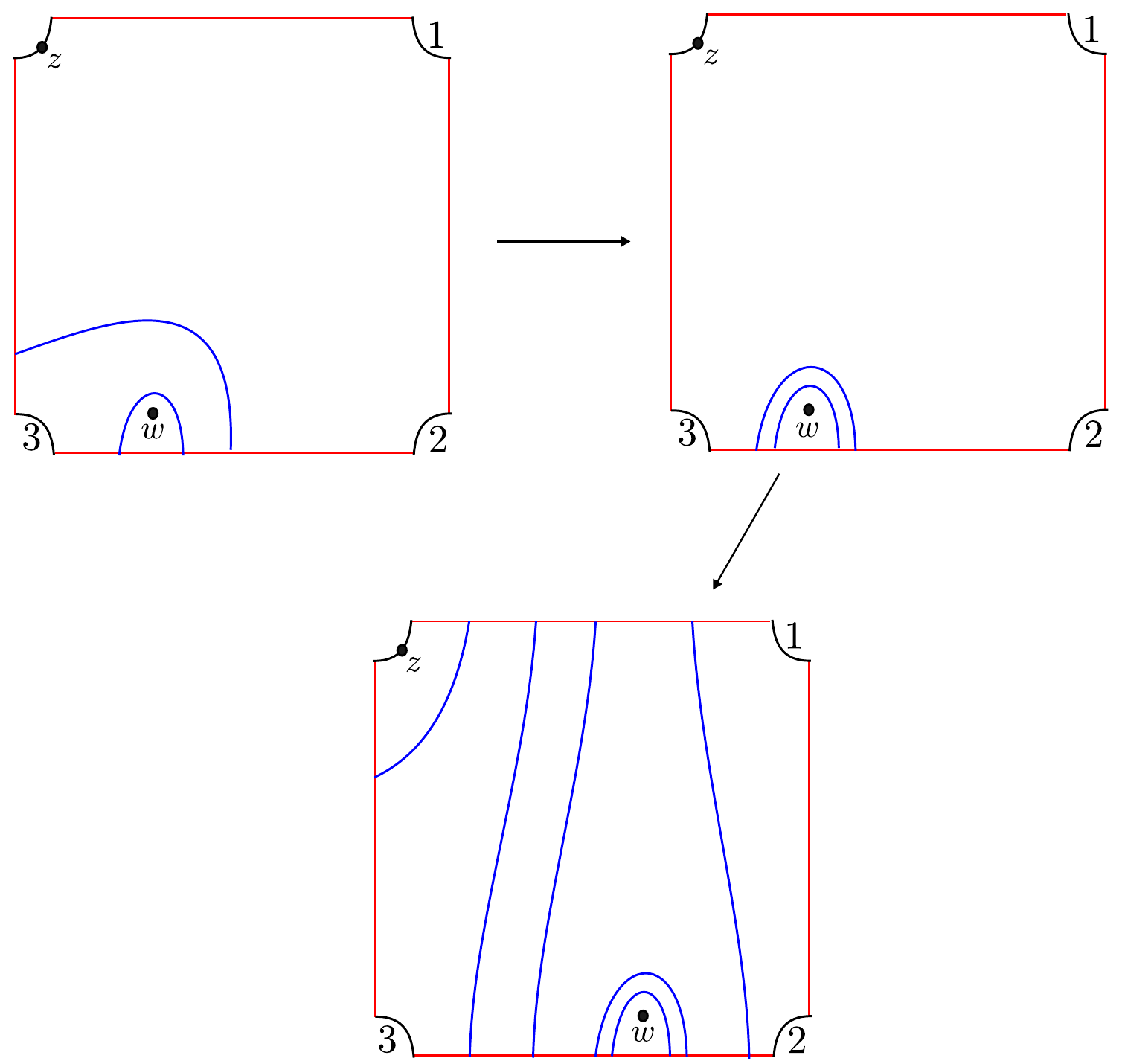}
    \caption{A schematic for Step 2. The first arrow represents pushing the strand of type $B/L$ down through $\rho_3$ to form a rainbow $B/M$. The second arrow represents adding the rest of the strands, namely, 1 strand of type $T/L$, 1 strand of type $V/R$, and $2m$ strands of type $V/L$ ($m=1$ in this case for simplicity).}
    \label{fig:Qmn}
\end{figure}

\begin{figure}[!htb]
\begin{minipage}[b]{0.48\textwidth}
    \centering
    \includegraphics[scale=0.6]{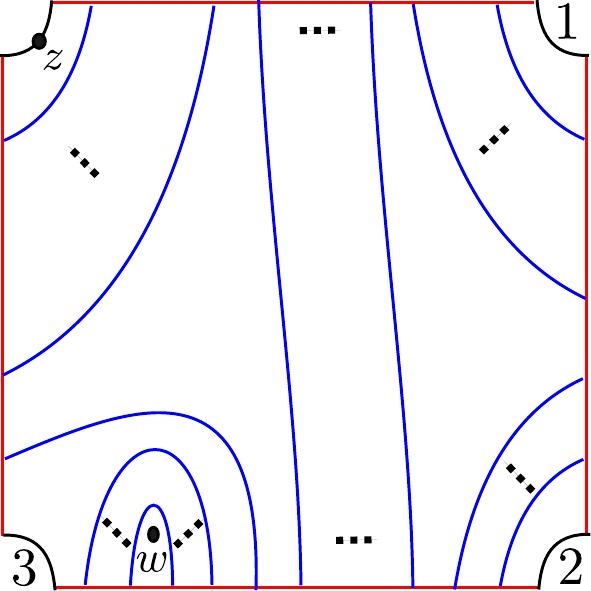}
    \caption{Schematic for the Heegaard diagram $H_{m,1}$.}
    \label{fig:Qm1}
\end{minipage}
\begin{minipage}[b]{0.5\textwidth}
    \centering
    \includegraphics[scale=0.6]{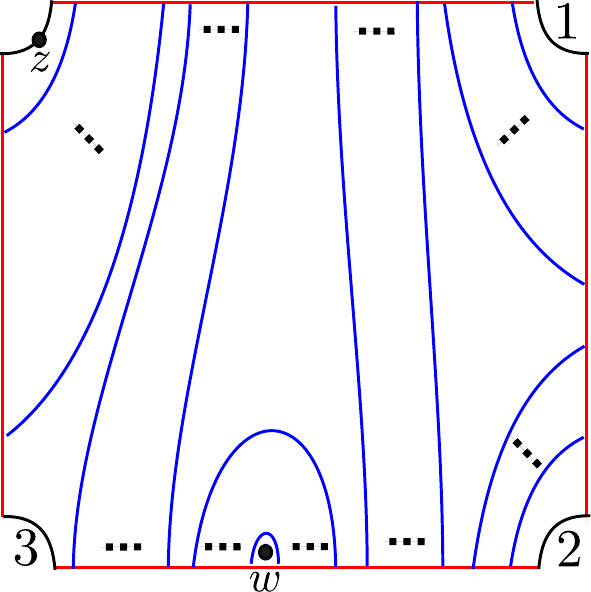}
    \caption{Schematic for $H_{m,n}$ when $n\ge 2$.}
    \label{fig:Qmn_diagram}
\end{minipage}
\end{figure}

By \cite{Che19}, there is an identification between $2$-bridge links and $(1,1)$-unknot patterns as follows: removing a regular neighborhood of one component of the $2$-bridge link leaves the other component as a pattern knot inside the solid torus.  Moreover, every $2$-bridge link admits an presentation called the \textit{Schubert normal form} \cite{Sch53}. Such a normal form is denoted by $b(p,q)$, where $(p,q)$ is a pair of coprime integers such that $p>0$ and $0<|q|<\frac{p}{2}$. The link $b(p,q)$ has two components if and only if $p$ is even, and $b(p,q)$ is isotopic to $b(p',q')$ if and only if $p = p'$ and $q' \equiv q^{\pm1}$ (mod $p$).

In order to prove that the pattern $Q_{m,n}'$ obtained from Proposition \ref{h diagram} is isotopic to $Q_{m,n}$, we verify that they are identified to isotopic $2$-bridge links. First, we determine the $2$-bridge link identified with $Q_{m,n}'$, following the procedure in \cite{Che19}. 

Let $\mathcal{H}$ be a bordered Heegaard diagram that gives rise to unknot patterns. One may associate $\mathcal{H}$ with a pair of non-negative integer $(r,s)$, where $r$ is the number of loops around $w$ and $z$, and $s$ is the number of middle strands which separates the loops the main region of the diagram, as defined in \cite{Che19}. An example of finding $r$ and $s$ for the Heegaard diagram of the Mazur pattern is shown in Figure \ref{fig:loopsstrands}. The parameter $(r,s)$ completely determines the entire diagram. The following theorem gives a general formula for the $2$-bridge link associated with a $(1,1)$-unknot pattern in terms of $(r,s)$.

\begin{figure}[!htb]
    \centering
    \includegraphics[scale=0.3]{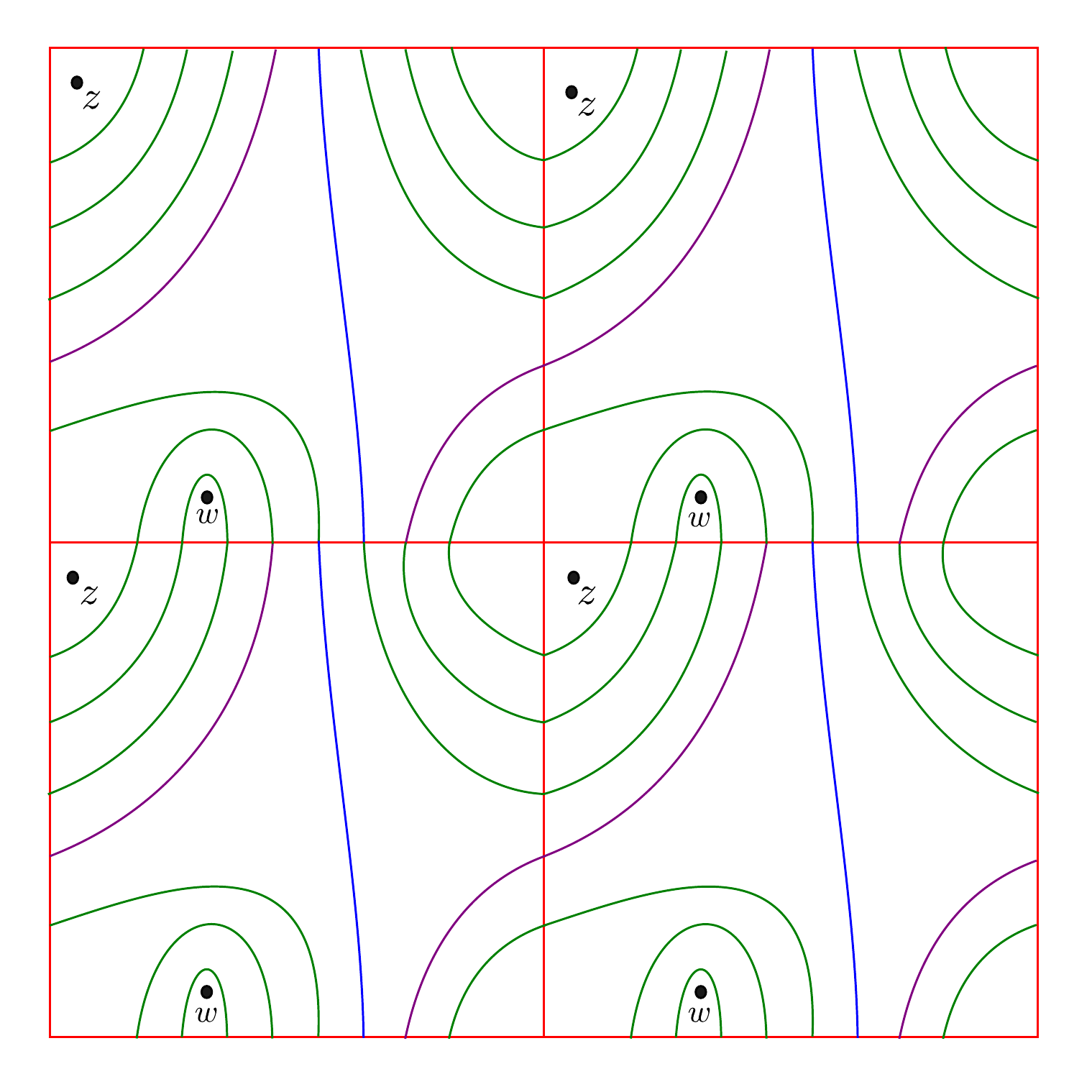}
    \caption{Determining the number of loops and strands for the Heegaard diagram of the Mazur pattern. Loops are colored in green, and strands are colored in purple, so $(r,s) = (3,1)$.}
    \label{fig:loopsstrands}
\end{figure}

\begin{theorem}\cite[Theorem 5.4]{Che19} \label{wenzhaos thm}
    Let $P$ be a $(1,1)$-unknot pattern obtained by a genus-one doubly pointed Heegaard diagram of parameter $(r,s)$. Then the link consists of $P$ and the meridian of the solid torus is the 2-bridge link $b(2|s| + 4|r|, \epsilon(r)(2|r|-1))$, where $\epsilon(r)$ is the sign of r.
\end{theorem}

For the Heegaard diagram we constructed in Proposition \ref{h diagram}, we have $(r,s) = (m+1,-(2mn+n-m-2))$, where the minus comes from sign conventions. By Theorem \ref{wenzhaos thm}, the $2$-bridge link for $Q_{m,n}'$ is $b(4mn+2n+2m,2m+1)$.

Next, we determine the $2$-bridge link for the generalized Mazur patterns $Q_{m,n}$. In this case, it is useful to recall the identification between Schubert normal form and  Conway rational tangles (for a review on rational tangles, see \cite{Lic97}): $b(p,q)$ is equivalent to the rational tangle $C(a_1,...,a_n)$, where $a_i$ is the $i$th coefficient in the continued fraction expansion

\begin{equation} \label{fraction equation}
    \dfrac{q}{p} = \dfrac{1}{a_1 + \dfrac{1}{a_2 + \cdots \dfrac{1}{a_{n-1} + \dfrac{1}{a_n}}}}.
\end{equation}
For example, the rational tangle $C(2,1,2)$ is equivalent to $b(8,3)$, which is the Whitehead link \cite{Ord06}.
The following proposition uses this identification to determine the 2-bridge link for $Q_{m,n}$.
\begin{figure}[!htb]
    \centering
    \includegraphics[scale=0.7]{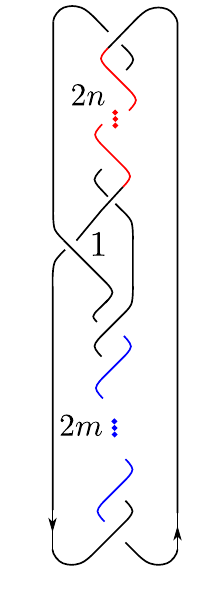}
    \caption{The Conway rational tangle $C(2m,1,2n)$, which corresponds to the $2$-bridge link associated with $Q_{m,n}$.}
    \label{fig:tanglediagram}
\end{figure}

\begin{proposition} \label{braid decomp}
    The $2$-bridge link associated with $Q_{m,n}$ has Schubert normal form $b(4mn+2n+2m,2m+1)$. 
\end{proposition}
\begin{proof}
We claim that the $2$-bridge link associated with $Q_{m,n}$ is the rational tangle $C(2n,1,2m)$ shown in Figure \ref{fig:tanglediagram}. If this is true, this $2$-bridge link is exactly $b(4mn+2n+2m,2m+1)$, since \[\frac{1}{2n+\frac{1}{1+\frac{1}{2m}}} = \frac{2m+1}{4mn+2n+2m}.\]
We induct on $n$ and $m$. If $n=m=1$, the tangle $C(2,1,2)$ is identified with Whitehead link $Q_{1,1}$, as discussed above. Suppose that $C(2n,1,2m)$ is the 2-bridge link for $Q_{m,n}$. 
Note that $C(2n, 1,2m+2)$ adds a full blue twist to $C(2n,1,2m)$, which corresponds to adding one longitudinal arc in the set of $m$ strands in $Q_{m,n}$, so it is associated with the pattern $Q_{m+1,n}$. Similarly, the tangle $C(2n+2, 1,2m)$ adds a full red twist to $C(2n,1,2m)$, which corresponds to adding one longitudinal arc in the set of $n$ strands in $Q_{m,n}$, so it is associated with $Q_{m,n+1}$. This completes the induction step.
\end{proof}

\begin{remark}
    The $2$-bridge link for the Mazur pattern in \cite[Figure 4]{Lev16} has Conway form $C(4,2,-2)$, which is different than $C(4,1,2)$, the one that we exhibit in Proposition \ref{braid decomp}. However, the $2$-bridge links represented by the two Conway forms are isotopic -- this can be verified either directly or through Equation \ref{fraction equation}. For the patterns $Q_{m,n}$, a similar relation holds: the tangle $C(2m,1,2n)$ is isotopic to $C(2m,2,-1,-(2n-1))$. Specializing to the patterns $Q_{m,1}$, we obtain that $C(2m,2,-1,-1)$ is isotopic to $C(2m,2,-2)$, indicating that the bridge presentation for these patterns can be described entirely using full twists. 
\end{remark}
\begin{proof}[Proof of Proposition \ref{h diagram}]
By Theorem \ref{wenzhaos thm} and Proposition \ref{braid decomp}, we see that the $2$-bridge links for $Q_{m,n}$ and $Q_{m,n}'$ both have Schubert normal form $b(4mn+2n+2m,2m+1)$. The patterns $Q_{m,n}$ and $Q_{m,n}'$ are therefore isotopic.
\end{proof}
As an immediate consequence of this discussion, we also obtain Proposition \ref{qmn isotopic}, which states that $Q_{m,n}$ is isotopic to $Q_{n,m}$.
\begin{proof}[Proof of Proposition \ref{qmn isotopic}]
    By Proposition \ref{braid decomp}, the 2-bridge links for $Q_{m,n}$ and $Q_{n,m}$ are $b(4mn+2n+2m,2m+1)$ and $b(4mn+2n+2m,2n+1)$, respectively. Since $(2m+1)(2n+1) \equiv 1$ (mod $4mn+2n+2m$), the two links are isotopic, and thus $Q_{m,n}$ is isotopic to $Q_{n,m}$ as desired.
\end{proof}

\begin{remark}
    In \cite{Hom14a}, Hom found bordered Heegaard diagrams for the $(p,1)$-cable patterns in the solid torus. These diagrams have parameters $(r,s) = (1,-(p-2))$, which implies that $(p,1)$-cables are exactly $Q_{0,p} = Q_{p,0}$ in our notation.
\end{remark}

\begin{remark}
    In this section, the generalized Mazur patterns $Q_{m,n}$ are not oriented. In the following sections, we orient them using the Heegaard diagrams in Proposition \ref{h diagram}, giving them a winding number of $-(m-n)$. With these orientations, Proposition \ref{qmn isotopic} states that $Q_{m,n}$ is isotopic to $rQ_{n,m}$, the reverse of $Q_{n,m}$. However, since neither $\tau$ nor $\epsilon$ depend on orientations,  we can assume without loss of generality that $m\ge n$ when computing $\tau(Q_{m,n}(K))$ and $\epsilon(Q_{m,n}(K))$.
\end{remark}

\section{Computation of $\tau(Q_{m,n}(K))$ and $\epsilon(Q_{m,n}(K))$} \label{IMMERSED COMPUTATION}
In this section, we compute the invariants $\tau(Q_{m,n}(K))$ and $\epsilon(Q_{m,n}(K))$ for generalized Mazur patterns, giving a proof of Theorem \ref{main thm}. This is the technical core of our paper. By Proposition \ref{qmn isotopic}, we may reduce our computation to $Q_{m,n}$ with $m \geq n$. 

\begin{theorem}\label{eps thm}
Let $Q_{m,n}$ be the generalized Mazur pattern embedded in the solid torus $V$. If $m\ne n$, then for any knot $K\subset S^3$, we have \begin{equation}\label{eps eq}
    \tau(Q_{m,n}(K)) = \begin{cases}
    |m-n|\tau(K) & \text{if } \tau(K) \le 0 \text{ and } \epsilon(K) \in \{0,1\}, \\
    |m-n|\tau(K) + |m-n| & \text{if } \tau(K) < 0 \text{ and } \epsilon(K)=-1, \\
    |m-n|\tau(K) + \min(m,n) & \text{if } \tau(K) > 0 \text{ and }\epsilon(K) = 1,\\
    |m-n|\tau(K) + \max(m,n) -1& \text{if } \tau(K) \ge 0 \text{ and }\epsilon(K) = -1.
    \end{cases}
    \end{equation}
In the case where $m=n$, we have 
\begin{equation}\label{eps eq 3}
    \tau(Q_{m,m}(K)) = \begin{cases}
    0 & \text{if } \tau(K) < 0, \\
    m-1 & \text{if } \tau(K) = 0, \\
    m & \text{if } \tau(K) > 0.
    \end{cases}
    \end{equation}
Also,
\begin{equation}\label{epsilon}
    \epsilon(Q_{m,n}(K)) = \begin{cases}
        0 & \text{if } \tau(K) = \epsilon(K) = 0, \\
        1 & \text{otherwise.}
    \end{cases}
\end{equation}
\end{theorem}

\begin{figure}[!htb]
    \centering
    \includegraphics[scale=0.8]{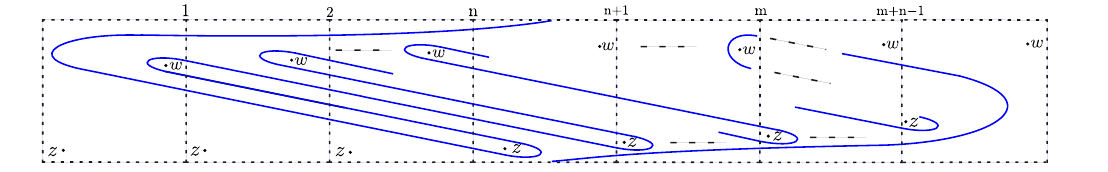}
    \caption{A lift of $\beta$ curve to the universal cover.}
    \label{fig:genbeta}
\end{figure}

\begin{figure}[!htb]
    \centering
    \includegraphics[scale=0.8]{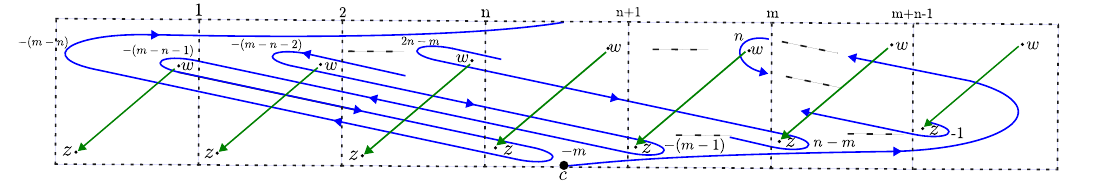}
    \caption{The Alexander grading at various points of $\tilde{\beta}$ curve in the row containing $c$.}
    \label{fig: tau by direction}
\end{figure}

We follow the strategy outlined in Section \ref{intro immersed}, recovering the $\tau$ and $\epsilon$ invariants of $Q_{m,n}(K)$ from its pairing diagram. Using the notation from Section \ref{pairing diagrams}, we denote the lifts of $\alpha_K$ and $\beta(Q_{m,n})$ in the universal cover of the doubly marked torus as $\tilde{\alpha}$ and $\tilde{\beta}$, respectively.

\begin{proof}[Proof of Theorem \ref{eps thm}]
When $\epsilon(K)$ = 0, the knot $K$ is $\epsilon$-equivalent to the unknot $U$ \cite{Hom14b}. It follows that $\tau(Q_{m,n}(K)) = \tau(Q_{m,n}(U))$ and $\epsilon(Q_{m,n}(K)) = \epsilon(Q_{m,n}(U))$, so we have $\tau(Q_{m,n}(K))= 0$ and $\epsilon(Q_{m,n}(K))= 0$, as desired. For the rest of the proof, we assume that $\epsilon(K)\ne 0$ and $m\ge n$.

We begin by a discussion of how to construct the pairing diagram of the satellite $Q_{m,n}(K)$, which consists the two lifts $\tilde{\beta}$ and $\tilde{\alpha}$, where the latter is pulled tight in the pegboard diagram. Recall that Proposition \ref{h diagram} gives us a way to construct the Heegaard diagram $\mathcal{H}$ of $Q_{m,n}$. We obtain 
$\tilde{\beta}$ by lifting the $\beta$-curve in $\mathcal{H}$ into the universal cover, as shown in Figure \ref{fig:genbeta}. The lift spans $m+n$ columns in the universal cover, which we label by $1$ to $m+n$ from left to right. For the lift $\tilde{\alpha}$ of $\alpha_K$, Lemma \ref{unstable segment shape} describes the unstable segment of $\tilde{\alpha}$ and Figure \ref{fig:immersedcfdall} shows its relevant neighborhood. Note that each column contains a translated copy of $\tilde{\alpha}$, but we only draw the copy contained in columns $n$ to $n+1$, because this is the only relevant one containing the distinguished generator in $\widehat{HF}(S^3)$. For example, the pairing diagram of $Q_{m,n}(K)$ in the case when $\tau(K)>0$ and $\epsilon(K)=1$ is shown in Figure \ref{fig:pairing tau>0 eps=1}.

From the pairing diagram, we may determine the absolute Alexander grading of the intersection points by Lemma \ref{relative filtration} and the discussion thereafter. In particular, note that if $x'$ can be obtained from $x$ by shifting the picture in the universal cover down a row, then $A(x)-A(x')=w(Q_{m,n})=-(m-n)$. Let $c$ be the fixed point under hyperelliptic involution, with $A(c)=0$. For ease of computation, we determine the grading of $x$ by first shifting the picture vertically so that the image $x'$ lies in the row containing $c$, and then determining the grading of $x'$ using Lemma \ref{relative filtration}. In Figure \ref{fig: tau by direction}, we display the grading of different points along $\tilde{\beta}$ in the row containing $c$.

We are now ready to recover $\tau(Q_{m,n}(K))$ and $\epsilon(Q_{m,n}(K))$. The lift $\tilde{\alpha}$ has the form described in Lemma \ref{unstable segment shape} and Figure \ref{fig:immersedcfdall}, which depends on $\tau(K)$ and $\epsilon(K)$, so it remains to inspect the following four cases.

\begin{figure}[!htb]
    \centering
    \includegraphics[scale = 0.8]{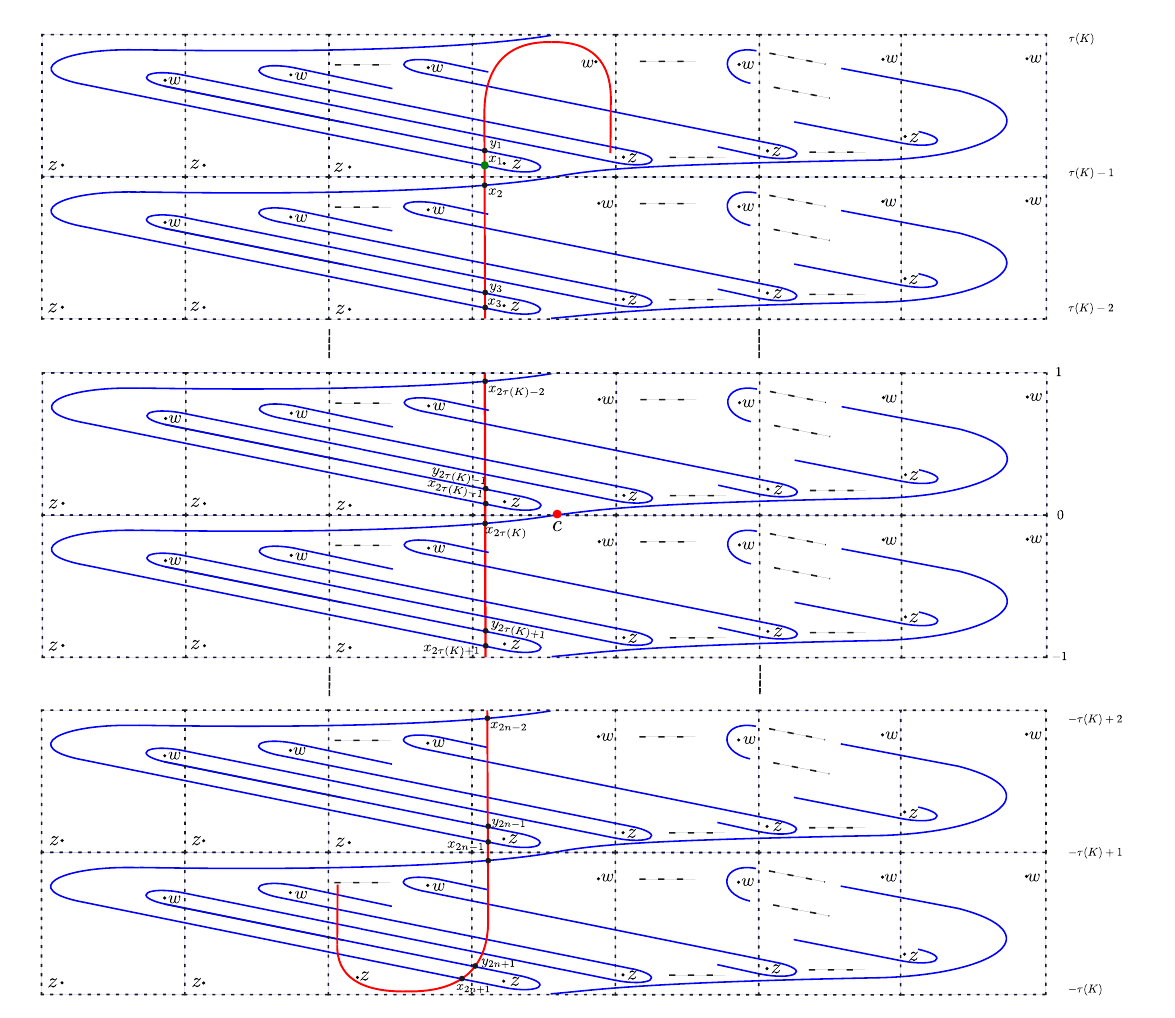}
    \caption{The pairing diagram for $Q_{m,n}(K)$ when $\tau(K) > 0, \epsilon(K) = 1$.}
    \label{fig:pairing tau>0 eps=1}
\end{figure}

\begin{figure}[!htb]
\[\begin{tikzcd}[sep = small]
	&&&&& {y_{2n+1}} \\
	&&&& {y_{2n-1}} & {x_{2n+1}} \\
	&&&& {x_{2n-1}} && {x_{2n}} \\
	&&& \textcolor{white}{\bullet} && {x_{2n-2}} \\
	& {y_3} && \dots \\
	{y_1} & {x_3} && \textcolor{white}{\bullet} \\
	{x_1} && {x_2}
	\arrow["{U^m}", from=7-1, to=7-3]
	\arrow["V"', from=6-1, to=7-1]
	\arrow["{U^n}", from=6-2, to=7-3]
	\arrow["V"', from=5-2, to=6-2]
	\arrow["{U^m}", from=6-2, to=6-4]
	\arrow["{U^m}", from=4-4, to=4-6]
	\arrow["{U^n}", from=3-5, to=4-6]
	\arrow["V"', from=2-5, to=3-5]
	\arrow["{U^m}", from=3-5, to=3-7]
	\arrow["{U^n}", from=2-6, to=3-7]
	\arrow["V"', from=1-6, to=2-6]
\end{tikzcd}\]
\caption{The relevant differentials in the $CFK_\mathcal{R}$ complex of $Q_{m,n}(K)$ when $\tau(K) > 0$, $\epsilon(K) = 1$.}
\label{fig:taug0eps1diff}
\end{figure}
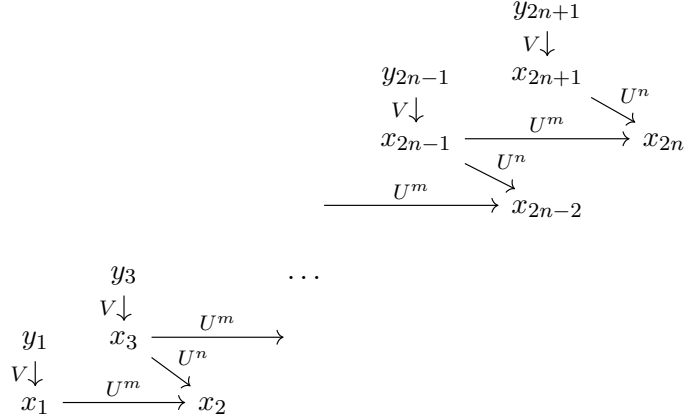

\begin{figure}[!htb]
    \centering
    \includegraphics[scale=0.9]{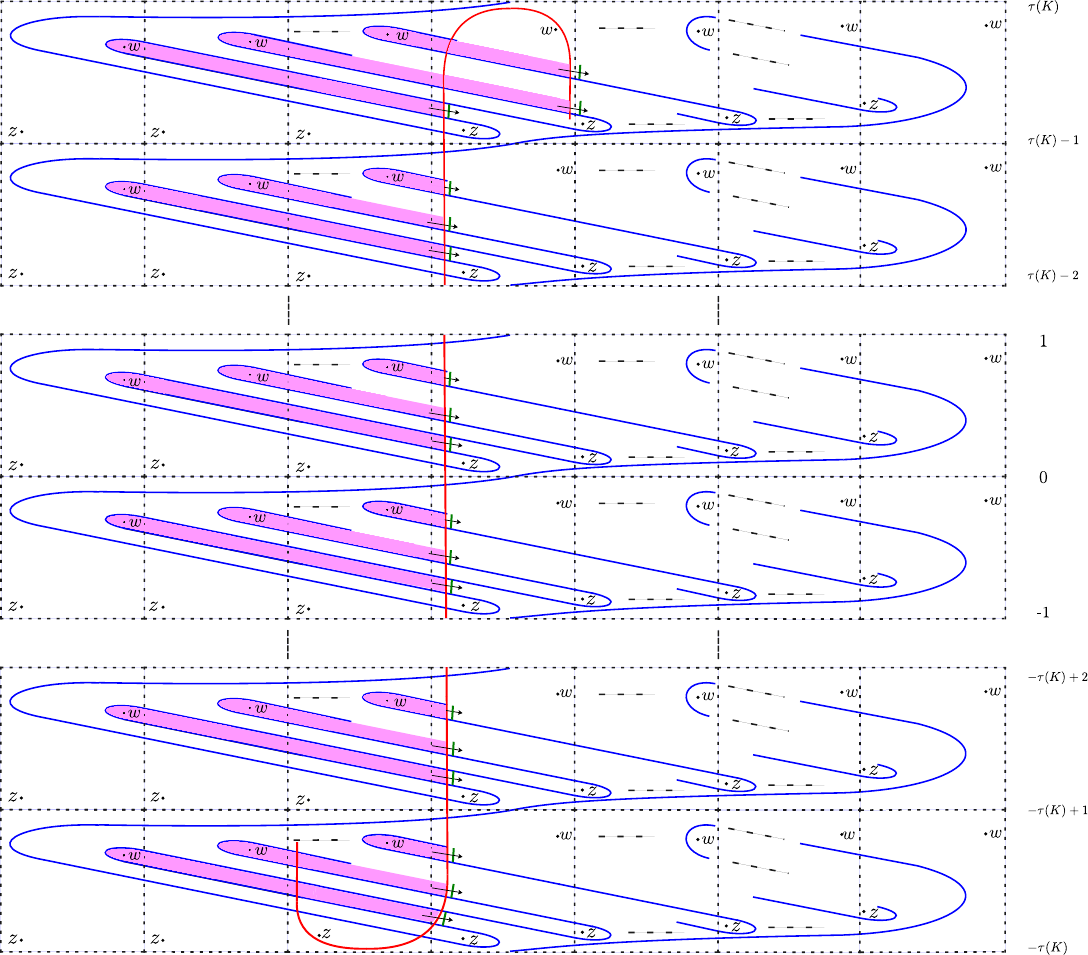}
    \caption{The disks highlighted in pink represent all the Whitney disks  of filtration difference $1$. Eliminating them by an isotopy, we get Figure \ref{fig:zbasepointisotopy2}.}
    \label{fig:zbasepointisotopy1}
\end{figure}

\underline{\textit{Case 1: $\tau(K) > 0$ and $\epsilon(K) = 1$}}. The pairing diagram of $Q_{m,n}(K)$ is shown in Figure \ref{fig:pairing tau>0 eps=1}, where the fixed point $c$ is marked in red. Following the strategy in Section \ref{isotopy explained}, we first eliminate all the intersection points of filtration difference 1. The Whitney disks connecting those points are highlighted in pink in Figure \ref{fig:zbasepointisotopy1}, and the  diagram after they are eliminated is shown in Figure \ref{fig:zbasepointisotopy2}. If $m>n$, we eliminate the blue disks (with filtration difference $n-1$), then the brown disks (with difference $n$), and finally the yellow disks (with difference $m-1$). If $m=n$, we eliminate the blue and yellow disks first, then the brown disks. Either way, $x_1$ is the remaining intersection point after the isotopies. 

\begin{figure}[!htb]
    \centering
    \includegraphics[scale=0.9]{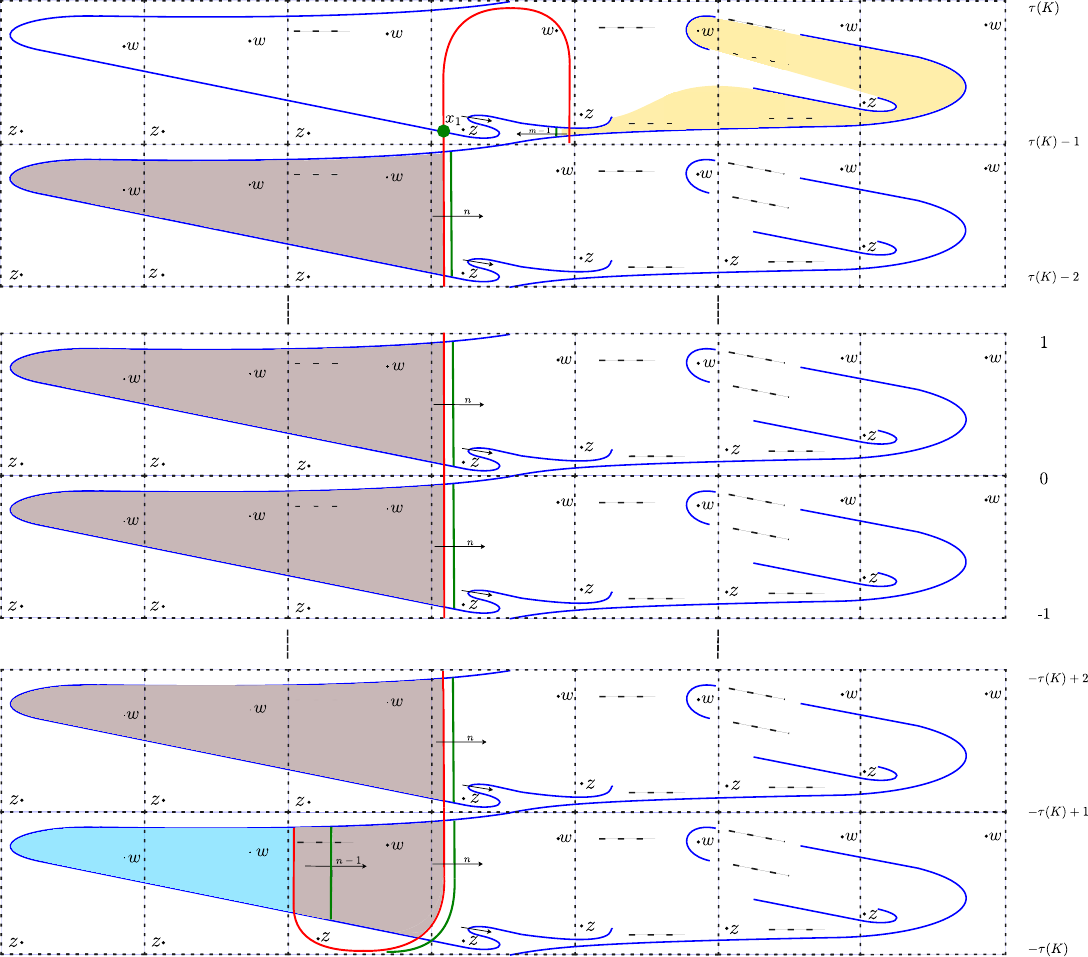}
    \caption{The result of eliminating Whitney disks in Figure \ref{fig:zbasepointisotopy1}. If $m>n$, we first eliminate the blue disks, then the brown disks, and finally the yellow disks. If $m=n$, we first eliminate the blue and yellow disks, then the brown disks. In both cases, we arrive at a complex with $x_1$ as the only remaining intersection point. The small arrows with label $n$ or $n-1$ represent respectively $n$ or $n-1$ small arrows placed in parallel; these are the A-buoys keeping track of the filtration difference.}
    \label{fig:zbasepointisotopy2}
\end{figure}

To recover the Alexander grading of $x_1$, we shift the diagram down $\tau(K)-1$ rows to get $x_1'$, and find that the algebraic intersection $l_{c,x_1'}\cdot \delta_{w,z}= -m$. It follows that
\[\tau(Q_{m,n}(K))= -A(x_1) = -(-(m-n)(\tau(K)-1) + l_{c,x_1'} \cdot \delta_{w,z}) = (m-n)\tau(K) + n.\]
To compute $\epsilon(Q_{m,n}(K))$, we need to construct the subcomplex of the $CFK_\mathcal{R}$ complex containing $x_1$. The relevant differentials of this subcomplex are drawn in Figure 
\ref{fig:taug0eps1diff}. Note that the cycle $\sum^{2n+1}_{i=1} x_{2i+1}$ survives in the homology $\widehat{HF}(S^3)$, so this cycle must be the distinguished element of some horizontally simplified basis. It remains to look at its position in the vertical complex to recover $\epsilon(Q_{m,n}(K))$. There is a vertical arrow from $y_i$ to $x_i$ for each odd $i$, and together they give us a vertical arrow from $\sum^{n}_{i=0}y_{2i+1}$ to $\sum^{n}_{i=0}x_{2i+1}$. Thus, the cycle $\sum^{2n+1}_{i=0} x_{2i+1}$ is a boundary with respect to the vertical differential, which means that $\epsilon(Q_{m,n}(K)) = 1$. 

\underline{\textit{Case 2: $\tau(K) \geq 0$ and $\epsilon(K) = -1$}}. The pairing diagram is shown in Figure \ref{fig:pairing tau geq 0 eps=-1}. As before, we perform isotopies that eliminate intersection points of the same filtration difference. If $\tau(K)=0$ and $m=n$, then the intersection point that survives the spectral sequence is $x_3$; otherwise, it is $x_1$. To compute the Alexander grading of $x_3$, we shift the diagram down by $\tau(K)-1$ rows, and find that the image $x_3'$ has grading $-m$. Thus, when $\tau(K)=0$ and $m=n$, we have
\[\tau(Q_{m,n}(K))= -A(x_3) = -(-(m-n)(\tau(K)-1) - m) = m.\] Otherwise, we calculate the grading of $x_1$ using the same method, and obtain that 
\[\tau(Q_{m,n}(K))= -A(x_1) = -(-(m-n)\tau(K) - (m-1)) = (m-n)\tau(K) + (m-1).\]
We construct the subcomplex of $CFK_\mathcal{R}$ containing $x_1$ and $x_3$, and the relevant part is shown in Figure \ref{fig:taug0eps-1diff}. The analysis for $\epsilon(Q_{m,n}(K))$ the same as the case when $\tau(K)>0$ and $\epsilon(K)=1$, displayed in Figure \ref{fig:taug0eps1diff}. Hence, we have $\epsilon(Q_{m,n}(K)) = 1$. 

\begin{figure}[!htb]
    \centering
    \includegraphics[scale=0.9]{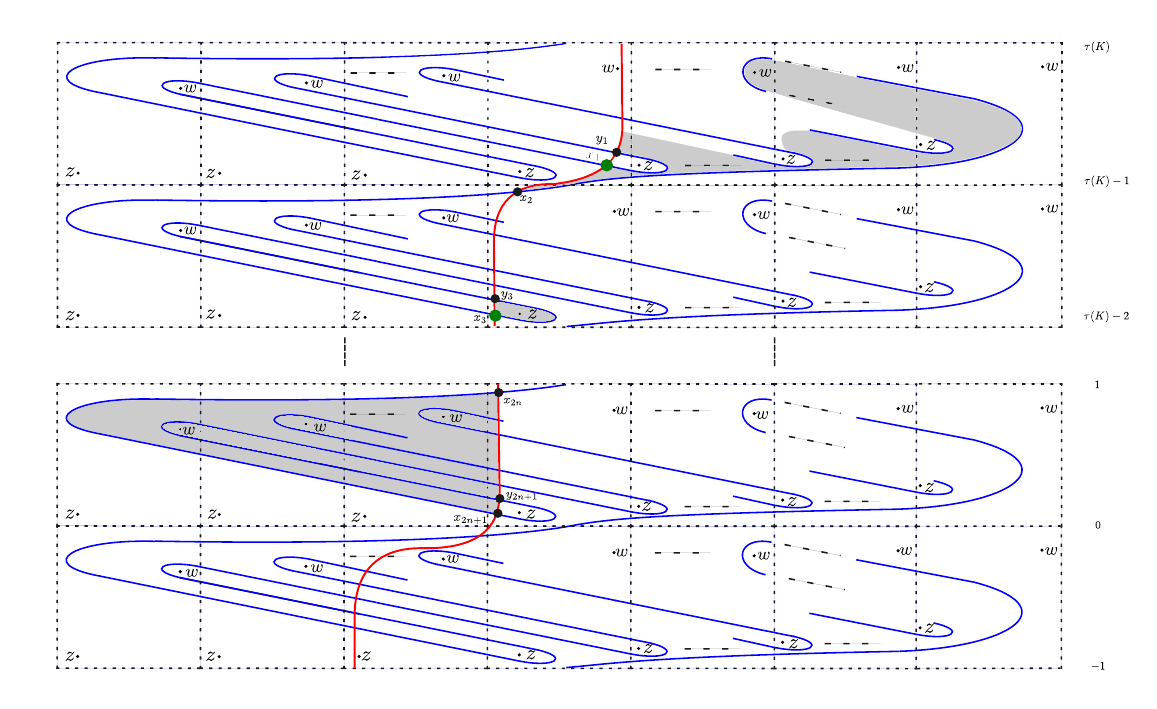}
    \caption{The pairing diagram for $Q_{m,n}(K)$ when $\tau(K) \geq 0, \epsilon(K) = -1$. Some bigons contributing to the differentials in the subcomplex are shaded. From top to bottom, they correspond to the differentials $x_1 \xrightarrow{U^m} x_{2}$, $y_3 \xrightarrow{V} x_3$, and $x_{2n+1} \xrightarrow{U^n} x_{2n}$, respectively.}
    \label{fig:pairing tau geq 0 eps=-1}
\end{figure}

\begin{figure}[!htb]
\[\begin{tikzcd}[sep = small]
	&&&&& {y_{2n+1}} \\
	&&&& {y_{2n-1}} & {x_{2n+1}} \\
	&&&& {x_{2n-1}} && {x_{2n}} \\
	&&& \textcolor{white}{\bullet} && {x_{2n-2}} \\
	& {y_3} && \dots \\
	{y_1} & {x_3} && \textcolor{white}{\bullet} \\
	{x_1} && {x_2}
	\arrow["{U^{m-1}}", from=7-1, to=7-3]
	\arrow["V"', from=6-1, to=7-1]
	\arrow["{U^n}", from=6-2, to=7-3]
	\arrow["V"', from=5-2, to=6-2]
	\arrow["{U^m}", from=6-2, to=6-4]
	\arrow["{U^m}", from=4-4, to=4-6]
	\arrow["{U^n}", from=3-5, to=4-6]
	\arrow["V"', from=2-5, to=3-5]
	\arrow["{U^m}", from=3-5, to=3-7]
	\arrow["{U^n}", from=2-6, to=3-7]
	\arrow["V"', from=1-6, to=2-6]
\end{tikzcd}\]
\caption{The relevant differentials in the $CFK_\mathcal{R}$ complex of $Q_{m,n}(K)$ when $\tau(K) \geq 0$, $\epsilon(K) = -1$.}
\label{fig:taug0eps-1diff}
\end{figure}
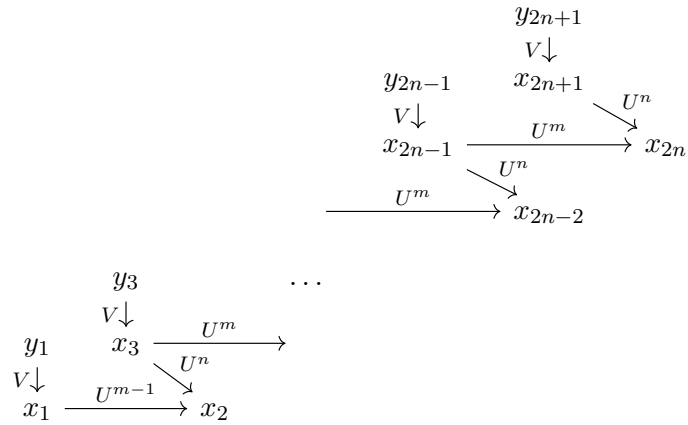

\underline{\textit{Case 3: $\tau(K) \leq 0$ and $\epsilon(K) = 1$}}. The pairing diagram is shown in Figure \ref{fig:pairing tau leq 0 eps=1}. After eliminating all pairs of intersection points with the same filtration difference, we are left with $x_{2n+1}$, which has Alexander grading $-(m-n)\tau(K)$. Therefore, we have 
\[\tau(Q_{m,n}(K)) = -A(x_{2n+1})= (m-n)\tau(K).\]
Figure \ref{fig:taul0eps1diff} shows the relevant differentials 
in the $CFK_\mathcal{R}$ complex containing $x_{2n+1}$, and one observes that $x_{2n+1}$ is the distinguished element in some horizontal simplified basis. Since there's a vertical differential pointing from $y_1$ to $U^{n}x_{2n+1}$, we have $\epsilon(Q_{m,n}(K)) = 1$.

\begin{figure}[!htb]
    \centering
    \includegraphics[scale = 0.75]{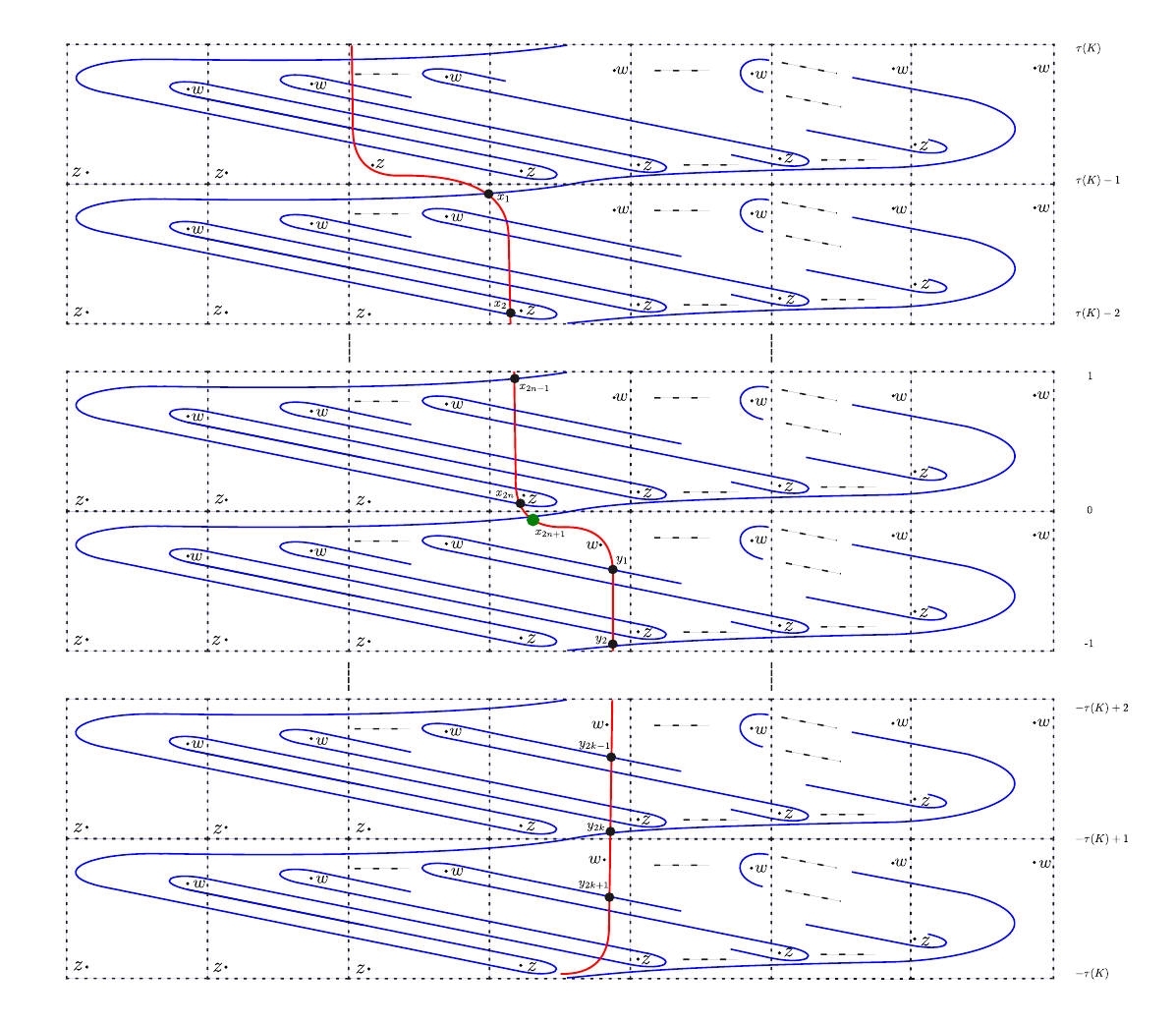}
    \caption{The pairing diagram for $Q_{m,n}(K)$ when $\tau(K) \leq 0, \epsilon(K) = 1$.}
    \label{fig:pairing tau leq 0 eps=1}
\end{figure}

\begin{figure}[!htb]
\[\begin{tikzcd}
	&&&&&&&& {x_2} && {x_1} \\
	&&&&&&&& {x_4} && {x_3} \\
	{y_{2k+1}} && {y_{2k-1}} && \cdots && {y_1} && \vdots && \vdots \\
	{y_{2k}} && \cdots && {y_2} && {x_{2n+1}} && {x_{2n}} && {x_{2n-1}}
	\arrow["{U^n}"{pos=0.3}, from=1-9, to=1-11]
	\arrow["{U^m}", from=1-9, to=2-11]
	\arrow["{U^n}"{pos=0.3}, from=2-9, to=2-11]
	\arrow["{U^m}", from=2-9, to=3-11]
	\arrow["{U^m}", from=3-9, to=4-11]
	\arrow["{U^n}"{pos=0.3}, from=4-9, to=4-11]
	\arrow["{U^m}"', from=4-9, to=4-7]
	\arrow["{V^n}", from=3-7, to=4-7]
	\arrow["{U^{m-n}V^{n-1}}"'{pos=0.1}, from=3-7, to=4-5]
	\arrow["{V^n}", from=3-3, to=4-3]
	\arrow["{U^{m-n}V^{n-1}}"'{pos=0.1}, from=3-3, to=4-1]
	\arrow["{V^n}", from=3-1, to=4-1]
	\arrow["{U^{m-n}V^{n-1}}"'{pos=0.1}, from=3-5, to=4-3]
\end{tikzcd}\]
\caption{The relevant differentials in the $CFK_\mathcal{R}$ complex of $Q_{m,n}(K)$ when $\tau(K) \leq 0$, $\epsilon(K) = 1$.}
\label{fig:taul0eps1diff}
\end{figure}
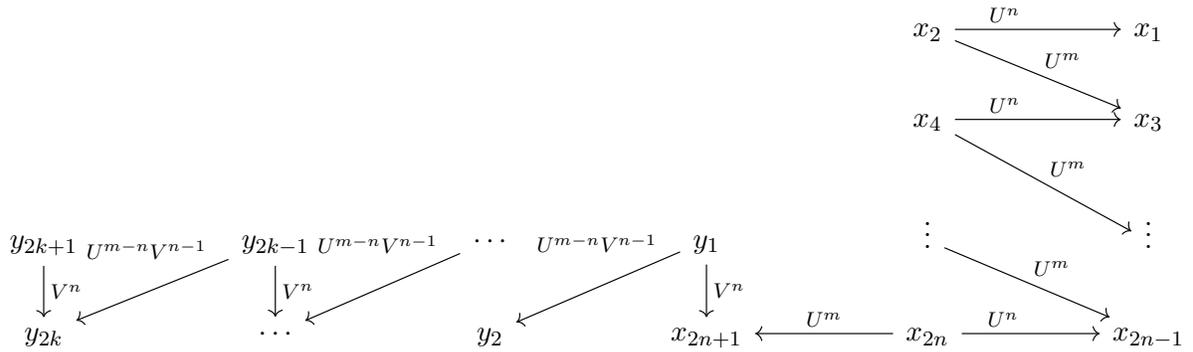

\underline{\textit{Case 4: $\tau(K) < 0$ and $\epsilon(K) = -1$}}. The pairing diagram is displayed in Figure \ref{fig:pairing tau < 0 eps=-1}, where the remaining intersection point is $x_{2n-1}$, which has Alexander grading $-(m-n)(\tau(K)-1)$. This gives
\[\tau(Q_{m,n}(K))= -A(x_{2n-1}) = (m-n)(\tau(K)) + (m-1)) = (m-n)\tau(K) + (m-1).\]
In Figure \ref{fig:taul0eps-1diff}, we show the relevant differentials in the subcomplex containing $x_{2n+1}$, which is a distinguished element in some horizontally simplified basis. Since a vertical differential sends $y_1$ to $U^{n}x_{2n-1}$, we conclude that $\epsilon(Q_{m,n}(K))=1$.
\begin{figure}[!htb]
    \centering
    \includegraphics[scale=0.8]{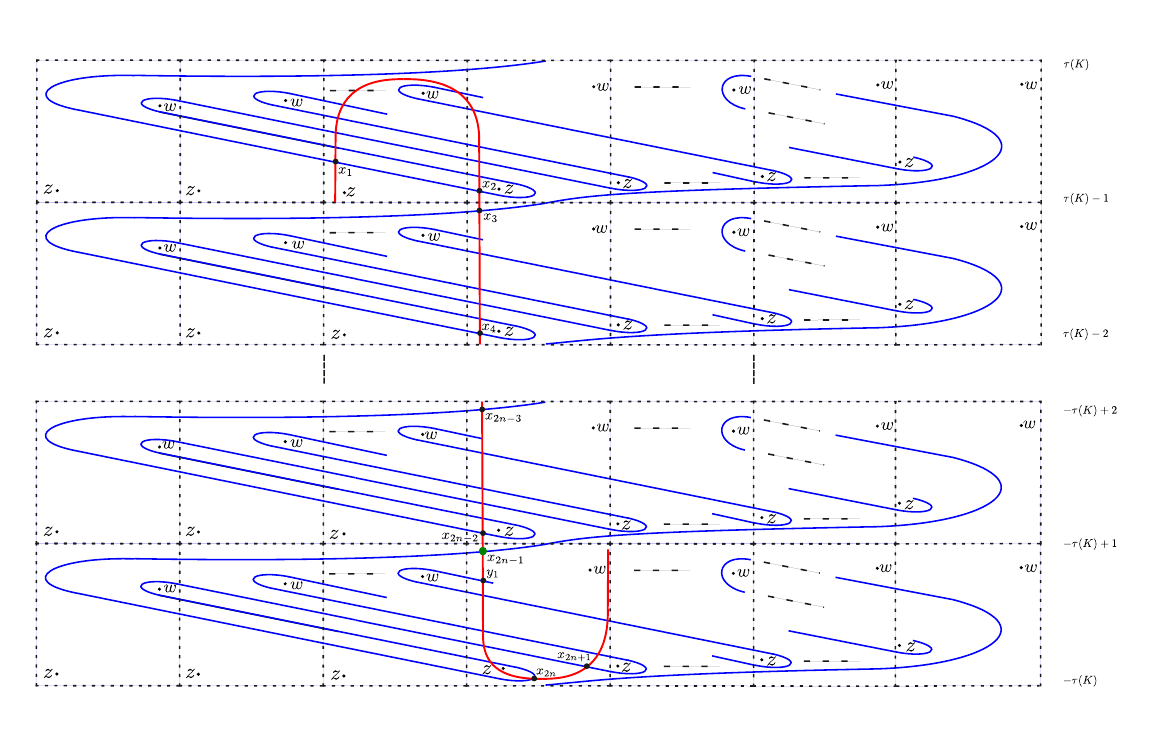}
    \caption{The pairing diagram for $Q_{m,n}(K)$ when $\tau(K) < 0, \epsilon(K) = -1$.}
    \label{fig:pairing tau < 0 eps=-1}
\end{figure}

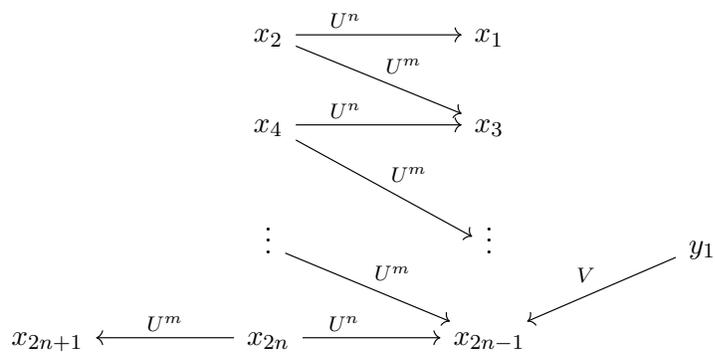
\begin{figure}[!htb]

\[\begin{tikzcd}
	&& {x_2} && {x_1} \\
	&& {x_4} && {x_3} \\
	&& \vdots && \vdots && {y_1} \\
	{x_{2n+1}} && {x_{2n}} && {x_{2n-1}} &&& \textcolor{white}{\bullet}
	\arrow["{U^n}"{pos=0.3}, from=1-3, to=1-5]
	\arrow["{U^m}", from=1-3, to=2-5]
	\arrow["{U^n}"{pos=0.3}, from=2-3, to=2-5]
	\arrow["{U^m}", from=2-3, to=3-5]
	\arrow["{U^m}", from=3-3, to=4-5]
	\arrow["{U^n}"{pos=0.3}, from=4-3, to=4-5]
	\arrow["{U^m}"', from=4-3, to=4-1]
	\arrow["V"', from=3-7, to=4-5]
\end{tikzcd}\]
\caption{The relevant differentials in the $CFK_\mathcal{R}$ complex of $Q_{m,n}(K)$ when $\tau(K) < 0$, $\epsilon(K) = -1$.}
\label{fig:taul0eps-1diff}
\end{figure}
\end{proof}

The proof of Theorem \ref{main thm}
allows us to find the three-genus of the patterns $Q_{m,n}$. Recall that by \cite{OS04a}, the genus of the knot is the largest Alexander grading supporting non-zero Floer homology. Moreover, for a satellite knot $P(K)$, we have a genus formula from \cite{Sch53}
\begin{equation}\label{genus equation}
    g(P(K)) = |w(P)|g(K) + g(P),
\end{equation}
where $|w(P)|$ is the absolute value of the winding number.

\begin{proposition}
For $m,n\ge 1$, the generalized Mazur patterns have genus $g(Q_{m,n}) = \min(m,n).$
\end{proposition}
\begin{proof}
We take $P$ to be $Q_{m,n}$ and reverse the orientation (where without loss of generality we assume that $m\ge n$), which has winding number $m-n$, and $K$ to be the right-handed trefoil $T_{2,3}$, which has $\tau(T_{2,3})=\epsilon(T_{2,3})= g(T_{2,3}) = 1$. By Equation \eqref{genus equation}, we have
\begin{equation} \label{genus equation 2}
g(Q_{m,n}(T_{2,3}))=(m-n)+g(Q_{m,n}).
\end{equation}
We want to compute $g(Q_{m,n}(T_{2,3}))$, which is equal to the largest Alexander grading for which $\widehat{HFK}(Q_{m,n}(K))$ is non-zero. Consider the pairing diagram in Figure \ref{fig:pairing tau>0 eps=1} from the proof of Theorem \ref{eps thm} and note that $x_1$ has the lowest Alexander grading, but after taking the reverse this generator has the largest Alexander grading. Since $x_1$ is also the distinguished generator,
\[g(Q_{m,n}(T_{2,3}))=A(x_1) = \tau(Q_{m,n}(T_{2,3})) = |m-n|\tau(T_{2,3}) + n = m.\]
Plugging into Equation \eqref{genus equation 2}, we get $g(Q_{m,n}) = n$. For any $m,n\ge 1$, we have $g(Q_{m,n}) = \min(m,n)$.
\end{proof}

\appendix \label{compute (m,n)}
\section{Computation of $\tau(Q_{m,n}(K))$ using bordered Heegaard Floer homology}
We give an alternative way of computing the $\tau$-invariant for satellites along generalized Mazur patterns, using the ordinary bordered Heegaard Floer homology. This is a generalization of the method used in \cite{Lev16}. As expected, we obtain the same value of $\tau(Q_{m,n}(K))$ as in Theorem \ref{main thm}.

The organization of the Appendix is as follows: in Section \ref{strategy}, we describe our strategy for computing $\tau(P(K))$; in Section \ref{Examples}, we give a few examples for computing $\tau(Q_{m,n}(K))$ for small values of $m$ and $n$; and in Section \ref{computation}, we carry out the computation for general $m$ and $n$. 

  \subsection{Bordered strategy to determine $\tau(P(K))$} \label{strategy} 
  In this section, we follow closely the notation in \cite{Lev16}. Recall that an alternative definition of the invariant $\tau(K)$ is 
     \[\tau(K)=-\max\{s\mid U^n\cdot HFK^-(K,s)\ne 0 \textrm{ for all } n\ge 0\}.\] 
 In other words, $\tau(K)$ is minus the Alexander grading of the non-vanishing generator for $\mathbb{F}[U]$ in $HFK^-(S^3, K)$. Recall also the pairing theorem \cite[Theorem 1.3]{LOT18}: for a satellite knot with pattern $P\subset V$ and companion $K$, we have 
    \begin{equation}\label{pairing}
    gCFK^-(P(K))\simeq CFA^-(V,P)\boxtimes \widehat{CFD}(X_K).
    \end{equation}
Here, we may again assume that the pattern knot $P$ is a $(1,1)$-unknot pattern. 

Our goal is to directly compute the associated graded on the left hand side. We discussed how to compute $CFA^{-}$ in Section \ref{CFA hat}, and 
$\widehat{CFD}(X_K)$ in Section \ref{sec: CFD hat}. The box tensor product $CFA^-(V,P)\boxtimes \widehat{CFD}(X_K)$ is defined to be the $\mathbb{F}$-vector space
\[CFA^-(V,P)\boxtimes_{\mathcal{I}} \widehat{CFD}(X_K)\]
with the differential now given by the combinatorial formula
\[\partial^\boxtimes (x\otimes y) = \sum_{k+1}(x,\rho_{i_1},\cdots, \rho_{i_k})D_{i_k}\circ\cdots \circ D_{i_1}(y)\]
where the sum is taken over all $k$-element sequences $i_1,\cdots, i_k$ (including the empty sequence when $k=0$) of elements in $\{\emptyset, 1, 2, 3, 12, 23, 123\}$.

The box tensor product in Equation \eqref{pairing} consists of one direct summand whose homology contains a $\mathbb{F}[U]$ part, obtained by box tensoring with a unstable neighborhood in $\widehat{CFD}(X_K)$, and other summands whose homology is $U$-torsion. We are interested only in the summand with non-vanishing homology, since the Alexander grading of its generator gives us $-\tau(P(K))$.

It remains determine the absolute Alexander grading in $gCFK^-(S^3, P(K))$, which we denote by $A_{Q_{m,n}(K)}$. This can be achieved by the following  proposition in \cite{Lev16}.


\begin{proposition}\cite[Proposition 2.2]{Lev16}\label{key prop}
Let $P \subset V$ be a based knot with winding number $m$. For each element $a \in \mathrm{CFA}^{-}(V, P) \cdot \iota_0$, there exists a constant $C_a$ with the following property: For any knot $K \subset S^3$, and any homogeneous element $x \in \iota_0 \widehat{\mathrm{CFD}}\left(X_K\right)$, we have
\begin{equation} \label{levine eq}
A_{P(K)}(a \otimes x)=m A_K(x)+C_a.
\end{equation}
\end{proposition}

In what follows, we will set $P$ to be the generalized Mazur pattern $Q_{m,n}$. Note that the constant $C_a$ in equation \eqref{levine eq} is independent of $K$. This allows us to determine $C_a$ by setting $K$ to be the unknot $O$. Since $Q_{m,n}(O)$ is also the unknot, the tensor product complex $CFA^-(V,Q_{m,n})\boxtimes \widehat{CFD(X_O)}$ is generated over $\mathbb{F}[U]$ by a single element with Alexander grading 0. Moreover, we have $A_O(x)=0$ for any $x\in \iota_0 \widehat{\mathrm{CFD}}\left(X_O\right)$. Equation \eqref{levine eq} then allows us to determine the constant $C_a$ for certain generators $a\in \mathrm{CFA}^{-}(V, Q_{m,n})$. 

For any knot $K$, we may apply Proposition \ref{key prop} again to compute the absolute Alexander gradings of the relevant generators of $CFK^-(S^3, Q_{m,n}(K))$. This computation is sufficient to determine $\tau(Q_{m,n}(K))$; we will carry out this computation in more detail in Section \ref{computation}. Before that, we give some examples of how this strategy may be applied for small values of $m$ and $n$.

\subsection{Examples}\label{Examples}
\begin{example}[The pattern $Q_{3,1}$]\label{Ex:Q31}
    First, we obtain the following Heegaard diagram for the pattern $Q_{3,1}$ by applying the inductive procedure in Proposition \ref{h diagram}. The resulting diagram is shown in Figure \ref{fig:Q31}. Let $x_1$ to $x_7$ be the intersections of the $\beta$ curve with the vertical $\alpha_1$ arcs, ordered from top to bottom, and let $y_1$ to $y_{12}$ be the intersections with the horizontal $\alpha_2$ arcs from left to right.
    
    \begin{figure}[htbp]
        \centering
        \includegraphics[scale=0.55]{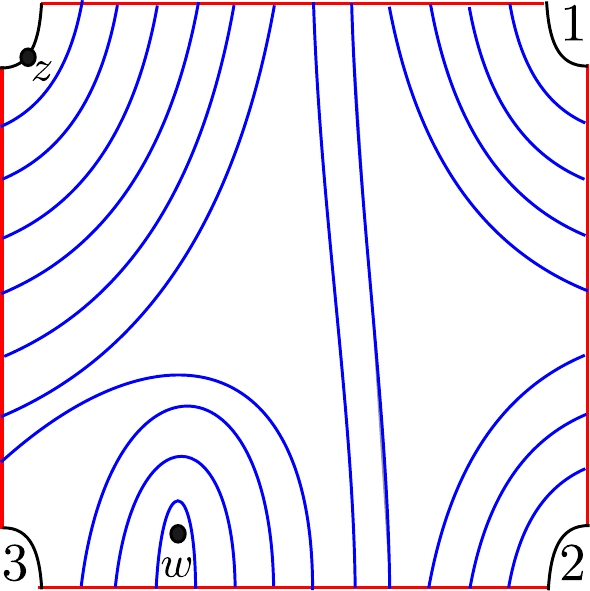}
        \caption{The Heegaard diagram for $Q_{3,1}$.}
        \label{fig:Q31}
    \end{figure}
    
    Next, we find all the pseudoholomorphic disks in the universal cover, shown in Figure \ref{fig:universal_cover}.
    
    \begin{figure}[htbp]
    \centering
    \includegraphics[scale = 0.5]{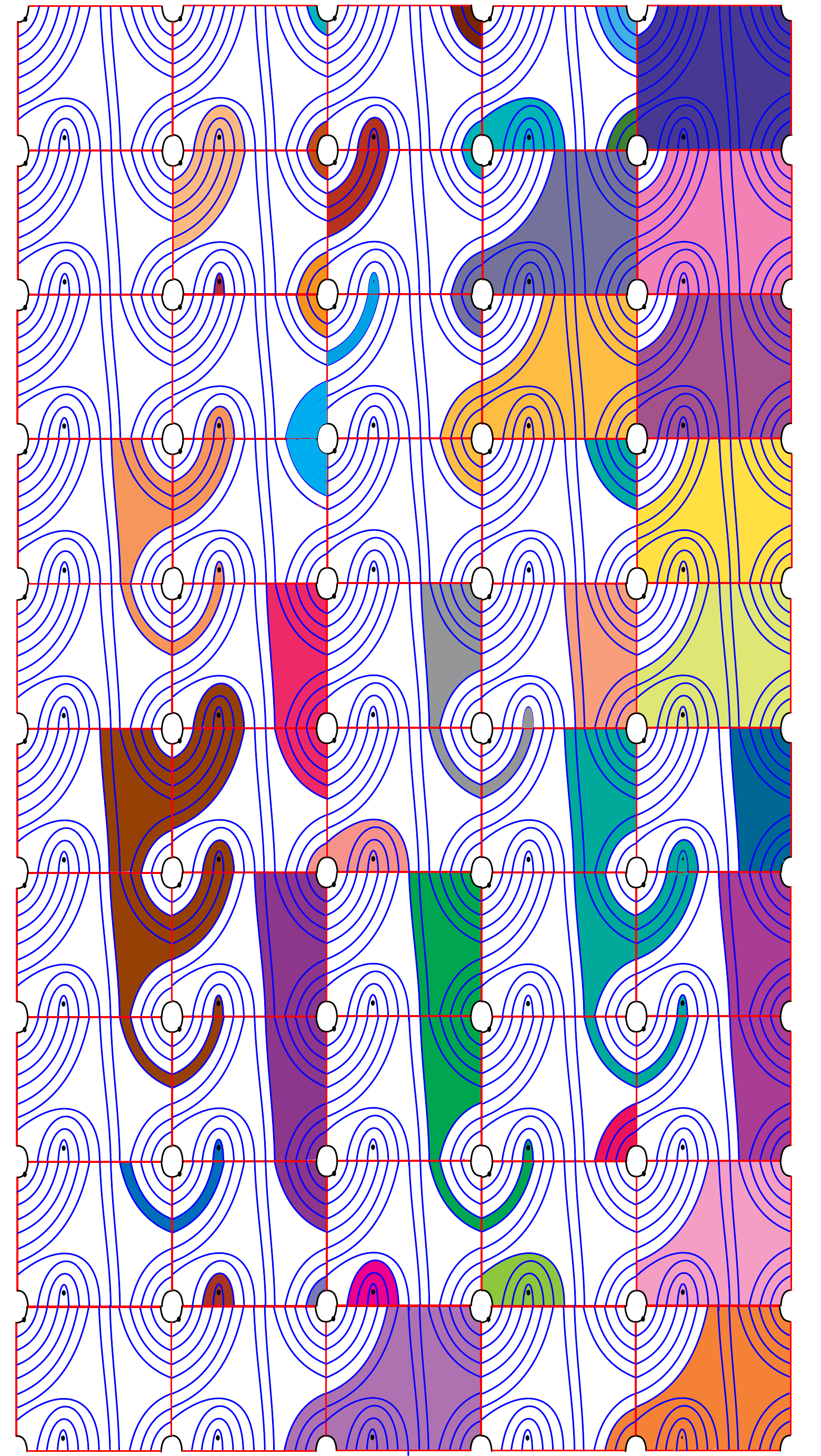}
    \caption{The universal cover of the Heegaard diagram for $Q_{3,1}$.}
    \label{fig:universal_cover}
    \end{figure}

    The corresponding $A_\infty$ multiplications are:
    \begin{center}
        \begin{AutoMultiColItemize}
        \item \crule[Apricot]{0.5cm}{0.25cm} $m_1(x_1) = Ux_6$
        \item \crule[Aquamarine]{0.5cm}{0.25cm} $m_2(x_1,\rho_1) = y_{12}$
        \item \crule[Bittersweet]{0.5cm}{0.25cm} $m_2(x_1,\rho_{12}) = x_{7}$
        \item \crule[BlueGreen]{0.5cm}{0.25cm} $m_2(x_1,\rho_{123}) = Uy_{7}$
        \item \crule[BlueViolet]{0.5cm}{0.25cm} $m_4(x_1,\rho_{3},\rho_{2},\rho_{1}) = Uy_{1}$
        \item \crule[BrickRed]{0.5cm}{0.25cm} $m_1(x_2) = Ux_{5}$
        \item \crule[Brown]{0.5cm}{0.25cm} $m_2(x_2,\rho_1) = y_{11}$
        \item \crule[BurntOrange]{0.5cm}{0.25cm} $m_2(x_2,\rho_{12}) = x_{6}$
        \item \crule[CadetBlue]{0.5cm}{0.25cm} $m_4(x_2,\rho_{123},\rho_2,\rho_1) = Uy_{6}$
        \item \crule[CarnationPink]{0.5cm}{0.25cm} $m_4(x_2,\rho_{3},\rho_2,\rho_1) = Uy_{2}$
        \item \crule[Cerulean]{0.5cm}{0.25cm} $m_1(x_3) = Ux_{4}$
        \item \crule[CornflowerBlue]{0.5cm}{0.25cm} $m_2(x_3,\rho_1) = y_{10}$
        \item \crule[Cyan]{0.5cm}{0.25cm} $m_2(x_3,\rho_{12}) = x_{5}$
        \item \crule[Dandelion]{0.5cm}{0.25cm} $m_4(x_3,\rho_{123},\rho_2,\rho_1) = Uy_{5}$
        \item \crule[DarkOrchid]{0.5cm}{0.25cm} $m_4(x_3,\rho_{3},\rho_2,\rho_1) = Uy_{3}$
        \item \crule[WildStrawberry]{0.5cm}{0.25cm} $m_2(x_4,\rho_1) = y_{9}$
        \item \crule[ForestGreen]{0.5cm}{0.25cm} $m_3(x_4,\rho_{12},\rho_1) = y_{8}$
        \item \crule[Fuchsia]{0.5cm}{0.25cm} $m_4(x_4,\rho_{12},\rho_{12},\rho_1) = y_{7}$
        \item \crule[Goldenrod]{0.5cm}{0.25cm} $m_4(x_4,\rho_{3},\rho_2,\rho_1) = Uy_{4}$
        \item \crule[Gray]{0.5cm}{0.25cm} $m_2(x_5,\rho_1) = Uy_{8}$
        \item \crule[Green]{0.5cm}{0.25cm} $m_3(x_5,\rho_{12},\rho_1) = Uy_{7}$
        \item \crule[GreenYellow]{0.5cm}{0.25cm} $m_4(x_5,\rho_{3},\rho_2,\rho_1) = Uy_{5}$ 
        \item \crule[JungleGreen]{0.5cm}{0.25cm} $m_2(x_6,\rho_1) = U^{2}y_{7}$
        \item \crule[Lavender]{0.5cm}{0.25cm} $m_4(x_6,\rho_3,\rho_2,\rho_1) = Uy_{6}$
        \item \crule[LimeGreen]{0.5cm}{0.25cm} $m_2(x_7,\rho_3) = Uy_{7}$
        \item \crule[Magenta]{0.5cm}{0.25cm} $m_1(y_1) = Uy_{6}$
        \item \crule[Mahogany]{0.5cm}{0.25cm} $m_1(y_2) = Uy_{5}$
        \item \crule[Maroon]{0.5cm}{0.25cm} $m_1(y_3) = Uy_{4}$
        \item \crule[Melon]{0.5cm}{0.25cm} $m_3(y_8,\rho_2,\rho_1) = y_{7}$
        \item \crule[MidnightBlue]{0.5cm}{0.25cm} $m_3(y_9,\rho_2,\rho_1) = y_{8}$
        \item \crule[Mulberry]{0.5cm}{0.25cm} $m_4(y_9,\rho_2,\rho_{12},\rho_1) = y_{7}$
        \item \crule[OrangeRed]{0.5cm}{0.25cm} $m_2(y_{10},\rho_2) = x_{5}$
        \item \crule[Orchid]{0.5cm}{0.25cm} $m_4(y_{10},\rho_{23},\rho_2,\rho_1) = Uy_{5}$
        \item \crule[NavyBlue]{0.5cm}{0.25cm} $m_1(y_{10}) = Uy_{9}$
        \item \crule[OliveGreen]{0.5cm}{0.25cm} $m_2(y_{11},\rho_2) = x_{6}$
        \item \crule[Orange]{0.5cm}{0.25cm} $m_4(y_{11},\rho_{23},\rho_2,\rho_1) = Uy_{6}$
        \item \crule[Peach]{0.5cm}{0.25cm} $m_1(y_{11}) = U^{2}y_{8}$
        \item \crule[Periwinkle]{0.5cm}{0.25cm} $m_2(y_{12},\rho_2) = x_{7}$
        \item \crule[Salmon]{0.5cm}{0.25cm} $m_2(y_{12},\rho_{23}) = Uy_{7}$
        \item \crule[RawSienna]{0.5cm}{0.25cm} $m_1(y_{12}) = U^{3}y_7$ 
        \end{AutoMultiColItemize}
    \end{center}
These relations can be represented by the following diagram
\begin{equation*}
        \begin{tikzcd}
            x_7 \arrow[rdd, swap, "U\rho_{3}"] & y_{12} \arrow[l, swap, "\rho_2"] \arrow[dd,swap, "\substack{U^3 + \\ U\rho_{23}}", near start] & x_1 \arrow[l, swap, "\rho_1"] \arrow[ll, swap, bend right, "\rho_{12}"] \arrow[ldd,swap,"U\rho_{123}",sloped] \arrow[dd, swap, "U"] & & y_{11} \arrow[lldd, swap, "\rho_2"] \arrow[dd, swap, "U^2"] & x_2 \arrow[l, swap, "\rho_1"] \arrow[dd, swap, "U"] \arrow[llldd, swap, bend right=60, "\rho_{12}"] & & y_{10} \arrow[lldd, swap, "\rho_2"] \arrow[dd, swap, "U"] & x_3 \arrow[l, swap, "\rho_1"] \arrow[dd, swap, "U"] \arrow[llldd, swap, bend right=60, "\rho_{12}"] & y_1 \arrow[dd, swap, "U"] & y_2 \arrow[dd, swap, "U"] & y_3 \arrow[dd, swap, "U"]
            \\
            \\
            & y_7  & x_6 \arrow[l, swap, "U^{2}\rho_1"] & & y_8 \arrow[lll, swap, bend left, "\rho_{2}\rho_{1}"] & x_5 \arrow[l, swap, "U\rho_1"] \arrow[llll, swap, bend left=40, "U\rho_{12}\rho_1"] & & y_9 \arrow[lll, swap, bend left, "\rho_{2}\rho_{1}"] \arrow[llllll, swap, bend left=50, "\rho_{2}\rho_{12}\rho_{1}"] & x_4 \arrow[l, swap, "\rho_1"] \arrow[llll, swap, bend left=40, "U\rho_{12}\rho_1"] \arrow[lllllll, swap, bend left = 50, "\rho_{12}\rho_{12}\rho_1"] & y_6 & y_5 & y_4
        \end{tikzcd}
    \end{equation*}
One may observe that a few differentials are not represented in the diagram above, namely 
\begin{equation*}
\begin{aligned}
x_1 &\stackrel{U \rho_3\rho_2\rho_1}{\longrightarrow} y_1, &
y_{11} &\stackrel{U \rho_{23}\rho_2\rho_1}{\longrightarrow} y_6, &
x_2 &\stackrel{U \rho_{123}\rho_2\rho_1}{\longrightarrow} y_6, &
x_6 &\stackrel{U \rho_3\rho_2\rho_1}{\longrightarrow} y_6, &
x_2 &\stackrel{U \rho_3\rho_2\rho_1}{\longrightarrow} y_2, \\
y_{10} &\stackrel{U \rho_{23}\rho_2\rho_1}{\longrightarrow} y_5, &
x_3 &\stackrel{U \rho_{123}\rho_2\rho_1}{\longrightarrow} y_5, &
x_5 &\stackrel{U \rho_3\rho_2\rho_1}{\longrightarrow} y_5, &
x_3 &\stackrel{U \rho_3\rho_2\rho_1}{\longrightarrow} y_3, &
x_4 &\stackrel{U \rho_3\rho_2\rho_1}{\longrightarrow} y_4.
\end{aligned}
\end{equation*}
However, we may apply the following change of basis to eliminate these arrows: $x_6' = x_6 + U\rho_3\rho_2\rho_1y_1$, $x_5' = x_5 + U\rho_3\rho_2\rho_1y_2$, and $x_4' = x_4 + U\rho_3\rho_2\rho_1y_3$.

\end{example}

\begin{example}[The pattern $Q_{1,2}$]
\begin{figure}[htbp]
    \centering
    \begin{minipage}{0.45\textwidth}
        \centering
        \includegraphics[scale=0.6]{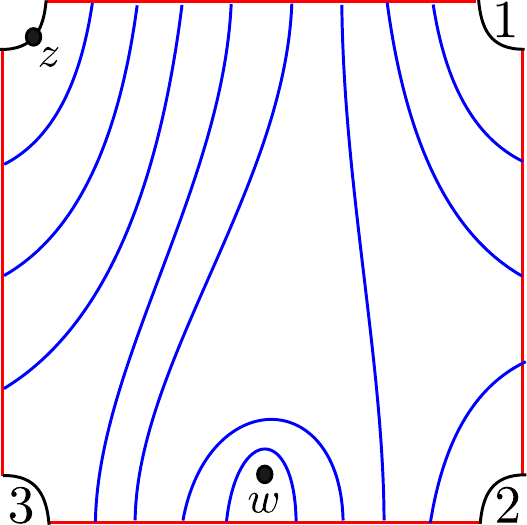}
        \caption{The Heegaard diagram for $Q_{1,2}$.}
        \label{fig:Q12}
    \end{minipage}
    \hfill
    \begin{minipage}{0.5\textwidth}
        \centering
        \includegraphics[scale=0.54]{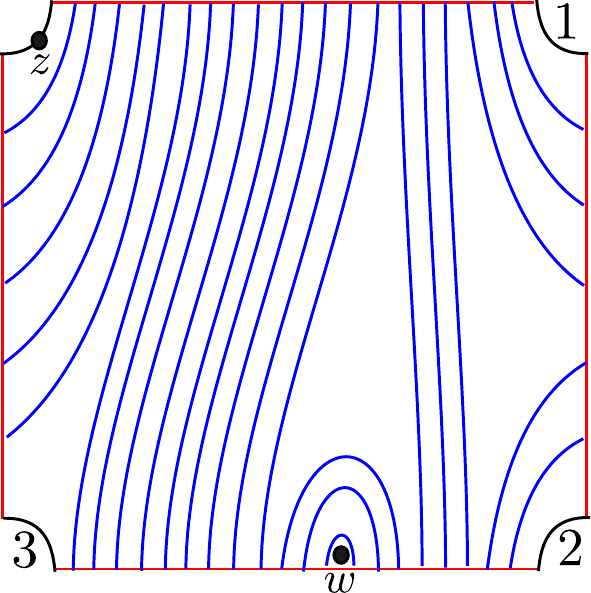}
        \caption{The Heegaard diagram for $Q_{2,3}$.}
        \label{fig:Q23}
    \end{minipage}
    \label{fig:Q12_Q23}
\end{figure}

    The Heegaard diagram for $Q_{1,2}$ is shown in Figure \ref{fig:Q12}. After counting  pseudoholomorphic disks in the universal cover, we obtain the following $A_\infty$ relations:
\begin{AutoMultiColItemize}
    \item $m_1(x_1)=Ux_2$
        \item $m_2(x_1, \rho_1)=y_8$
        \item $m_2(x_2, \rho_1)=y_7$
        \item $m_2(x_1, \rho_{12})=x_3$
        \item $m_2(x_1, \rho_{123})=U^2y_7$
        \item $m_2(x_3, \rho_{3})=U^2y_7$
        \item $m_2(y_8, \rho_2) = x_3$
        \item $m_2(y_8, \rho_{23}) = U^2y_7$
        \item $m_1(y_8) = U^2y_7$
        \item $m_4(x_1,\rho_{123},\rho_2, \rho_1)=Uy_3$
        \item $m_4(x_1,\rho_{3},\rho_2, \rho_1)=Uy_1$
        \item $m_5(x_1,\rho_3,\rho_2,\rho_{12},\rho_1)=U^2y_4$
        \item $m_3(x_2,\rho_{12},\rho_1)=y_6$
        \item $m_4(x_2,\rho_3,\rho_2, \rho_1)=Uy_2$
        \item $m_5(x_2,\rho_3,\rho_2,\rho_{12},\rho_1)=U^2y_5$
        \item $m_2(x_3,\rho_1)=Uy_6$
        \item $m_4(x_3,\rho_3,\rho_2, \rho_1)=Uy_3$
        \item $m_1(y_1)=Uy_2$
        \item $m_1(y_3)=Uy_6$
        \item $m_1(y_4)=Uy_5$
        \item $m_3(y_7,\rho_2,\rho_1)=y_6$
        \item $m_3(y_8,\rho_{23},\rho_2, \rho_1)=Uy_3$
        \item $m_3(y_1, \rho_{2}\rho_{1}) = Uy_4$
        \item $m_3(y_2, \rho_{2}\rho_{1}) = Uy_5$
  \end{AutoMultiColItemize}
\noindent After applying the change of basis in Remark \ref{rmk change of basis}, these relations are represented by the following $CFA^{-}$ complex
    \begin{equation*}
        \begin{tikzcd}
        y_3 \arrow[dd, "U", swap] 
        & x_3 \arrow[l, "U\rho_{3}\rho_{2}\rho_{1}", swap] \arrow[ldd, "U\rho_1"] \arrow[rrdd, swap, "U^{2}\rho_{3}"] & & y_8 \arrow[ll, swap, "\rho_{2}"]  \arrow[dd, swap, "\substack{U + \\ U^{2}\rho_{23}}"]
            & & x_1 \arrow[lldd, swap,"\mathclap{U^2\rho_{123}}",sloped] \arrow[ll, swap,  "\rho_{1}"] \arrow[llll, bend right=30, swap, "\rho_{12}"] \arrow[dd, "U"] \arrow[lllll, bend right=60, swap, "U\rho_{123}\rho_{2}\rho_{1}"] & y_{1} \arrow[dd, "U"] & y_{4} \arrow[dd, "U"]\\
             & & & & & \\  y_6 & & & y_7 \arrow[lll, swap, bend left, "\rho_{2}\rho_{1}"]
            & & x_2 \arrow[ll, swap, "\rho_{1}"] \arrow[lllll, swap, bend left, "\rho_{12}\rho_{1}"] & y_2 & y_5
    \end{tikzcd}
       \end{equation*}

As before, one may check that the summands in the box tensor product with $\widehat{CFD}(X_K)$ look like the figures given for each value of $\tau(K)$ after plugging in $m=1$ and $n=2$. The only difference is that there are no dashed arrows, since there are no $\rho_2$ arrows pointing to $x_2$ in $CFA^-(V, Q_{1,2})$ that will pair nontrivially with the $D_2$ differential fron $\lambda$ to $\xi_0$. 
\end{example}

\begin{example}[The pattern $Q_{2,3}$]
We describe the $A_\infty$-module $CFA^-(Q_{2,3}(K)).$ Using the inductive procedure outlined in Proposition \ref{h diagram}, we obtain the Heegaard diagram in Figure \ref{fig:Q23}. 

    We provide a less exhaustive $CFA^{-}$ structure below in Figure \ref{cfa23}, not including the isolated components. There are also arrows from $y_{18}$ and $x_3$ to $y_{16}$ featuring compositions of the arrows going through $y_{17}$, which have also been omitted for clarity.
    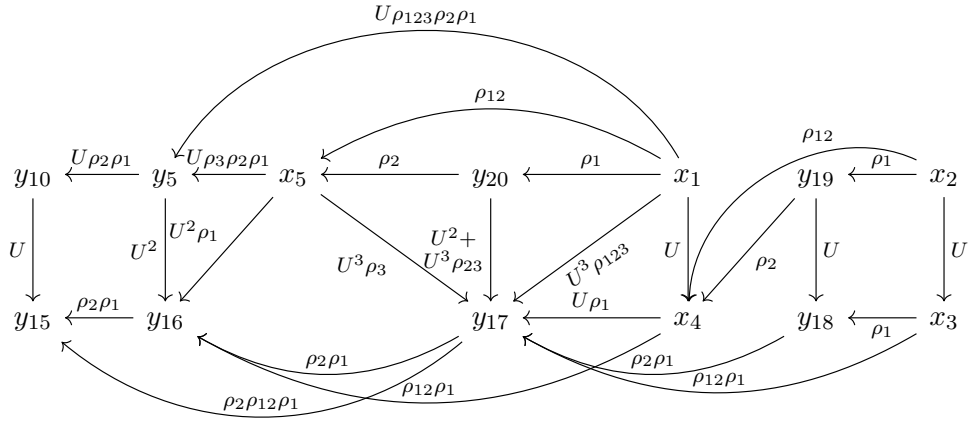
\begin{figure}[htbp]
    \begin{tikzcd}
         y_{10} \arrow[dd, swap, "U"] & y_5 \arrow[l, swap, "U\rho_{2}\rho_{1}"] \arrow[dd, swap, "U^2"] & x_5 \arrow[l, swap, "U\rho_{3}\rho_{2}\rho_{1}"] \arrow[ldd, swap, "U^{2}\rho_1"] \arrow[rrdd, swap, "U^{3}\rho_{3}"] 
        & & y_{20} \arrow[ll, swap, "\rho_{2}"]  \arrow[dd, swap, "\substack{U^{2} +\\ U^{3}\rho_{23}}"]
            & & x_1 \arrow[lldd, swap, "\mathclap{U^{3}\rho_{123}}",sloped] \arrow[ll, swap, "\rho_{1}"] \arrow[llll, bend right=30, swap, "\rho_{12}"] \arrow[dd, swap, "U"] \arrow[lllll, bend right=60, swap, "U\rho_{123}\rho_{2}\rho_{1}"]
            & y_{19} \arrow[ldd, "\rho_{2}"] \arrow[dd, "U"] & x_2 \arrow[l, swap, "\rho_1"] \arrow[dd, "U"] \arrow[lldd, swap, bend right=57, "\rho_{12}", near start]\\
         & & & & & &\\
         y_{15} & y_{16} \arrow[l, swap, "\rho_2\rho_1"] & & & y_{17} \arrow[llll, swap, bend left=40, "\rho_2\rho_{12}\rho_1"] \arrow[lll, swap, bend left, "\rho_{2}\rho_{1}"] & & x_4 \arrow[ll, swap, "U\rho_{1}"] \arrow[lllll, swap, bend left, "\rho_{12}\rho_{1}"]
            & y_{18} \arrow[lll, swap, bend left, "\rho_{2}\rho_{1}"] & x_{3} \arrow[l, "\rho_{1}"] \arrow[llll, swap, bend left, "\rho_{12}\rho_{1}"]
    \end{tikzcd}
    \caption{Part of the $A_{\infty}$-module $CFA^{-}(Q_{2,3}(K))$.}
    \label{cfa23}
    \end{figure}
\end{example}

\subsection{Computation of $\tau(Q_{m,n}(K))$}\label{computation}
In the case of general $m,n$, the invariant $\tau(Q_{m,n}(K))$ is given by the following theorem.
\begin{theorem}\label{tau thm}
Let $Q_{m,n}$ be the generalized Mazur pattern embedded in the solid torus $V$. If $m\ne n$, then for any knot $K\subset S^3$, we have \begin{equation}\label{tau eq}
    \tau(Q_{m,n}(K)) = \begin{cases}
    |m-n|\tau(K) & \text{if } \tau(K) \le 0 \text{ and } \epsilon(K) \in \{0,1\}, \\
    |m-n|\tau(K) + |m-n| & \text{if } \tau(K) < 0 \text{ and } \epsilon(K)=-1, \\
    |m-n|\tau(K) + \min(m,n) & \text{if } \tau(K) > 0 \text{ and }\epsilon(K) = 1,\\
    |m-n|\tau(K) + \max(m,n) -1& \text{if } \tau(K) \ge 0 \text{ and }\epsilon(K) = -1.
    \end{cases}
    \end{equation}
In the case where $m=n$, we have 
\begin{equation}\label{tau eq 3}
    \tau(Q_{m,m}(K)) = \begin{cases}
    0 & \text{if } \tau(K) < 0, \\
    m-1 & \text{if } \tau(K) = 0, \\
    m & \text{if } \tau(K) > 0.
    \end{cases}
    \end{equation}

\end{theorem}
    
The bordered Heegaard diagrams corresponding to the the generalized Mazur patterns $Q_{m,n}$ is given in Proposition \ref{h diagram}. We label the intersection points of the $\alpha$-arcs and $\beta$-arcs as follows: let the intersection points down the vertical arc be $x_1$ to $x_{1+2m}$, and the points along the horizontal arc be $y_1$ to $y_{2m+2n+2nm-2}$. 
Reading off these diagrams, the $A_\infty$ multiplications have $m$ squares which connect to each other, and $mn + n - 1$ isolated components. The first square contains a head, so one may consider the structure to be as a train with $m$ carriages, with $\rho_2$ arrows connecting carriages. We may see this representation in Figure \ref{CFAgen}. We index the squares by $r = {1,...,m}$ unless otherwise specified.

\begin{proposition}\label{m relations}
The generators $x_1$ to $x_{1+2m}$ lie in $\iota_{0}$, and the remaining generators lie in $\iota_{1}$. 
The multiplications for the isolated components are as follows:

\begin{enumerate}
    \item $y_{(2m+1)l+j} \xrightarrow[]{U} y_{(2m+1)(l+1)-j}$ for $j=1,...,m$, $l=0,1$
    \item $y_{(2m+1)l + (j-1)} \xrightarrow[]{U} y_{(2m+1)(l+1) - (j-1)}$ for $j=1,...,m+1$, $2\le l\le n-1$.
\end{enumerate}

When $n=1$, we only use Equation $1$, and $l=0$. When $n=2$, we use Equations $1$ with both $l=0,1$. When $n \geq 3$, we use both equations.

The multiplications which make up the squares are as follows

\begin{itemize}
    \item $x_r \xrightarrow[]{\rho_{1}} y_{2m+2n+2mn-2-r+1}$ for $r=1,...,m$
    \item $x_r \xrightarrow[]{U} x_{2m-r+1}$ for $r=1,...,m$
    \item $y_{2m+2n+2mn-2-r+1} \xrightarrow[]{U^{m-r+1}} y_{2n+2mn-2+r}$ for $r=1,...,m$
    \item $x_1 \xrightarrow[]{U^{n}\rho_{123}} y_{2n+2mn-1}$
    \item $y_{2m+2n+2mn-2} \xrightarrow[]{U^{n}\rho_{23}} y_{2n+2mn-1}$
    \item $x_{1+2m} \xrightarrow[]{U^{n}\rho_{3}} y_{2n+2mn-1}$
    \item $y_{2m+2n+2mn-2-r+1} \xrightarrow[]{\rho_{2}} x_{2m-r+2}$ for $r=1,...,m$
    \item $x_{r+1} \xrightarrow[]{\rho_{12}} x_{2m-r+1}$ for $r = 0,...,m-1$
    \item $y_{2n+2mn-2+r} \xrightarrow[]{\rho_{2}\widehat{\rho_{12}}^{l-1}\rho_{1}} y_{2n+2mn-2+r-l}$ for $r = 2,...,m$ and $l = 1,...,r-1$
    \item $x_{2m-r+1} \xrightarrow[]{U^{m-r}\widehat{\rho_{12}}^{l}\rho_{1}} y_{2n+2mn-2+r-l}$ for $r = 1,...,m$ and $l = 0,...,r-1$
    \item $x_{2m+1} \xrightarrow[]{U\rho_{3}\rho_{2}\rho_{1}} y_{2m+1}$, when $n\ge 2$
    \item $y_{2m+1} \xrightarrow[]{U^{n-1}} y_{2n+2mn-2}$, when $n\ge 2$
    \item $x_{2m+1} \xrightarrow[]{U^{m}\rho_{1}} y_{2n+2mn-2}$, when $n\ge 2$
\end{itemize}

There are also the following multiplications which appear in both the isolated components and the squares when $n \geq 2$:

\begin{itemize}
    \item $y_j \xrightarrow[]{U\rho_{2}\rho_{1}} y_{2m+1+j}$ for $j=1,...,2mn+n-2m-2$
\end{itemize}
Combinations of these elements also appear in the $A_{\infty}$ structure, as long as the index of any of the terms do not exceed $2mn+n-1$.

One notices that by this, the isolated components are connected to each other, but are isolated from the component containing the main squares.

Let $a = y_{2mn+n-2m-1}$, and $b = y_{2mn+2n-3}$. The $A_{\infty}$ module can be represented by the following figure:
  \begin{figure}[H]
    \begin{tikzcd}[sep=small]
     a \arrow[ddd, swap, "U", dotted, red] & & y_{2m+1} \arrow[ll, swap, "U^{n-2}\rho_{2}\rho_{12}^{n-3}\rho_{1}", dotted, red] \arrow[ddd, swap, "U^{n-1}", dotted] & x_{2m+1} \arrow[l, swap, "U\rho_{3}\rho_{2}\rho_{1}", dotted] \arrow[lddd, swap, "U^{m}\rho_{1}", dotted] \arrow[rddd, swap, "U^{n}\rho_{3}"] 
        & y_{2mn+2m+2n-2} \arrow[l, swap, "\rho_{2}"]  \arrow[ddd, swap, "\substack{U^m + \\ U^{n}\rho_{23}}", near start]
            & x_1 \arrow[lddd,"\mathclap{U^n\rho_{123}}",sloped] \arrow[l, swap,  "\rho_{1}"] \arrow[ll, bend right, swap, "\rho_{12}"] \arrow[ddd, , swap, "U"] & \cdots \arrow[lddd, "\rho_2"] & y_{2mn+m+2n-1} \arrow[lddd, "\rho_2"] \arrow[ddd, "U"] & x_m \arrow[l, "\rho_1", swap] \arrow[ddd, "U"]\\
            & & & & & \\
            & & & & & & & & \\
             b & & y_{2n+2mn-2} \arrow[ll, "\rho_{2}\rho_{1}", dotted, red, swap] & & y_{2mn+2n-1} \arrow[ll, swap, "\rho_2\rho_1", dotted]
            & x_{2m} \arrow[l, swap, "U^{m-1}\rho_{1}"] & \cdots & y_{2mn+m+2n-2} & x_{m+1} \arrow[l, swap, "\rho_1"]
    \end{tikzcd}
    \caption{The $CFA^{-}$ structure for generalized Mazur patterns, with some of the differentials connecting the second row not included for clarity. The dotted lines are included when $n \geq 2$, and the red dotted lines are included when $n \geq 3$. The isolated components are not included here either, although they may be easily determined by Proposition \ref{m relations}. One in fact notices that when $n \geq 3$, the isolated component $a \xrightarrow{U} b$ is no longer isolated. See Section \ref{Examples} for examples of these relations.}
    \label{CFAgen}
    \end{figure}
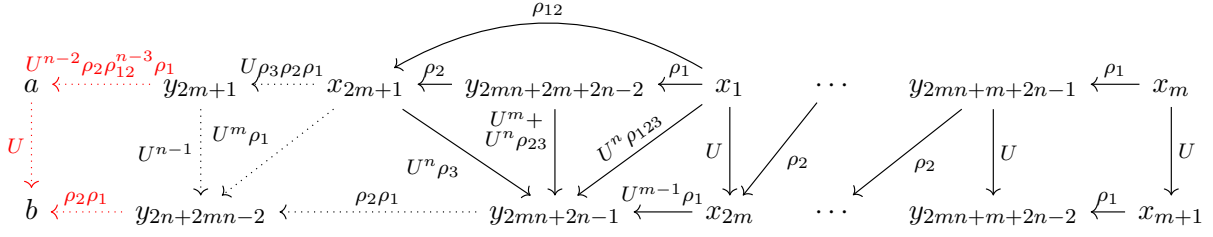
    
\end{proposition}

\begin{remark} \label{rmk change of basis}
There are also the following multiplications: $x_j \xrightarrow[]{\rho_{3}\rho_{2}\rho_{1}} y_j$ for $j=1,...,2m$. These may be removed by the following change of basis: $x_{2m-j+1}^{\prime} = x_{2m-j+1} + \rho_{3}\rho_{2}\rho_{1}y_j$. This isolates $y_j \xrightarrow[]{U} y_{2n-j+1}$ for $j = 1,...,n$. Furthermore, there are also $U\rho_{2}\rho_{1}$ relations between the isolated components, outlined in \ref{m relations}. We may remove many of these by a similar change of basis as the one above, with the appropriate substitutions.
\end{remark}

 From the orientation choice letting $w(Q_{m,n}) = -(m-n)$ and Proposition \ref{key prop}, we have
\[A_{Q_{m,n}(K)}(a \otimes x) = -(m-n)A_{k}(x) + C_{a}.\]
The following lemma allows us to determine the coefficient $C_a$ for certain $a$.

\begin{lemma} \label{lem coef}
    The constants associated to the generators of $CFA^{-}(V,Q_{m,n}) \cdot \iota_{0}$ via Propsition \ref{key prop} are $C_{x_{m+1}} = 0$, $C_{x_{m}} = C_{x_{m+2}} = -1$,..., and $C_{x_{m+1-r}} = C_{x_{m+1+r}} = -r$ for $r = 1,...,m$.
\end{lemma}
\begin{proof}
Let O $\subset S^3$ be the unknot, and $X_{O}$ its complement equipped with the 0-framing. Note that $Q_{m,n}(O)$ is also the unknot. By the pairing theorem, we have
\[gCFK^\infty (Q_{m,n}(O)) = CFA^-(V,Q_{m,n})\boxtimes \widehat{CFD}(X_O).\]
The associated graded on the left hand side has a single generator in its homology, with 0 Alexander grading since $\tau(Q(O))=0$. On the right hand side, the type D structure $\widehat{CFD}(X_{O})$ has a single generator $\xi_{0}$ in 0 grading. The tensor complex $CFA^{-}(V,Q_{m,n}) \boxtimes \widehat{CFD}(X_{O})$ has a summand:
\begin{figure}[H]
\begin{tikzcd}
& x_1\otimes \xi_0 \arrow[ld] \arrow[d, "U"]& \cdots \arrow[ld] \arrow[d, "U"] &  x_m\otimes\xi_0 \arrow[ld] \arrow[d, "U"]
\\
x_{2m+1}\otimes\xi_0 & x_{2m}\otimes\xi_0 & \cdots & x_{m+1}\otimes\xi_0
\end{tikzcd}
\end{figure}
\noindent It is easy to verify that this is indeed a summand. The only relations in $CFA^-(V,Q_{m,n})$ that pair nontrivially with the type D structure are arrows connecting $x_1$, $x_{2m}$ or $x_{2m+1}$ with labels $U^p$ or $U^p\rho_{12}$ (for some integer $p\ge 0$), and all of them are captured within the fundamental complex in Proposition \ref{m relations}. One may observe further that the homology of this summand is generated by $x_{m+1} \otimes \xi_{0}$, which implies that $A(x_{m+1})=0$. The lemma follows. 
\end{proof}

\begin{proof}[Proof of Theorem \ref{tau thm}]Consider the box tensor product of $CFA^{-}(V,Q)$ with $\widehat{CFD}(X_{K})$, where K is a knot in $S^3$, and $X_{K}$ is its exterior equipped with the 0-framing. Note that the generators of the ``isolated'' components do not affect $\tau(Q_{m,n}(K))$ since their tensor products with $\widehat{CFD}(X_{K})$ produce summands of $CFK^{-}(Q_{m,n}(K))$ which are U-torsion.
In the case where $\epsilon(K) = 0$, $\widehat{CFD}(X_{K})$ has a summand isomorphic to $\widehat{CFD}(X_{O})$, so the tensor complex $CFA^{-}(V,Q_{m,n}) \boxtimes \widehat{CFD}(X_{K})$ has a summand isomorphic to (diagram). It follows immediately that $\tau(Q_{m,n}(K)) = \tau(Q_{m,n}(O)) = 0$. Thus, we restrict our discussion to the cases where $\epsilon(K) = \pm 1$. Let $s = 2|\tau(K)|$, and apply the following change of variables for simplicity:
$$w = y_{2n+2mn-2},\ \ y = y_{2m+2n+2mn-2},\ \ y' = y_{2n+2mn-1},\ \ z = y_{2m+2n+2mn-3},\ \  z' = y_{2n+2mn}.$$ 
Moreover, we denote $\lambda_{l_1}^1$ as $\lambda$ and $\kappa_{k_1}^1$ as $\kappa$ in the unstable neighborhood of $CFD(X_K)$.

Recall that the homology of the tensor complex contains a $\mathbb{F}[U]$-free part and a $U$-torsion part. Our goal is to determine the Alexander grading for generator of the free part of the homology. Towards this end, we consider three cases cases according to $\tau(K)$. 

\item \textit{Case 1:} When $\tau > 0$, the unstable chain in $\widehat{CFD}(X_{K})$ (along with the possible $D_1$ differential from $\kappa$ to $\eta_0$ if $\epsilon(K)=-1$, and the $D_2$ differential from $\lambda$ to $\xi_0$ if $\epsilon(K)=1$) is as follows
\begin{figure}[H]
   \begin{tikzcd}
       \kappa \arrow[r, dashed, red, "D_1"] &\eta_0 \arrow[r, "D_3"] &\mu_1 \arrow[r, "D_{23}"] &\cdots \arrow[r, "D_{23}"] &\mu_s &\xi_0 \arrow[l, swap, "D_1"] & \lambda \arrow[l, dashed, swap, blue, "D_2"]
   \end{tikzcd}
 \end{figure}
\noindent The red dashed arrow is included if $\epsilon(K)=-1$ and $n>1$, while the blue dashed arrow is included if $\epsilon(K)=1$ and $m > 1$. The tensor complex has the following direct summand:
\begin{figure}[H]
\begin{tikzcd}
    y_{2m+1}\otimes \kappa \arrow[rd, dashed, red, "U^{n-1}"] &x_{1+2m}\otimes \eta_0 \arrow [d, dashed, red, "U^m"] \arrow[rd, "U^n"] & y\otimes\mu_1 \arrow[d, "U^m"] \arrow[rd, "U^n"]&\cdots \arrow[rd, "U^n"] & y\otimes\mu_{s-1} \arrow[d, "U^m"] \arrow[dr, "U^n"] & y\otimes\mu_s \arrow[d, "U^m"]& x_1\otimes\xi_0 \arrow[l] \arrow[d, "U"]\\
    &w\otimes \kappa & y'\otimes\mu_1 & \cdots & y'\otimes\mu_{s-1} & y'\otimes\mu_s & x_{2m}\otimes \xi_0 \arrow[l, "U^{m-1}"]\\ 
    & & & & & z'\otimes\lambda \arrow[u, blue, dashed] & z\otimes\lambda \arrow[l, dashed, blue, "U^{m-1}"] \arrow[u, blue, dashed]
\end{tikzcd}
\end{figure}
By Proposition \ref{key prop} and Lemma \ref{lem coef}, we have \begin{align*}
A(x_{1+2m}\otimes\eta_0)&= -(m-n)A(\eta_0) - m = (m-n)\tau(K) - m,\\
A(x_1\otimes\xi_0) &= -(m-n)A(\xi_0) - m = -(m-n)\tau(K) - m,\\
A(x_{2m}\otimes\xi_0)&= -(m-n)A(\xi_0) - (m-1) = -(m-n)\tau(K) - (m-1),
\end{align*}
which allows us to compute that
\begin{align*}
    A(y'\otimes\mu_j) &= (m-n)\tau(K) + (n-m)j,\\
    A(y\otimes\mu_j) &= (m-n)\tau(K) + (n-m)j -m,\\
    A(w\otimes\kappa) &= (m-n)\tau(K),\\
    A(y_{2m+1}\otimes\kappa) &= (m-n)\tau(K)-(n-1),\\
    A(z'\otimes\lambda) &= -(m-n)\tau(K),\\
    A(z\otimes\lambda) &= -(m-n)\tau(K)-(m-1).
\end{align*}
We then split our calculation into the following five subcases:
\begin{itemize}
        \item \textit{Case 1.1:} If $\epsilon(K)=1$ and $m\ge n$, the free part of the homology is generated by 
    \[U^nz'\otimes \lambda+\bigg(\sum_{r=1}^{s-1}U^{(m-n)(s-1-r)}y\otimes \mu_r\bigg)+U^{(m-n)(s-1)}x_{1+2m}\otimes \eta_0,\]
    which has Alexander grading $-(m-n)\tau(K)-n$. This gives us
    \[\tau(Q_{m,n}(K))=(m-n)\tau(K)+n=|m-n|\tau(K)+n.\]
    \item \textit{Case 1.2:} If $\epsilon(K)=1$ and $m<n$, the free part is generated by the element \[x_{1+2m}\otimes\eta_0 + \bigg(\sum_{r=1}^{s-1} U^{(n-m)r}y\otimes\mu_{r}\bigg) + U^{(n-m)s+1}x_{2m}\otimes\xi_0,\] 
with Alexander grading $(m-n)\tau(K) - m$. In this case, we have 
\[\tau(Q_{m,n}(K)) = -(m-n)\tau(K) + m = |m-n|\tau(K) + m.\]
    \item \textit{Case 1.3:} If $\epsilon(K)=-1$ and $m>n$, the free part is either generated by
     \[x_{2m}\otimes \xi_0+U^{m-n-1}\bigg(\sum_{r=1}^{s-1}U^{(m-n)(s-1-r)}y\otimes \mu_r\bigg)+U^{(m-n)s-1}x_{1+2m}\otimes \eta_0\]
     if $n=1$, or by the above sum plus an additional term 
$U^{(m-n)(s+1)}y_{1+2m}\otimes \kappa$
     if $n>1$. In both cases, the generator has Alexander grading $-(m-n)\tau(K)-(m-1)$, and 
     \[\tau(Q_{m,n}(K)) = (m-n)\tau(K) + (m-1) = |m-n|\tau(K) + (m-1).\]
    \item \textit{Case 1.4:} If $\epsilon(K)=-1$ and $m<n$, the generator is
     \[y_{2m+1}\otimes \kappa +U^{n-m-1}x_{1+2m}\otimes\eta_0 + U^{n-1-m}\bigg(\sum_{r=1}^{s-1} U^{(n-m)r}y\otimes\mu_{r}\bigg) + U^{(n-m)(s+1)}x_{2m}\otimes\xi_0,\] 
     which has Alexander grading $(m-n)\tau(K)-(n-1)$, giving us 
     \[\tau(Q_{m,n}(K)) = -(m-n)\tau(K) + (n-1) = |m-n|\tau(K) + (n-1).\]
     \item \textit{Case 1.5:} If $\epsilon(K)=-1$ and $m=n$, the free part of the homology is generated by 
     \[Uy_{2m+1}\otimes \kappa+x_{1+2m}\otimes \eta_0+\bigg(\sum_{r=1}^{s-1}y\otimes \mu_r\bigg)+Ux_{2m}\otimes \xi_0,\]
     with Alexander grading $(m-n)\tau(K)-m=-m$, and thus 
     \[\tau(Q_{m,n}(K)) = m.\]
\end{itemize}
Taken together, when $\tau(K)>0$ and $m\ne n$, we have 
\[\tau(Q_{m,n}(K)) = \begin{cases}
    |m-n|\tau(K) + \min(m,n) & \text{if } \epsilon(K)=1, \\
    |m-n|\tau(K) + \max(m,n)-1 & \text{if } \epsilon(K)=-1.
\end{cases}.\]
Moreover, when $\tau(K)>0$ and $m=n$, we have 
\[\tau(Q_{m,n}(K))=m, \ \ \ \text{for all } \epsilon(K).\]

\item \textit{Case 2}: When $\tau < 0$, the unstable chain in $\widehat{CFD}(X_K)$ is the following:
\begin{figure}[H]
   \begin{tikzcd}
       \lambda \arrow[r, dashed, blue, "D_2"] &\xi_0 \arrow[r, "D_{123}"] &\mu_1 \arrow[r, "D_{23}"] &\cdots \arrow[r, "D_{23}"] &\mu_s \arrow[r, "D_2"] &\eta_0 \arrow[r, dashed, red, "D_1"] &\kappa
       \end{tikzcd}
\end{figure}
\noindent where again the red dashed arrow included if $\epsilon(K)=-1$ and $n>1$, and the blue dashed arrow is included if $\epsilon(K)=1$ and $m >1$. The summand with nonvanishing homology in the tensor complex is given by:
\begin{figure}[H]
\begin{tikzcd}[column sep = small]
    z\otimes\lambda \arrow[rd, blue, dashed] \arrow[d,  blue, dashed, "U^{m-1}"]& x_1\otimes\xi_0 \arrow[d, "U"] \arrow[rd, "U^n"] & y\otimes\mu_1 \arrow[d, "U^m"] \arrow[rd, "U^n"] & \cdots \arrow[rd, "U^n"] & y\otimes\mu_{s-1} \arrow[rrr, dashed, bend left, red, "U"] \arrow[d, "U^m"] \arrow[rd, "U^n"] & y\otimes\mu_s \arrow[d, "U^m"] \arrow[r] & x_{1+2m}\otimes \eta_0 \arrow[d, dashed, red, "U^m"] &y_{1+2m}\otimes \kappa \arrow[ld, dashed, red, "U^{n-1}"] \\
     z'\otimes\lambda & x_{2m}\otimes\xi_0 & y'\otimes\mu_1 & \cdots & y'\otimes\mu_{s-1}  & y'\otimes\mu_s \arrow[r, dashed, red]& w\otimes\kappa &
\end{tikzcd}
\end{figure}
 We calculate inductively that 
\begin{align*}
    A(y'\otimes \mu_j)&= -(m-n)\tau(K)+j(n-m),\\
    A(y\otimes \mu_j)&= -(m-n)\tau(K)+j(n-m)-m.
\end{align*}
In the case when $\epsilon(K)=-1$ and $n>1$, we have
\begin{align*}
    A(w\otimes \kappa)&= (m-n)\tau(K),\\
    A(y_{1+2m}\otimes \kappa)&= (m-n)\tau(K)-(n-1),
\end{align*}
and in the case when $\epsilon(K)=1$ and $m>1$, we have
\begin{align*}
    A(z\otimes \lambda)&= -(m-n)\tau(K)-(m-1),\\
    A(z'\otimes \lambda)&= -(m-n)\tau(K).
\end{align*}
Next, we split our calculation into four cases, depending on the value of $\epsilon(K)$ and whether $m\ge n$
\begin{itemize}
    \item \textit{Case 2.1:} If $m\ge n$ and $\epsilon(K)=1$, the free part of the homology is generated by $z'\otimes \lambda$, which has Alexander grading 
    \[A(z'\otimes \lambda)= -(m-n)\tau(K),\]
giving us 
\[\tau(Q_{m,n}(K))=(m-n)\tau(K)=|m-n|\tau(K).\]
    \item \textit{Case 2.2:} If $m< n$ and $\epsilon(K)=1$, the free part is generated by $y'\otimes \mu_s$, with Alexander grading $A(y'\otimes \mu_s)=(m-n)\tau(K)$. As such, we have 
    \[\tau(Q_{m,n}(K))=-(m-n)\tau(K)=|m-n|\tau(K).\]
    \item \textit{Case 2.3:} If $m\ge n$ and $\epsilon(K)=-1$, the free part is generated by $x\otimes \xi_0$, with Alexander grading $A(x\otimes \xi_0)=-(m-n)\tau(K)-(m-1)$. This gives us 
    \[\tau(Q_{m,n}(K))=(m-n)\tau(K)+(m-1)=|m-n|\tau(K)+(m-1).\]
    \item \textit{Case 2.4:} If $m< n$ and $\epsilon(K)=-1$, the free part is generated by $y_{1+2m}\otimes \kappa + U^{n-m-1}x_{1+2m}\otimes \eta_0$, with Alexander grading $A(y_{1+2m}\otimes \kappa)=(m-n)\tau(K)-(n-1)$. Therefore, we have
    \[\tau(Q_{m,n}(K))=-(m-n)\tau(K)+(n-1)=|m-n|\tau(K)+(n-1).\]
\end{itemize}
In summary, when $\tau(K)<0$, we have \[\tau(Q_{m,n}(K)) = \begin{cases}
    |m-n|\tau(K) & \text{if } \epsilon(K)=1, \\
    |m-n|\tau(K) + \max(m,n) - 1 & \text{if } \epsilon(K)=-1.
\end{cases}.\]


\item \textit{Case 3:}
When $\tau = 0$, the unstable neighborhood in $\widehat{CFD}(X_K)$ is 
\begin{figure}[H]
   \begin{tikzcd}
       \lambda \arrow[r, blue, dashed, "D_2"] &\xi_0 \arrow[r, "D_{12}"] &\eta_0\arrow[r, red, dashed, "D_1"] & \kappa,
       \end{tikzcd}
\end{figure}
\noindent and the differential in the tensor complex is
\begin{figure}[H]
\begin{tikzcd}
    z\otimes\lambda \arrow[rd, dashed, blue] \arrow[d, dashed, blue, "U^{m-1}"]& x_1\otimes\xi_0 \arrow[d, "U"] \arrow[r] &x_{1+2m}\otimes\eta_0 \arrow[d, dashed, red, "U^m"] & y_{2m+1}\otimes \kappa \arrow[ld, dashed, red, "U^{n-1}"]\\
     z'\otimes\lambda & x_{2m}\otimes\xi_0\arrow[r, red, dashed, "U^{m-1}"] & w\otimes\kappa & 
\end{tikzcd}
\end{figure}
\noindent where we include red dashed arrow  if $\epsilon(K)=-1$ and $n>1$, and the blue dashed arrow if $\epsilon(K)=1$ and $m >1$. One may check that this is indeed a summand. We divide our computation into the following three cases
\begin{itemize}
    \item \textit{Case 3.1:} If $\epsilon(K)=1$, either $z'\otimes \lambda$ or $x_{2m}\otimes \xi_0$ generates the free part of the homology, depending on whether $m>1$ or $m=1$, respectively. In both cases, the Alexander grading of the generator is 0, and therefore \[\tau(Q_{m,n}(K))=0=|m-n|\tau(K).\]
    \item \textit{Case 3.2:} If $\epsilon(K)=-1$ and $m\ge n$, the free part of the homology is generated by either $x_{2m}\otimes \xi_0 + U^{m-n}y_{2m+1}\otimes \kappa$ if $n>1$, or by $x_{2m}\otimes \xi_0$ if $n=1$. In both cases, the Alexander grading of the generator is $-m+1$. Hence, we get \[\tau(Q_{m,n}(K)) =m-1 = |m-n|\tau(K)+\max(m,n)-1.\]
    \item \textit{Case 3.3:} If $\epsilon(K)=-1$ and $m<n$, the free part is generated by $U^{n-m}x_{2m}\otimes \xi_0 + y_{2m+1}\otimes \kappa$, which has Alexander grading $-n+1$, giving us 
    \[\tau(Q_{m,n}(K)) =n-1 = |m-n|\tau(K)+\max(m,n)-1.\]
\end{itemize}
All in all, when $\tau(K)=0$ we get
\[
\tau(Q_{m,n}(K)) = \begin{cases}
    |m-n|\tau(K) & \text{if } \epsilon(K)=1, \\
    |m-n|\tau(K)+\max(m,n)-1 & \text{if } \epsilon(K)=-1.
\end{cases}
\]
\end{proof}

\bibliographystyle{amsalpha}
\bibliography{bib}
\end{document}